\documentclass[10pt]{amsart}

\usepackage{etex}

\usepackage[usenames,dvipsnames]{xcolor}
\usepackage[bookmarksopen,bookmarksdepth=2]{hyperref}
\usepackage[normalem]{ulem}
\hypersetup{colorlinks=true,citecolor=NavyBlue,linkcolor=NavyBlue,urlcolor=Green} 
\usepackage{amsmath,amsthm,amssymb,amscd,mathtools}
\usepackage{lastpage}
\usepackage{fancyhdr}
\usepackage{mathrsfs}
\usepackage{tikz-cd}
\usepackage{stmaryrd}
\usepackage[backend=biber, style=alphabetic,maxbibnames=99,maxalphanames=99]{biblatex}
\usepackage[scaled=1.15]{urwchancal}
\DeclareFontFamily{OT1}{pzc}{}
\DeclareFontShape{OT1}{pzc}{m}{it}%
{<-> s * [1.15] pzcmi7t}{}
\DeclareMathAlphabet{\mathpzc}{OT1}{pzc}{m}{it}
\newcommand{\mbb}[1]{\mathbb{#1}}
\newcommand{\mcal}[1]{\mathcal{#1}}

\newcommand{\NN}{\mathbb{N}}

\newcommand{\ZZ}{\mathbb{Z}}
\newcommand{\QQ}{\mathbb{Q}}
\newcommand{\CC}{\mathbb{C}}

\newcommand{\FF}{\mathbb{F}}
\newcommand{\GG}{\mathbb{G}}

\newcommand{\Gal}{\mathrm{Gal}}

\newcommand{\Ind}{\mathrm{Ind}}
\DeclareMathOperator{\im}{Im}

\DeclareMathOperator{\Hom}{Hom}
\DeclareMathOperator{\Ext}{Ext}

\DeclareMathOperator{\Spec}{Spec}

\DeclareMathOperator{\Lie}{Lie}
\DeclareMathOperator{\Mod}{Mod}

\newcommand{\GL}{\mathrm{GL}}

\DeclareMathOperator{\Tr}{\mathrm{Tr}}

\DeclareMathOperator{\Fil}{Fil}

\newcommand{\p}{\mathfrak{p}}

\newcommand{\mO}{\mathcal{O}}
\newcommand{\mP}{\mathbb{P}}
\newcommand{\mA}{\mathbb{A}}

\newcommand{\mm}{\mathfrak{m}}

\DeclareMathOperator{\End}{End}

\newcommand{\id}{\mathrm{id}}

\newcommand{\mF}{\mathcal{F}}

\DeclareMathOperator{\Frob}{Frob}

\DeclareMathOperator{\Ker}{Ker}
\DeclareMathOperator{\Coker}{Coker}

\newtheorem{thm}{Theorem}[section]
\newtheorem{prop}[thm]{Proposition}

\newtheorem{lem}[thm]{Lemma}
\newtheorem{cor}[thm]{Corollary}

\theoremstyle{remark}
\newtheorem{rmk}[thm]{Remark}
\newtheorem{eg}[thm]{Example}
\theoremstyle{definition}
\newtheorem{dfn}[thm]{Definition}
\newtheorem{notation}[thm]{Notation}
\newtheorem{construction}[thm]{Construction}

\theoremstyle{plain}

\theoremstyle{remark}

\theoremstyle{definition}

\newcommand{\mX}{\mathcal{X}}

\newcommand{\et}{\mathrm{\acute{e}t}}

\newcommand{\proet}{\mathrm{pro\acute{e}t}}
\newcommand{\proket}{\mathrm{prok\acute{e}t}}

\newcommand{\dR}{\mathrm{dR}}

\newcommand{\OBdR}{\mO\mbb{B}_\dR}

\newcommand{\BdR}{\mbb{B}_\dR}

\DeclareMathOperator{\gr}{gr}
\newcommand{\HT}{\mathrm{HT}}

\newcommand{\Fl}{{\mathscr{F}\!l}}

\DeclareMathOperator{\Spa}{Spa}
\newcommand{\TT}{\mbb{T}}
\newcommand{\VB}{V\!B}

\newcommand{\fn}{\mathfrak{n}}
\newcommand{\fb}{\mathfrak{b}}
\newcommand{\fg}{\mathfrak{g}}
\newcommand{\sm}{{\mathrm{sm}}}
\newcommand{\la}{{\mathrm{la}}}
\newcommand{\an}{{\mathrm{an}}}

\newcommand{\Kpp}{K^pK_p}

\newcommand{\fh}{\mathfrak{h}}

\newcommand{\mrm}[1]{\mathrm{#1}}
\newcommand{\mfk}[1]{\mathfrak{#1}}

\newcommand{\RHom}{R\!\Hom}

\newcommand{\mscr}[1]{\mathscr{#1}}

\newcommand{\Sym}{\mrm{Sym}}

\newcommand{\OBdRR}[1]{\OBdR^{+,\la}/\Fil^{#1}}

\newcommand{\mE}{\mcal{E}}
\newcommand{\pmD}{{}^{p}\mcal{D}}
\newcommand{\mD}{\mcal{D}}
\newcommand{\Perv}{\mrm{Perv}}

\newcommand{\Fib}{\mrm{Fib}}

\newcommand{\QCoh}{\mrm{QCoh}}

\newcommand{\VBn}{\VB^{\text{naïve}}}
\newcommand{\rla}{\mrm{rla}}

\DeclareMathOperator{\Rep}{Rep}

\usepackage{mathtools}

\newcommand{\ti}[1]{\tilde{#1}}

\newcommand{\OC}{\mO\mbb{C}}

\newcommand{\Darith}{D_{\mrm{arith}}}
\newcommand{\hatDarith}{\hat{D}_{\mrm{arith}}}

\newcommand{\hatotimes}{\mathbin{\hat{\otimes}}}

\begin{document}
\title{Theta Operator Equals Fontaine Operator on Modular Curves}
\author{Yuanyang Jiang}
\address{Institut Mathématiques d'Orsay, Université Paris-Saclay, 307 Rue Michel Magat Bâtiment 307, 91400 Orsay, France}
\email{yuanyang.jiang@universite-paris-saclay.fr}
\begin{abstract}
Inspired by \cite{Pan2209.06II},
we give a new proof that for an overconvergent modular eigenform $f$ of weight $1+k$ with $k\in\ZZ_{\ge1}$, assuming that its associated Galois representation $\rho_{f}:\Gal_{\QQ}\to \GL_{2}(\bar{\QQ}_{p})$ is irreducible, then $f$ is classical if and only if the  associated Galois representation $\rho_{f}$  is de Rham at $p$. For the proof,
we prove that theta operator $\theta^{k}$ coincides with Fontaine operator in a suitable sense. 
\end{abstract}
\maketitle

\tableofcontents
\section{Introduction}
We fix a prime number $p$. Let \(G:=\GL_{2}\). For any neat open compact subgroup $K^{p}=\prod_{l\nmid p}K_{l}\subset G(\mA_{f}^{p})$, and $K_{p}\subset G(\QQ_{p})$, let $X_{\Kpp}$ be the compactified modular curve, which is a scheme over $\QQ$, and $\mX_{\Kpp}$ be the analytification of $X_{\Kpp}\times_{\QQ}\CC_{p}$. We will fix $K^{p}$ from now on. Let $S$ be a finite set of places of $\QQ$ including $p$ and $\infty$, such that for any $l\notin S$, $K_{l}=\GL_{2}(\ZZ_{l})$, and write $K^{S}:=\prod_{l\notin S}K_{l}$. We define the Hecke algebra         \( \TT(K^{p}):=C_{c}^{\infty}(K^{p}\backslash G(\mA_{f}^{p})/K^{p},\ZZ), \)
and $\TT^{S}:=C_{c}^{\infty}(K^{S}\backslash G(\mA^{S})/K^{S},\ZZ)$. Then $\TT^{S}$ is a commutative ring freely generated by $T_{l}$ and $S_{l}^{\pm1}$ for $l\notin S$. Let $k$ be a positive integer. 

Let $\rho:\Gal_{\QQ}\to \GL_{2}(\bar{\QQ}_{p})$ be an irreducible Galois representation. If it is unramified away from finitely many places, odd and de Rham at $p$ of weight $0,k$, then Fontaine-Mazur conjecture predicts that it is associated to a modular eigenform $f\in H^{0}(\mX_{\Kpp},\omega^{k})$. 
This conjecture has been proven under certain generic conditions in \cite{Kisin2009fontaine}, \cite{Emerton2011local-global}, and is reproved in \cite{Pan2209.06II} using different method.
 
The result of this note is the following classicality theorem, which proves the conjecture for ``overconvergent" Galois representations:
\begin{thm}[Corollary \ref{corClassicality}]\label{thmMainThmClassicality}
Let $f\in M_{1+k}^{\dagger}(K^{p}):=\varinjlim_{K_{p}}M_{1+k}^{\dagger}(\Kpp)$ be an overconvergent modular $\TT^{S}$-eigenform  of weight $1+k$ with $k\in\ZZ_{\ge 1}$. Assume that its associated Galois representation $\rho_{f}:\Gal_{\QQ}\to \GL_{2}(\bar{\QQ}_{p})$ is irreducible. Then
$f$ is a classical modular form if and only if $\rho_{f}$ is de Rham at $p$. 
\end{thm}
\begin{rmk}
By ``eigenform", we refer to the eigenform 
with respect to the action of $\TT^{S}$, that is, with respect to the action of $T_{l}$ and $S_{l}$ for $l\notin S$. Note that the result is false if we only require $f$ to be a generalized eigenform. We say $f$ is a classical form if $f$ lies in the image of $M_{1+k}(K^{p}):=\varinjlim_{K_{p}}M_{1+k}(\Kpp)\hookrightarrow M^{\dagger}_{1+k}(K^{p})$.
\end{rmk}
\begin{rmk}
For those $f$ of finite slope, this result is proven in \cite{Kisin2003overconvergent}. In general,
this can be shown by combining Theorem 1.0.1 of \cite{Pan22} and Theorem 1.1.2 of \cite{Pan2209.06II}. Our note intends to give a different and simple proof.
\end{rmk}

Our proof is inspired by that of \cite{Pan2209.06II}.
Let us explain the ingredient of our proof. Given any overconvergent eigenform $f\in M_{1+k}^{\dagger}(K^{p})$, we denote its corresponding ideal in $\TT^{S}\otimes L$ as $\p_{f}$, such that the action of $\TT^{S}$ on $f$ factors as $\TT^{S}\to (\TT^{S}\otimes L)/\p_{f}\cong L$, where $L$ is a finite extension of $\QQ_{p}$. Then $f$ has the associated Galois representation $\rho_{f}:\Gal_{\QQ}\to \GL_{2}(L)$ characterized by the Eichler-Shimura relation. By assumption, $\rho_{f}$ is absolutely irreducible.

As a first step, we realize $\rho_{f}$ in Emerton's completed cohomology using the results of \cite{Pan22} as follows. 
In \cite{Emerton06}, Emerton introduces the completed cohomology         \[
        R\Gamma(K^{p},\QQ_{p}):=(R\varprojlim_{n}\varinjlim_{K_{p}}R\Gamma(X_{\Kpp}(\CC)^{\an},\ZZ/p^{n}))[1/p],
        \] and \( \tilde{H}^{i}(K^{p},\QQ_{p}):=H^{i}(R\Gamma(K^{p},\QQ_{p})), \)
        which carries the action of     \( \Gal_{\QQ}\times G(\QQ_{p})\times\TT(K^{p}) \).
         The actions of $\Gal_{\QQ}$ and of $\TT(K^{p})$ are related by the Eichler-Shimura relation.

% Note  that
% $\tilde{H}^{i}(K^{p},\QQ_{p})$ is an admissible Banach representation of $G(\QQ_{p})$, so by \cite{ST03algebras}, the subspace of $G(\QQ_{p})$-locally analytic vectors $\tilde{H}^{i}(K^{p},\QQ_{p})^{\la}$ is dense in $\tilde{H}^{i}(K^{p},\QQ_{p})$. 
Let $\tilde{H}^{1}(K^{p},\QQ_{p})^{\la}$ be the subspace of $G(\QQ_{p})$-locally analytic vectors, on which $\fg$ acts by taking derivative. 
% To make the statement uniform for the case $k=1$ and $k>1$, we will use the derived locally analytic vectors (see Definition \ref{dfnRla} for details):                 \[
%                 R\Gamma(K^{p},\QQ_{p})^{R-\la}:=\varinjlim_{G_{n}\subset G}R\Gamma\left(G_{n}(\QQ_{p}),R\Gamma(K^{p},\QQ_{p})\hatotimes\mO_{G_{n}}\right) ,
%                 \] where $G_{n}$ form a basis of open affinoid $G_{n}$ subgroups of $G$, and $\mO_{G_{n}}$ is its ring of analytic functions, and $G_{n}(\QQ_{p})$ acting diagnally with its action on $\mO_{G_{n}}$ induced by left mulplication. Let $\fg:=\Lie G(\QQ_{p})$, then $\fg$ acts on $R\Gamma(K^{p},\QQ_{p})^{R-\la}$ induced by its action on $\mO_{G_{n}}$ by right multiplication. Let $\fb\subset \fg$ (resp. $B\subset G$) be the upper-triangular Borel sub-Lie algebra (resp. subgroup). 
The main result of \cite{Pan22} describes the $\fb$-isotypic component of $\tilde{H}^{1}(K^{p},\QQ_{p})^{\la}$. Note that we have localized at the ideal $\p_{f}\subset \TT^{S}$ to kill the contribution from $\tilde{H}^{0}$.
% Readers are refered to Subsection \ref{subsectionNotations} for the notations.
\begin{thm}[Theorem \ref{PanPilloni}]\label{thmPanPilloniCohomologyIntro}
For $k\in\ZZ_{\ge1}$, we denote $$(\tilde{\rho}_{k})_{\p_{f}}:=\Hom_{\fb}((k-1,0),\tilde{H}^{1}(K^{p},\QQ_{p})^{\la}_{\p_{f}}).$$ Then
 we have a $B(\QQ_{p})\times\Gal_{\QQ_{p}}\times \TT(K^{p})$-equivariant isomorphism  \begin{align*}(\tilde{\rho}_{k})_{\p_{f}}\hatotimes_{\QQ_{p}}\CC_{p}\cong (N_{0})_{\p_{f}}\oplus (N_{k})_{\p_{f}}(-k),
\end{align*} where        \(
        (N_{k})_{\p_{f}}\cong (M_{1+k}^{\dagger})_{\p_{f}},
        \) and  $(N_{0})_{\p_{f}}$ lies in an exact sequence         \[0\to \varinjlim_{K_{p}}H^{1}(\mX_{\Kpp},\omega^{1-k})_{\p_{f}}\to (N_{0})_{\p_{f}}\to (M_{1-k}^{\dagger})_{\p_{f}}\to 0 .
                \]
        Here $M_{\bullet}^{\dagger}:=\varinjlim_{K_{p}}M_{\bullet}^{\dagger}(\Kpp)$, and  $(-k)$ refers to the Tate twist.
\end{thm}
We know that $(N_{k})_{\p_{f}}[\p_{f}]\ne 0$ and thus         \( \tilde{\rho}_{k,L}[\p_{f}]:=(\tilde{\rho}_{k})_{\p_{f}}[\p_{f}]\otimes_{\QQ_{p}}L\ne 0 \).
         By the Eichler-Shimura relation and \cite{BLR91quotients}, we know that $\tilde{\rho}_{k,L}[\p_{f}]\cong\rho_{f}\otimes_{L}W$, where $W$ is a topological $\TT(K^{p})\times B(\QQ_{p})$-module.

So far, we have realized $\rho_{f}$ in the completed cohomology. 
The isomorphism above shows that $\rho_{f}$ is Hodge-Tate of weight $0,k$, where by our convention, the cyclotomic character has Hodge-Tate weight $-1$. 
Let us denote by \(\Theta\) the Sen operator acting on \(\rho_{f}\otimes\CC_{p}\) (\cite{Sen1973lie}), then \[\rho_{f}\otimes\CC_{p}\cong (\rho_{f}\otimes\CC_{p})^{\Theta=0}\oplus (\rho_{f}\otimes\CC_{p})^{\Theta=-k}.
\]

For such $\rho_{f}$, its de Rhamness is characterized by its Fontaine operator (\cite{Fontaine2004arithmetique}), which is a morphism         \[
        N^{k}: (\rho_{f}\otimes\CC_{p})^{\Theta=0}\to (\rho_{f}\otimes \CC_{p})^{\Theta=-k}(k) .
        \]   By \cite{Fontaine2004arithmetique}, $\rho_{f}$ is de Rham if and only if $N^{k}=0$. The same argument also applies 
        to the infinite dimensional representation $(\tilde{\rho}_{k})_{\p_{f}}$, and we are led to study the Fontaine operator of $(\tilde{\rho}_{k})_{\p_{f}}$. The following result describes the Fontaine operator in terms of the classical theta operator:
\begin{thm}[Theorem \ref{mainthmFontaine=ThetaCohomologyVer}]\label{thmFontaine=ThetaIntroCohomologyVer}
In terms of the isomorphism in Theorem \ref{thmPanPilloniCohomologyIntro},
 the Fontaine operator $N^{k}:((\tilde{\rho}_{k})_{\p_{f}}\hatotimes_{\QQ_{p}}\CC_{p})^{\Theta=0}\to ((\tilde{\rho}_{k})_{\p_{f}}\hatotimes_{\QQ_{p}}\CC_{p})^{\Theta=-k}(k) $ is given by         \[
         N^{k}:(N_{0})_{\p_{f}}\to (M^{\dagger}_{1-k})_{\p_{f}}\xrightarrow{\theta^{k}}(M^{\dagger}_{1+k})_{\p_{f}}\cong (N_{k})_{\p_{f}} ,
         \] where $\theta^{k}$ is the theta operator as in \cite[Section 4]{Coleman1996classical}.
 \end{thm} 
Given Theorem \ref{thmFontaine=ThetaIntroCohomologyVer}, 
we can describe the Fontaine operator for $\tilde{\rho}_{k,L}[\p_{f}]$ as \[N_{0,L}[\p_{f}]\to M^{\dagger}_{1-k,L}[\p_{f}]\xrightarrow{\theta^{k}}  M^{\dagger}_{1+k,L}[\p_{f}],\] with $(-)_{L}$ standing $-\hatotimes_{\QQ_{p}}L$.
By the $q$-expansion principle,  $ M^{\dagger}_{1-k,L}[\p_{f}]\xrightarrow{\theta^{k}}  M^{\dagger}_{1+k,L}[\p_{f}]$ is injective, so the kernel of $N^{k}$ is given by $ H^{1}(\Fl,\omega^{1-k,\sm})_{L}[\p_{f}]$ by Theorem \ref{thmPanPilloniCohomologyIntro}. Hence $\rho_{f}$ is de Rham if and only if $H^{1}(\Fl,\omega^{1-k,\sm})_{L}[\p_{f}]\ne 0$. This implies the classicality of $\p_{f}$.

 Now we sketch a proof of Theorem \ref{thmFontaine=ThetaIntroCohomologyVer}. The idea is to realize the Fontaine operator geometrically as in \cite{Pan2209.06II}. In
  \cite{Scholze15}, Scholze  introduces the perfectoid modular curve at the infinite level        \(
                \mX_{K^{p}}\sim \varprojlim_{K_{p}}\mX_{\Kpp} ,
                \) which carries an "affinoid" Hodge-Tate map $\pi_{\HT}:\mX_{K^{p}}\to \Fl\cong\mP^{1}$. Moreover, Scholze shows that $\mX_{K^{p}}$ is related to Emerton's completed cohomology  as      \[
                        R\Gamma(K^{p},\QQ_{p})\hatotimes\CC_{p}\cong R\Gamma(\mX_{K^{p},\an},\mO_{\mX_{K^{p}}})\cong R\Gamma(\Fl_{\an},\hat{\mO}) ,
                        \] where $\hat{\mO}:=\pi_{\HT,*}\mO_{\mX_{K^{p}}}$. 
 We can further consider the subsheaf $\mO^{\la}\subset \hat{\mO}$ consisting of $\GL_{2}(\QQ_{p})$-locally analytic vectors. By Theorem 4.4.6 of \cite{Pan22}, we have \begin{equation}
 \label{equationla=Ola}
         \tilde{H}^{i}(K^{p},\QQ_{p})^{\la}\hatotimes\CC_{p}\cong H^{i}(\Fl,\mO^{\la}) ,
 \end{equation} and thus         \[
  \RHom_{\fb}((k-1,0),R\Gamma(K^{p},\QQ_{p})^{\la})\hatotimes\CC_{p}\cong R\Gamma(\Fl,\RHom_{\fb}((k-1,0),\mO^{\la})    )     .
         \] 
\begin{dfn}\label{dfnOmegaLSM}
For a neat open compact subgroup \(K\subset G(\QQ_{p})\), we denote by \(\pi_{K}\) the natural morphism \(\mX_{K^{p}}\to \mX_{K}\), and denote by \(\pi_{K}^{*}(\omega^{\ell}_{\mX_{K}})^{\sm}\) the subsheaf of \(\pi_{K}^{*}(\omega^{\ell}_{\mX_{K}})\) consisting of \(G(\QQ_{p})\)-smooth vectors. We further denote
 by \(\omega^{\ell,\sm}:=\pi_{\HT,*}(\pi_{K}^{*}(\omega^{\ell}_{\mX_{K}})^{\sm})\in D(\Fl_{\an})\). 
 
 Then \(H^{0}(\Fl,\omega^{\ell,\sm})\) (resp. $H^{0}(\Fl,i_{*}i^{-1}\omega^{\ell,\sm})$)  is the space of modular forms (resp. overconvergent modular forms) of weight $\ell$ (with tame level $K^{p}$ and arbitrary level at $p$).
\end{dfn}

In \cite{Pilloni22}, Pilloni computes explicitly $\RHom_{\fb}((k-1,0),\mO^{\la})$ in terms of $\omega^{\ell,\sm}$ as follows:
\begin{thm}[Theorem \ref{mainthmPanPilloni}]\label{thmPilloniIntro}
For $k\in\ZZ_{\ge 1}$, we have a $B(\QQ_{p})\times \Gal_{\QQ_{p}}\times \TT(K^{p})$-equivariant isomorphism in $D(\Fl_{\an})$        \[
        \RHom_{\fb}((k-1,0),\mO^{\la})\otimes\chi^{(1-k,0)}\cong \mcal{N}_{0}\oplus \mcal{N}_{k}(-k) ,
        \] where $\mcal{N}_{k}\cong i_{*}i^{-1}\omega^{1+k,\sm}[-1]$, and $\mcal{N}_{0}$ lies in a distinguished triangle         \[
                \omega^{1-k,\sm}\to \mcal{N}_{0}\to i_{*}i^{-1}\omega^{1-k,\sm}[-1]\xrightarrow{+1},
                \]
 where $\chi^{(1-k,0),\sm}$ denotes a twist of $B(\QQ_{p})$-action,  
 $i$ is the embedding of $\infty$ into $\Fl$, and $i^{-1}$ is taking the stalk at $\infty$. Moreover,
  taking $R\Gamma(\Fl,-)_{\p_{f}}$, we get back the isomorphism in Theorem \ref{thmPanPilloniCohomologyIntro}.
\end{thm}
% \begin{rmk}
% For any $\ell\in\ZZ_{p}$, $\omega^{\ell,\sm}$ is a $\TT(K^{p})\times B(\QQ_{p})$-equivariant sheaf on $\Fl$, such that $H^{0}(\Fl,i_{*}i^{-1}\omega^{\ell,\sm})$  is the space of overconvergent modular forms of weight $\ell$ (with tame level $K^{p}$ and arbitrary level at $p$). When $\ell\in \ZZ$, an overconvergent form is classical if and only if it can be extended to a section in $H^{0}(\Fl,\omega^{\ell,\sm})$. See  Subsection \ref{subsectionAutShv} for details.
% \end{rmk}
\begin{rmk}\label{rmkPerversity}
We will prove that         \( \mcal{N}_{0} \) and     \( \mcal{N}_{k} \) are perverse sheaves on $\Fl$ (in the sense of Definition \ref{dfnPerverseShv}). We expect this to be true for general Shimura varieties, which hopefully will allow us to construct a finer version of Fontaine operators. 
 \end{rmk} 
 Following the construction of the Fontaine operator, it turns out that we can define a "geometric Fontaine operator" between perverse sheaves
        \[
        N^{k}:\mcal{N}_{0}\to \mcal{N}_{k} ,
         \] which when taking $R\Gamma(\Fl,-)_{\p_{f}}$ gives back the classical Fontaine operator in
         Theorem \ref{thmFontaine=ThetaIntroCohomologyVer}.

         We describe this geometric Fontaine operator $N^{k}$ in this note:

\begin{thm}[Theorem \ref{thmFontaine=Theta}]
\label{thmFontaine=ThetaIntro}
In terms of the isomorphism in Theorem \ref{thmPilloniIntro},
 the geometric Fontaine operator $N^{k}:\mcal{N}_{0}\to \mcal{N}_{k} $ is given by         \[
         N^{k}:\mcal{N}_{0}\to i_{*}i^{-1}\omega^{1-k,\sm}[-1]\xrightarrow{\theta^{k}}i_{*}i^{-1}\omega^{1+k,\sm}[-1]\cong \mcal{N}_{k}  ,
         \] where $\theta^{k}$ is the classical theta operator as in \cite{Coleman1996classical}.
\end{thm}
Theorem \ref{thmFontaine=ThetaIntroCohomologyVer} follows from Theorem \ref{thmFontaine=ThetaIntro}
by primitive comparison (Theorem \ref{thmPrimitiveComparison}), so it reduces to proving Theorem \ref{thmFontaine=ThetaIntro}. 
The idea of proving Theorem \ref{thmFontaine=ThetaIntro} is to observe that $N^{k}$ is some kind of differential operator, and is also $B(\QQ_{p})$-equivariant. The two conditions actually uniquely pin down the morphism, which follows from the simple observation that for $\ell\ne \ell'$, there is no $B(\QQ_{p})$-equivariant $\mO^{\sm}$-linear morphism from $i_{*}i^{-1}\omega^{\ell,\sm}\to i_{*}i^{-1}\omega^{\ell',\sm}$.

\begin{rmk}[Comparison with \cite{Pan2209.06II}]
Let us stress that our proof is very much inspired by that of \cite{Pan2209.06II}. The comparison with Pan's proof will be given in Remark \ref{rmkComparisonPan}. The main difference is that one obtain more symmetries and structures
after taking $\fb$-cohomology, which is an interesting phenomenon by itself, and is generalizable to general Shimura varieties.
\end{rmk}

\subsection{Organization}\label{subsectionOrganization} 
We sketch the structure of the paper. Readers should refer to the beginning of each section for details.
In Section \ref{sectionGeomSen}, we will fix the set-up and recall the results of geometric Sen theory from \cite{Pan22}, \cite{Pilloni22}, \cite{Juan2022.05GeoSen}, \cite{Juan2022.09locallyShi}. 
Section \ref{secGeoFontaine} will compute the \(\fb\)-cohomology (Theorem \ref{thmPilloniIntro}), define the \emph{geometric Fontaine operator} (Corollary \ref{corGeomFontaine}), and state the main theorem (Theorem \ref{thmFontaine=ThetaIntro}). The rest of the section will involve some reduction of the proof. 
Section \ref{sectionPfofMainThmdeRhamSheaves} will finish the proof of Theorem \ref{thmFontaine=ThetaIntro}.
Finally, in Section \ref{sectionArithCor}, we finish the proof of Theorem \ref{thmMainThmClassicality}.

% In Subsection \ref{subsectionPerverse}, we will define perverse t-structure.
% After setting up the notations, we will provide the precise form of our main theorem (Theorem \ref{thmFontaine=Theta}) in Subsection \ref{subsectionStatement}. 

% In Section \ref{sectionPfofMainThmdeRhamSheaves}, we will finish the proof of  Theorem \ref{thmFontaine=ThetaIntro}. We will start by explaining the strategy of the proof in Subsection \ref{subsectionStrategy}. Theorem \ref{thmFontaine=Theta} will be proven in Subsection \ref{subsectionGeneral(a,b)a-bge0}.

% In Subsection \ref{subsectionGeneral(a,b)a-ble0}, we also compute the Fontaine operator for $k<0$, where the Cousin complexes as in \cite{BoxerPilloni2021higherColeman} appears (see Part (2) of Theorem \ref{thmFontaine=Theta} for the statement). The cases for $k>0$ and $k<0$ are related by BGG complex. Along the way, we also give a new proof of the description of locally algebraic vectors in the completed cohomology as in \cite{Emerton06}.

% Later, we will use the theory to prove some technical lemmas that are needed in the proof of Theorem \ref{thmFontaine=Theta}.

\subsection{Notations and conventions}\label{subsecNotationsAndConventions}
  Let us fix some notation. We fix a prime $p$, and write $\CC_{p}$ for the completion of $\bar{\QQ}_{p}$. We fix a compatible system of \(p^{n}\)-th roots of unity. We denote by \(\QQ_{p}(\zeta_{p^{\infty}})\) the algebraic extension of \(\QQ_{p}\) by \(\zeta_{p^{n}}\) for all \(n\), and by \(\QQ_{p,\infty}\) its \(p\)-adic completion. 
  
  Let $G:=\GL_{2}$. Let $B\subset G$ be the Borel subgroup of upper-triangular matrices, and denote its Levi decomposition as $B=TU$,  
where $T$ is the subgroup of diagonal matrices. Let $\fg:=\mfk{gl}_{2}(\QQ_{p})$, and $\fb:=\Lie B$. For $(a,b)\in\CC_{p}^{\oplus 2}$, denote by $(a,b)$ the character of $\fb$,       \((a,b): \begin{pmatrix}
x & y\\0 & z
\end{pmatrix}\mapsto ax+bz. \) For $(a,b)\in\ZZ^{\oplus2}$, we denote by $\chi^{(a,b)}$ the character of $B(\QQ_{p})$ mapping $\chi^{(a,b)}:\begin{pmatrix}
x& y \\0&z
\end{pmatrix}\mapsto x^{a}z^{b}$, and we will write $-\otimes\chi^{(a,b)}$ to mean twisting $B(\QQ_{p})$-action by $\chi^{(a,b)}$. 

For any abelian category $\mcal{A}$ (or some stable $\infty$-category) and for any $X, Y\in\mcal{A}$, we will write $[X-Y]$ to refer to some extension of $Y$ by $X$ in $\mcal{A}$. More precisely, if we write $Z\cong [X-Y]$, this means that there exists a short exact sequence (or a distinguished triangle)       \[
         X\to Z\to Y\xrightarrow{+1}
        \] in $\mcal{A}$. In this note, we will use this notation when $\mcal{A}=\mrm{Perv}$, that is, the category of perverse sheaves (as in Definition \ref{dfnPerverseShv}).

We will use both perverse t-structure and 
natural t-structure of derived category $D(\Fl_{\an})$ of sheaves on the flag variety. For any $F\in D(\Fl_{\an})$, we denote by $H^{i}(F)$ its cohomology with respect to the natural t-structure, which is a sheaf on $\Fl_{\an}$.

We will write $(-)$ to mean Tate twist of Galois action.
By convention, the cyclotomic character $\chi_{\mrm{cycl}}:\Gal_{\QQ_{p}}\to \QQ_{p}^{\times}$ is defined to have Hodge-Tate weight $-1$.

In order to simplify some cohomological argument,
we will use the solid formalism of \cite{CS19} for dealing with topological vector spaces. In particular, sheaves have values in \(\QQ_{p}\)-solid spaces without assuming otherwise. Let \(\mF\) be such a sheaf on a site \(\tilde{X}\). Then we denote by \(H^{0}(\tilde{X},\mF)\) the space of its global sections, and if \(S\) is a profinite set, we denote by \(\mF(S)\) the sheaf on \(\ti{X}\) valued in (non-condensed) vector spaces, sending \(U\to H^{0}(U,\mF)[S]\).

For complete Huber pairs \((R,R^{+})\) with a pseudo-uniformizer \(\varpi\), 
\(R\) and \(R^{+}\) are regarded as condensed rings using \(\varpi\)-adic topology, and \(R\) is regarded as an analytic ring as  \(R_{\square}:=(R,R^{+})_{\square}\).
Throughout the paper, \(\ZZ_{p}\), \(\QQ_{p}\) and \(\CC_{p}\) are regarded as analytic rings via the analytic structure induced from \(\ZZ_{p,\square}:=(\ZZ_{p},\ZZ_{p})_{\square}\).

We will work with objects in the derived category by default.
We will define $-\hatotimes_{\ZZ_{p}}-$ to be the \emph{derived} solid tensor product $-\otimes_{\ZZ_{p},\square}-$. We use the same convention for \(-\hatotimes_{\CC_{p}}-\) or \(-\hatotimes_{\QQ_{p}}-\). In the paper,
we usually work with Banach spaces or LB spaces. In this situation, thanks to \cite[Lemma 3.13]{JRC2021solid}, solid tensor products coincide with classical completed tensor products of Banach spaces, so there is no conflict of notations.

We will mainly work with sheaves on the analytic site. For \(f:X\to Y\), we write \(f^{-1}\) for the pull-back functor. We will reserve \(f^{*}\) for the pull-back functor of quasi-coherent sheaves. We will also see \(i^{\dagger}\) in Definition \ref{dfnidagger}, which is the pull-back functor of quasi-coherent sheaves to the dagger neighborhood, whose underlying functor of analytic sheaves coincides with \(i^{-1}\).

Without specifying otherwise, $(-)^{\la}$ (resp. $(-)^{R-\la}$) is always taking (resp. derived) locally analytic vectors with respect to $G(\QQ_{p})$-actions. The  definition of the latter is given in Definition \ref{dfnRla}. 

In what follows, all the isomorphisms and the identifications are unique up to a sign, and we ignore them systematically.

\subsection{Acknowledgment}
I would like to thank my advisor Vincent Pilloni for introducing me to the subject, for his constant encouragement and support, and for numerous fruitful discussions that make this work possible.
I also benefit greatly from conversations with  Lue Pan and Juan Esteban Rodr\'iguez Camargo, and I want to thank them for sharing their beautiful ideas. I would like to thank Longke Tang for suggesting the stacky approach that simplifies the proof of Proposition \ref{propf*equalsE0}.
I want to thank Vincent Pilloni, Lue Pan, Arthur-César Le Bras and Juan Esteban Rodr\'iguez Camargo for their comments and corrections on the earlier drafts of this work. 
I wish to express special thanks to the anonymous referee for the careful proofreading of this paper which improved greatly the presentation.
I also want to thank George Boxer, Valentin Hernandez, Andrew Graham, Arthur-César Le Bras, Zhouhang Mao, Junhui Qin, Liang Xiao, Qixiang Wang,  Zhixiang Wu and Daming Zhou for helpful exchanges.
Part of this work was done during my stay at BICMR, and I would like to thank Liang Xiao and Jun Yu for their hospitality. This work was written when the author was a PhD student at the Université Paris-Saclay, supported by CDSN (contrat doctoral spécifique normalien).

\section{Prerequisites}
\label{sectionGeomSen}

This section is about the set-up and prerequisites around the geometric Sen theory. 
The structure of the section 
is as follows. Subsection \ref{subsectionNotations} will set up the notation, introduce the infinite level modular curve, and Hodge-Tate period maps as in \cite{Scholze15}. Subsection \ref{subsecLocAnalyrticVect} will introduce the derived functor of taking locally analytic vectors, following \cite{Pan22}, \cite{JRC2021solid} and \cite{JacintoJoaquínJuan2023SolidLARepII}. 
Subsection \ref{subsectionEquvShvFl} introduces equivariant quasicoherent sheaves on \(\Fl\), and Subsection \ref{subsectionAutShv} introduces certain automorphic sheaves, which are \(\mrm{Solid}_{\QQ_{p}}\)-valued sheaves on the analytic site \(\Fl_{\an}\). In Subsection \ref{subsectionGeoSen}, we sum up the main result of ``geometric Sen theory'', which relates the equivariant quasicoherent sheave to the
automorphic sheaves.
Subsection \ref{subsecArithSenFontaine} defines a general formalism of defining the arithmetic Sen operator, the functor \(E_{0}(-)\), and the Fontaine operator. 
\subsection{Setup}
\label{subsectionNotations}
We fix a neat open compact subgroup $K^{p}=\prod_{l\nmid p}K_{l}\subset G(\mA_{f}^{p})$. Let $S$ be a finite set of places of $\QQ$ including $p$ and $\infty$, such that for any $l\notin S$, $K_{l}=\GL_{2}(\ZZ_{l})$, and write $K^{S}:=\prod_{l\notin S}\GL_{2}(\ZZ_{l})$. We define the Hecke algebra         \( \TT(K^{p}):=C_{c}^{\infty}(K^{p}\backslash G(\mA_{f}^{p})/K^{p},\ZZ), \)
and $\TT^{S}:=C_{c}^{\infty}(K^{S}\backslash G(\mA^{S})/K^{S},\ZZ)$. Then $\TT^{S}$ is a commutative ring and is generated by $T_{l}$ and $S_{l}$ for $l\notin S$.
Concretely, \(\TT^{S}\) is isomorphic to the commutative algebra over \(\ZZ\) freely generated by \(S_{l}^{\pm 1}\) and \(T_{l}\) for \(l\notin S\). We endow both \(\TT^{S}\) and \(\TT(K^{p})\) with the discrete topology.

For any open compact subgroup $K_{p}\subset G(\QQ_{p})$, let $X_{\Kpp}$ be the compactified modular curve over $\Spec\QQ$, and $\mX_{\Kpp}$ be the analytification of $X_{\Kpp}\times_{\QQ}\CC_{p}$. We endow it with the standard log structure at cusps as in \cite[Example 2.1.2]{DLLZ2022logarithmicJAMS}. We will write $\Omega^{1}_{\mX_{\Kpp},\mrm{log}}:=\Omega^{1}_{\mX_{\Kpp}}(C)$, where $C$ denotes the cusps. We see by definition that $\mX_{K^{p}K_{p}'}\to \mX_{\Kpp}$ is Kummer \'etale for $K_{p}$ sufficiently small.
% We will use the period sheaf $\mO\mbb{B}^{+}_{\dR,\log,\mX_{\Kpp}}$ on the pro-Kummer-\'etale site of $\mX_{\Kpp}$ defined in \cite{DLLZ2022logarithmicJAMS}, \cite{DLLZ2019logarithmicFoundational}. 

In  \cite{Scholze15}, Scholze proves the following:
\begin{thm}[{\cite[Theorem 3.1.2]{Scholze15}}]\label{thmScholzeHTPeriod}
There exists a perfectoid space $\mX_{K^{p}}$ such that                 \(
                \mX_{K^{p}}\sim \varprojlim_{K_{p}}\mX_{\Kpp} .
                \) Moreover, there exists a Hodge-Tate period map                 
                \[
                                \pi_{\HT}:\mX_{K^{p}}\to \Fl:=B\backslash G\cong \mP^{1} ,
                                \] which is affinoid in the sense that 
there exists a basis $\mcal{B}$ of open affinoid subsets of $\Fl$, such that for any $U\in\mcal{B}$, $\pi_{\HT}^{-1}(U)$ is affinoid.
\end{thm}
From the above argument, we know that for $K_{p}$ small enough, $\mX_{K^{p}}$ is an object in the pro-Kummer-\'etale site of $\mX_{\Kpp}$. So it makes sense to evaluate $\mO\mbb{B}_{\dR,\log,\mX_{\Kpp}}^{+}$ on $\mX_{K^{p}}$.
\begin{notation}
We denote $\hat{\mO}:=\pi_{\HT,*}\mO_{\mX_{K^{p}}}$.
\end{notation}
\begin{notation}
We will write $\mD(\Fl_{\an})$ for the derived category of sheaves 
on $\Fl_{\an}$ with values in solid $\QQ_{p}$-vector spaces, where $\Fl_{\an}$ denotes $\Fl$ endowed with its analytic topology.
\end{notation}
\begin{lem}\label{lemPushForwardVanish}
We have $R\pi_{\HT,*}\mO_{\mX_{K^{p}}}\cong \pi_{\HT,*}\mO_{K^{p}}$ in $\mD(\Fl_{\an})$. 
\end{lem}
\begin{proof}
We need to prove that \(R^{i}\pi_{\HT,*}(\mO_{\mX_{K^{p}}})=0\) for \(i>0\) in the category of solid \(\QQ_{p}\)-vector spaces. For this, it suffices to show that for any profinite set \(S\), and for \(V\in\mcal{B}\),
\(H^{i}(\pi_{\HT}^{-1}(V),\mO_{\mX_{K^{p}}}(S))=0\) for \(i>0\), where \(\mO_{\mX_{K^{p}}}(S)\) denotes the sheaf of abelian groups over \((\mX_{K^{p}})_{\an}\), sending \(U\) to \(C^{0}(S,\mO_{\mX_{K^{p}}}(U))\cong C^{0}(S,\CC_{p})\hatotimes_{\CC_{p}}\mO_{\mX_{K^{p}}}(U)\). Note that \(C^{0}(S,\CC_{p})\) is perfectoid, and if we put \(\underline{S}:=\Spa(C^{0}(S,\CC_{p}))\),  since $\pi_{\HT}^{-1}(V)$ is affinoid perfectoid, we know that \(\pi_{\HT}^{-1}(V)\times \underline{S}\) is also affinoid perfectoid by \cite[Proposition 6.18]{Scholze12}.
Then for \(i>0\), 
\[H^{i}(\pi_{\HT}^{-1}(V),\mO_{X^{K^{p}}}(S))\cong H^{i}(\pi_{\HT}^{-1}(V)\times_{\CC_{p}}\underline{S},\mO_{\pi_{\HT}^{-1}(V)\times_{\CC_{p}}\underline{S}})\cong 0,
\] by almost purity (\cite[Proposition 6.14]{Scholze12}). 
\end{proof}
The cohomology of the sheaf \(\hat{\mO}\) computes the completed cohomology of \cite{Emerton06}.
\begin{dfn}[Completed cohomology]\label{dfnCompletedCohomology}
We define \[R\Gamma(K^{p},\ZZ/p^{n}):=\varinjlim_{K_{p}}R\Gamma_{\et}(X_{\Kpp,\bar{\QQ}},\ZZ/p^{n}),\] where \(R\Gamma_{\et}(X_{\Kpp,\bar{\QQ}},\ZZ/p^{n})\) is equipped with the trivial condensed structure.

For any solid \(\ZZ_{p}\)-algebra \(R\), we define \emph{the completed cohomology} as \[R\Gamma(K^{p},R):=
(R\varprojlim_{n}R\Gamma(K^{p},\ZZ/p^{n}))\hatotimes_{\ZZ_{p}}R,
\] and \(\ti{H}^{i}(K^{p},R):=H^{i}(R\Gamma(K^{p},R))\).
\end{dfn}
\begin{rmk}\label{rmkCompletedCohomology}
In the original definition of \cite{Emerton06}, one works with the open modular curves, but in the curve case, the open and the proper curves give the same completed cohomology. See for example \cite[\S 4.4.1]{Pan22}.
\end{rmk}
\begin{thm}[Primitive Comparison]\label{thmPrimitiveComparison}
        We have isomorphisms of solid \(\CC_{p}\)-spaces \begin{align}\label{alignEqualHatO}
                R\Gamma(K^{p},\CC_{p})\cong R\Gamma(\mX_{K^{p}},{\mO}_{\mX_{K^{p}}})\cong R\Gamma(\Fl,R\pi_{\HT,*}\mO_{\mX_{K^{p}}})\cong R\Gamma(\Fl,\hat{\mO}).
\end{align}
\end{thm}
\begin{proof}
All the isomorphisms follow from the previous lemma except the first one.
This is essentially in the proof of \cite[Theorem 4.2.1]{Scholze15}. See also \cite[Corollary 4.4.3]{Pan22}. Note that by primitive comparison in \cite[Theorem 1.3]{Scholze13}, 
\[R\Gamma(\mX_{\Kpp},\mO^{+}_{\mX_{\Kpp}}/p^{n})\cong R\Gamma(\Kpp,\ZZ/p^{n}),\] where both sides are discrete \(\ZZ/p^{n}\)-modules,
and then one can take \(\varinjlim_{K_{p}}\), and then \(-\hatotimes_{\ZZ_{p}}\CC_{p}\) to obtain the desired isomorphism of solid \(\QQ_{p}\)-spaces.
\end{proof}

\subsection{Locally analytic vectors}\label{subsecLocAnalyrticVect}
As in \cite{Pan22} and \cite{Pan2209.06II}, we will realize the Sen operator and the Fontaine operator geometrically via the locally analytic vectors in the completed cohomology. 

In this subsection, we recall some prerequisites around locally analytic vectors.
Since we will be working in the derived category,
we will define the functor taking derived locally analytic vectors. 
This is introduced in \cite{Pan22} and rewritten using condensed math in \cite{JRC2021solid}.

For $G$ an analytic group over $\Spa(\QQ_{p},\ZZ_{p})$, let 
$\mO_{G}:=H^{0}(G,\mO)$ denote its ring of analytic functions.
Then we have two actions of $G(\ZZ_{p}):=G(\Spa(\QQ_{p},\ZZ_{p}))$ on $\mO_{G}$:                 \[
                g*_{1}f(g'):=f(g^{-1}g'),\;g*_{2}f(g'):=f(g'g),\;g,g'\in G(\ZZ_{p}) .
                \] Clearly $*_{1}$ and $*_{2}$ commute with each other. 

\begin{dfn}[{\cite[Definition 4.20]{JRC2021solid}, \cite[\S 6.2]{JacintoJoaquínJuan2023SolidLARepII}}]\label{dfnSolidRep}
Let $K$ be a locally compact \(p\)-adic Lie group.
A \emph{solid $K$-representation} $V$ over $\QQ_{p}$ is defined to be a solid \(\QQ_{p}\)-space equipped with an action of $\QQ_{p,\square}[K]$,
 where $\QQ_{p,\square}[K]:=\ZZ_{p,\square}[K]\otimes_{\ZZ_{p}}\QQ_{p}$, and $\ZZ_{p,\square}[K]$ is the Iwasawa algebra. We denote the category of solid \(K\)-representations by \(\mrm{Rep}_{\QQ_{p,\square}}(K)\). 
\end{dfn}
\begin{rmk}
The typical examples are continuous Banach (or LB) $K$-representations. The main benefit with condensed formalism is that we can work comfortably with the derived category. To simplify the notation, we will write as $-\hatotimes_{\QQ_{p}}-$
the derived solid tensor product $-\otimes^{L}_{\QQ_{p,\square}}-$. In the case of Banach $K$-representations, this 
derived tensor product is concentrated in degree $0$ and coincides with the classical completed tensor product (Lemma 3.13 of \cite{JRC2021solid}), so there is no clash between notations.
\end{rmk}
\begin{dfn}[\cite{JRC2021solid}]\label{dfnRla}
Let $G$ be an analytic group (i.e. a group object in the category of rigid varieties) over $\Spa(\QQ_{p},\ZZ_{p})$.
Let \(G'\subset G\) be an open affinoid subgroup. 
Let $\mO_{G'}:=H^{0}(G',\mO_{G'})$. Note that \(G(\QQ_{p}):=G(\Spa(\QQ_{p},\ZZ_{p}))\) is a locally compact Lie group, and \(G'(\QQ_{p})\) is an open compact subgroup. 

If $V$ is a solid $G'(\QQ_{p})$-representation, we define the functor of taking the \emph{derived $G'(\QQ_{p})$-analytic vectors} as               \[
                V^{R-G'(\QQ_{p})-\an}:=R\Gamma((G'(\QQ_{p}),*_{1,3}),V\hatotimes_{\QQ_{p}}\mO_{G'}) ,
                \] where $*_{1,3}$ denotes the diagonal action of $G'(\QQ_{p})$ on $V\hatotimes\mO_{G'}$, in which the action on $\mO_{G'}$ is induced by left multiplication. Note that the action $(G'(\QQ_{p}),*_{2})$ (resp. $(\Lie(G),*_{2})$) on $\mO_{G'}$ induces an additional action of $G'(\QQ_{p})$ (resp. $\Lie(G)$) on $V^{R-G'(\QQ_{p})-\an}$. 

Let \(V\) be a solid \(G(\QQ_{p})\)-representation. We define the \emph{derived $G(\QQ_{p})$-locally analytic vectors} as                 \[
                V^{R-G(\QQ_{p})-\la}:=\varinjlim_{G_{n}\subset G}V^{R-G_{n}(\ZZ_{p})-\an} ,
                \] and the \emph{derived \(G(\QQ_{p})\)-smooth vectors} as \[V^{R-G(\QQ_{p})-\sm}:=\varinjlim_{G_{n}\subset G}R\Gamma(G_{n}(\QQ_{p}),V),\]
                with the colimits going through all the open affinoid subgroups $G_{n}$ of $G$. 
                We will write $V^{R-\la}$ when it causes no confusion. Note that we still have an action of $\Lie(G)$ on $V^{R-\la}$. 

We denote by $\mO_{G,1}$ the stalk of \(\mO_{G}\) at $1\in G(\QQ_{p})$. Then  by definition, \[V^{R-\la}\cong (V\hatotimes\mO_{G,1})^{R-\sm}:= \varinjlim_{G_{n}\subset G}R\Gamma(G_{n}(\QQ_{p}),V\hatotimes H^{0}(G_{n},\mO)).
\]

We say that \(V\) is \emph{locally analytic} if the natural morphism \(V^{R-\la}\to V\) is an isomorphism. We denote by \(\mrm{Rep}_{\QQ_{p,\square}}^{\la}(G'(\QQ_{p}))\) the category of \(G'(\QQ_{p})\)-locally analytic representations, which is a full subcategory of \(\mrm{Rep}_{\QQ_{p,\square}}(G'(\QQ_{p}))\).
% The same definition works verbatim for objects in the derived category of solid representations.
\end{dfn}
\begin{lem}\label{lemRlaCommutatesWithTensor}
Let \(G\) be a compact \(p\)-adic Lie group.
Let \(V\) be a solid representation of \(G\) over \(\QQ_{p,\square}\), and let \(M\in D(\QQ_{p,\square})\), which we regard  as a solid representation of \(G\) by putting the trivial action. 
% Then we can \(V\hatotimes_{\QQ_{p}}M\) also . 
Then \[(V\hatotimes_{\QQ_{p}}M)^{R-\la}\cong V^{R-\la}\hatotimes_{\QQ_{p}}M.\]
\end{lem}
% \begin{rmk}
% This lemma will be used for \(M=\CC_{p}\), \(V=R\Gamma(K^{p},\QQ_{p})\) and \(G=G(\ZZ_{p})\).
% \end{rmk}
\begin{proof}
By shrinking \(G\), we assume that \(G\) is uniform. 
Note that \[(V\hatotimes_{\QQ_{p}}M)^{{R-G-\an}}\cong R\Gamma(G,V\hatotimes_{\QQ_{p}}\mO_{G}\hatotimes_{\QQ_{p}}M),
\] where \(\mO_{G}\) denotes the ring of analytic functions on \(G\),  and \[V^{R-G-\an}\hatotimes m\cong R\Gamma(G,V\hatotimes_{\QQ_{p}}\mO_{G})\hatotimes M.
\] So by the definition of \((-)^{R-\la}\), it suffices to show that for any uniform pro-\(p\) group \(G\), and for any solid representation \(V\) of \(G\), \[R\Gamma(G,V\hatotimes M)\cong R\Gamma(G,V)\hatotimes M.
\] 
By definition, \(R\Gamma(G,-):=\RHom_{G}(\QQ_{p},-)\), where \(\QQ_{p}\) denotes the trivial representation.
Now by Lazard-Serre theorem (\cite[Theorem 5.7]{JRC2021solid}), 
\(\QQ_{p}\) admits a projective resolution \[0\to \QQ_{p,\square}[G]^{\oplus \binom{d}{d} }\to \cdots\to \QQ_{p,\square}[G]^{\oplus \binom{d}{0} }\to \QQ_{p}\to 0.
\] Therefore, 
\[R\Gamma(G,V)\cong \left[
        V^{\oplus\binom{d}{0}}\to \cdots \to V^{\oplus\binom{d}{d}}
\right],\]
where the transition maps are induced by elements in \(\QQ_{p,\square}(G)(*)\), independent of \(V\). 

Then \[R\Gamma(G,V)\hatotimes M\cong \left[
        V^{\oplus\binom{d}{0}}\hatotimes M\to \cdots \to V^{\oplus\binom{d}{d}}\hatotimes M
\right],\]
while \(R\Gamma(G,V\hatotimes M)\) is also represented by the same complex.
\end{proof}
\begin{cor}\label{corRLAcommuteTensorLA}
Let \(G\) be a compact Lie group.
Let \(V,M\) be solid representations of \(G\) over \(\QQ_{p,\square}\), and \(M\) is derived locally analytic.
% i.e. \(M\cong M^{R-\la}\).
% Then we can \(V\hatotimes_{\QQ_{p}}M\) also . 
Then \[(V\hatotimes_{\QQ_{p}}M)^{R-\la}\cong V^{R-\la}\hatotimes_{\QQ_{p}}M.\]
\end{cor}
\begin{proof}
Since \(M\) is derived locally analytic, we know \(M\cong \varinjlim_{G'\subset G} M^{R-G'-\la}\). So by replacing \(M\) by \(M^{R-G'-\an}\) and \(G\) by \(G'\), we can assume that \(M\) is derived \(G\)-analytic. Then we have an isomorphism of \(G\)-representations
\[
(M\hatotimes \mO_{G},*_{M}\times *_{1})\cong (M\hatotimes \mO_{G},*_{1})
\] where \(*_{1}\) denotes the action on \(\mO_{G,1}\) by left multiplication,
and \(*_{M}\) denotes the action on \(M\). 
Then 
\begin{align*}
R\Gamma(G,(V\hatotimes M\hatotimes \mO_{G},*_{V}\times *_{M}\times *_{1}))\cong R\Gamma(G,(V\hatotimes M\hatotimes \mO_{G},*_{V}\times *_{1}))\\\cong R\Gamma(G,(V\hatotimes \mO_{G},*_{V}\times *_{1}))\hatotimes M
\end{align*}
by the proof of Lemma \ref{lemRlaCommutatesWithTensor}. Replacing \(G\) by \(G'\subset G\), and taking colimits along \(G'\subset G\), we obtain the desired isomorphism.
\end{proof}

In the case of modular curves, we will consider the following sheaves.
\begin{notation}
We
denote $\mO^{\la}:=H^{0}((\hat{\mO})^{R-\la})\in D(\Fl_{\an})$, and denote by
$\mO^{\sm}\subset \mO^{\la}$ the subsheaf of \(\hat{\mO}\) consisting of the $G(\QQ_{p})$-smooth vectors.         
\end{notation}
\begin{prop}[{\cite[Proposition 4.3.15]{Pan22}}]\label{propORlaConcentrat}
The sheaf $\hat{\mO}^{R-\la}\in D(\Fl_{\an})$ is concentrated in degree \(0\). In other words, \(\mO^{\la}\cong (\hat{\mO})^{R-\la}\). 
\end{prop} 
\begin{cor}\label{corPrimitiveLA}
We have isomorphisms of \(\Gal_{\QQ_{p}}\times \TT(K^{p})\times G(\QQ_{p})\)-modules
 \[R\Gamma(K^{p},\QQ_{p})^{R-\la}\hatotimes \CC_{p}\cong 
R\Gamma(K^{p},\CC_{p})^{R-\la}\cong R\Gamma(\Fl,\mO^{\la}).   \]

\begin{proof}
The first isomorphism follows from Lemma \ref{lemRlaCommutatesWithTensor}.
By Theorem \ref{thmPrimitiveComparison},
\[R\Gamma(K^{p},\CC_{p})\cong R\Gamma(\Fl,\hat{\mO}).\] Therefore, \[R\Gamma(K^{p},\CC_{p})^{R-\la}\cong R\Gamma(\Fl,\hat{\mO})^{R-\la}\cong R\Gamma(\Fl,\hat{\mO}^{R-\la}).\]
By Proposition \ref{propORlaConcentrat}, the latter is isomorphic to \(R\Gamma(\Fl,\mO^{\la})\).
\end{proof}
\end{cor}
% by (\ref{alignEqualHatO}), we know                
% Hence to understand $\fb$-cohomology of $R\Gamma(K^{p},\CC_{p})^{R-\la}$, Therefore, we can study the $\fb$-cohomology of $\mO^{\la}$.
% 
\subsection{Equivariant sheaves on \(\Fl\)}\label{subsectionEquvShvFl}
\begin{notation}[\cite{BB83}]\label{notationG0B0N0}
Let \(\fg:=\Lie(\GL_{2})\), \(\fg^{0}:=\fg\otimes \mO_{\Fl}\), and let \(\fb^{0}\) (resp. \(\fn^{0}\)) be the sub-vector bundle of \(\fg^{0}\) whose total space has the description \[\fb^{0}=\{(X\in \fg,x\in \Fl):X\in\fb_{x}\},\; (\text{resp. } \fn^{0}=\{(X\in \fg,x\in \Fl):X\in\fn_{x}\}).
\] 
\end{notation}
\begin{dfn}\label{dfnEquivShvOnFl}
        Let $U$ be any open affinoid subspace of $\Fl$.
        For any open affinoid subgroup $G'\subset \GL_{2}(\QQ_{p})$ such that $G'\times U\to \Fl$ factors through \(U\), we denote by $\QCoh_{G'}(U)$ the derived category of the $G'$-equivariant quasi-coherent sheaves on $U$ (in the sense of \cite{Andreychev21pseudocoherent}). More precisely, \[\QCoh_{G'}(U):= \varinjlim_{[n]\in\Delta}\QCoh(U\times G^{\prime,n}).
        \] Roughly, \(\mF\in \QCoh_{G'}(U)\)
        is equivalent to the data of \(\mF\in \QCoh(U)\) equipped with \(\mF\to \mF\hatotimes H^{0}(G',\mO_{G'})\) plus further compatibilities. In particular, we have an action of \(\fg:=\Lie(\GL_{2})\) on \(R\Gamma(U,\mF)\).

        Let $\QCoh_{\fg}(U)$ denote the $2$-colimit of $\QCoh_{G'}(U)$ for all small enough $G'\subset \GL_{2}(\QQ_{p})$ along restriction maps. 
\end{dfn}

\begin{dfn}[ON Banach sheaf]
        \label{dfnONBanach}
Let \(X\) be a rigid variety over \(\Spa(\QQ_{p},\ZZ_{p})\), and \(\mF\in\QCoh(X)\), then we say \(\mF\) is an \emph{ON Banach sheaf} if there exists an analytic cover $\{U_{i}\} $ of \(X\), 
such that \(\mF|_{U_{i}}\cong \widehat{\bigoplus}_{j\in J_{i}}\mO_{U_{i}}\). 
\end{dfn}
\begin{rmk}\label{rmkWhyONBanach}
Given \(\mF\in\QCoh(X)\), we say \(\mF\) is \emph{static} if \(R\Gamma(U,\mF)\) is concentrated in degree \(0\) for any open affinoid subspace \(U\subset X\).
Note that the category of quasi-coherent sheaves in \cite{Andreychev21pseudocoherent} has no t-structure, since the pull-back functor along an open immersion 
is not exact. However, since \(\widehat{\bigoplus}_{j\in J_{i}}\mO_{U_{i}}\) is flat over \(\mO_{U_{i}}\), ON Banach sheaves gives examples of static sheaves. 
\end{rmk}
\begin{dfn}[{\cite[Definition 2.3.5]{Juan2022.05GeoSen}}]
        \label{dfnRelLA}
        Let \(G'\) and \(U\) be as in Definition \ref{dfnEquivShvOnFl}. 
        We say that  $\mF\in\QCoh_{G'}(U)$ is \emph{relative locally analytic} if it is ON Banach, and there is an analytic cover \(\{U_{i}\}_{i\in I}\)
        of \(X\), such that \(\mF|_{U_{i}}\) admits an ON lattice \(\mF^{+}\) over \(\mO^{+}_{U_{i}}\), with a basic \(\{v_{i}\}\) such that there is an open subgroup \(G''\) of \(G'\) and \(\epsilon>0\),  \(G''\) stabilize \(\mF^{+}\), and fix \(v_{i}\!\!\!\mod p^{\epsilon}\in \mF^{+}/p^{\epsilon}\). This definition extends to all \(\mF\in\QCoh_{\fg}(U)\).
        
        % in the sense of  over $U$ equipped with the  action of some small open subgroup of $G(\QQ_{p})$. 
\end{dfn}

For later application, we will need the following category.
\begin{dfn}
        We define $\QCoh_{\fg}(U)^{\fn^{0}}$ to be the category of pairs $(\mF,i)$ with $\mF\in\QCoh_{\fg}(U)$ and a homotopy equivalence $i$ between the morphisms $\mF\to\mF\otimes(\fn^{0})^{\vee}$ and $0$. Note that $\mF\to \mF\otimes(\fn^{0})^{\vee}$ is indeed defined in $\QCoh_{\fg}(U)$ because the action of $\fn^{0}$ commutes with that of $\fg$ and of $\mO_{\Fl}$.
\end{dfn}
\begin{notation}\label{egONBanachN0}
 If \(\mF\) is relative locally analytic, and $\fn^{0}$ acts on $\mF$ by zero, then $\mF$ gives rise to an object in $\QCoh_{\fg}(U)^{\fn^{0}}$. We denote as $\QCoh^{\rla}_{\fg}(U)^{\fn^{0}}$ the subcategory of $\QCoh_{\fg}(U)^{\fn^{0}}$ generated (under filtered colimit, extensions and taking idempotents) by relative locally analytic modules that are killed by $\fn^{0}$.

Note that relative locally analytic modules are static by definition, and in particular $\QCoh_{\fg}^{\rla}(U)^{\fn^{0}}$  is not stable as a subcategory of $\QCoh_{\fg}(U)^{\fn^{0}}$.
\end{notation}

\begin{eg}
For any $(a,b)\in\ZZ^{\oplus 2}$, $\chi^{(a,b)}:B\to \GG_{m}$ is an algebraic representation, which by Beilinson-Berstein localization gives rise to a line bundle on $\Fl$, which we denote as $\omega_{\Fl}^{(b,a)}$ following the normalization of \cite{Pilloni22}. Then $\omega^{(b,a)}_{\Fl}\in \QCoh^{\rla}_{\fg}(U)^{\fn^{0}}$. 
\end{eg}

\subsection{Automorphic sheaves on \(\Fl\)}\label{subsectionAutShv}
We now go on to construct certain ``automorphic sheaves" on $\Fl$ via the Hodge-Tate period map (see \cite[Section 3]{Pilloni22} for details). Let $\bar{B}$ be the opposite Borel subgroup of \(G\) consisting of lower-triangular matrices. We will refer to objects in the following category as automorphic sheaves:
\begin{dfn}\label{dfnAutomorphicShv}
Let \(U\) be an open subspace of \(\Fl\). 
Let $\Mod(\mO^{\sm}\otimes\TT(K^{p})|_{U})$ denote the category of $\mO^{\sm}$-modules in $D(U_{\an},\QQ_{p,\square})$ equipped with an action of \(\TT(K^{p})\), where we regard \(\TT(K^{p})\) as a constant sheaf over \(\Fl_{\an}\). 

Here we do not assume any compatibility between the actions of \(\mO^{\sm}\) and \(\TT(K^{p})\).
In other words, we are considering the category of \(\mO^{\sm}\otimes\TT(K^{p})\)-modules, where the tensor product is the push-out in the category of non-commutative rings.
\end{dfn}
\begin{eg}
We have \(\mO^{\sm},\mO^{\la}\in \Mod(\mO^{\sm}\otimes\TT(K^{p}))\).
\end{eg}
\begin{notation}
        For any algebraic representation 
$V$ of $\bar{B}$, we have an automorphic vector bundle $\mcal{V}_{\Kpp}$ over $\mcal{Y}_{\Kpp}$ by \cite{Milne1990canonical}, which extends over $\mcal{X}_{\Kpp}$,  and we define the associated sheaf $\mcal{V}^{\sm}$ over $\Fl$ as  \( \mcal{V}^{\sm}:=\pi_{\HT,*}(\varinjlim_{K_{p}}\pi_{K_{p}}^{-1}\mcal{V}_{\Kpp}), \) 
where $\pi_{K_{p}}:\mX_{K^{p}}\to \mX_{\Kpp}$.

Then  \(\mcal{V}^{\sm}\in \Mod(\mO^{\sm}\otimes\TT(K^{p}))\), and is further equipped with an action of $G(\QQ_{p})\times \Gal_{\QQ_{p}}$, where the action of $G(\QQ_{p})$ is smooth. We also write $\mcal{V}^{\la}:=\mO^{\la}\otimes_{\mO^{\sm}}\mcal{V}^{\sm}\in\Mod(\mO^{\sm}\otimes\TT(K^{p}))$.
\end{notation}
 It is worth pointing out that $\pi_{\HT,*}$ also produces no higher cohomology here thanks to Tate's acyclicity on $\mX_{\Kpp}$ (see \cite[Proposition 3.2(3)]{Pilloni22}).
\begin{notation}
For $(a,b)\in\ZZ^{\oplus 2}$,
we denote by $\omega^{(a,b),\sm}$ the associated sheaf by putting \(V\) to be the character $\chi^{(a,b)}:\bar{B}\to T\xrightarrow{(a,b)} \GG_{m}$. 
Equivalently, we 
consider the line bundle $\omega^{(a,b),\sm}_{\mX_{\Kpp}}$ on the finite level modular curve $\mX_{\Kpp}$, and $\omega^{(a,b),\sm}:=\pi_{\HT,*}(\varinjlim_{K_{p}}\pi_{K_{p}}^{-1}\omega^{(a,b)}_{\mX_{\Kpp}})$. 

 We note that in more classical term, $\omega^{(a,b)}_{\mX_{\Kpp}}\cong \omega^{\otimes (-b)}_{E}\otimes\omega^{\otimes a}_{E^t}$. 
 In particular, the classical modular forms of weight $1+k$ are given by         \( H^{0}(\Fl,\omega^{(1,-k),\sm}) \),  whose associated Galois representations are of Hodge-Tate weight $0,k$.
\end{notation}
\begin{notation}
We denote \(\Omega^{1}_{\mX_{\Kpp},\log}:=\Omega^{1}_{\mX_{\Kpp}}(D)\), where \(D\) denotes the cusps. 

By Kodaira-Spencer isomorphism, we know $\omega^{(1,-1)}_{\mX_{\Kpp}}\cong\Omega^{1}_{\mX_{\Kpp},\log}$.
We denote $\Omega^{1,\sm}_{\log}:=\omega^{(1,-1),\sm}=\pi_{\HT,*}(\pi_{K_{p}}^{-1}\Omega^{1}_{\mX_{\Kpp},\log})$, and $\Omega^{1,\la}_{\log}:=\Omega^{1,\sm}_{\log}\hatotimes_{\mO^{\sm}}\mO^{\la}$. 
\end{notation}

\begin{notation}
We denote as $D^{\sm}:=\mcal{V}_{\mrm{std}}^{\sm}$ the associated sheaf on $\Fl$ for the standard $2$-dimensional representation $V_{\mrm{std}}$ of $\bar{B}$; explicitly, for quasi-compact open subspace $U$ such that $\pi_{\HT}^{-1}(U)=\pi_{K_{p}}^{-1}(U_{K_{p}})\subset \mX_{K^{p}}$ for small enough open subgroup $K_{p}\subset G(\QQ_{p})$ and $U_{K_{p}}\subset \mX_{\Kpp}$, then $D^{\sm}(U):=\varinjlim_{K_{p}}H^{1}_{\dR}((E_{\Kpp}|_{U_{K_{p}}})/U_{K_{p}})$, where $E_{\Kpp}|_{U_{K_{p}}}$ denotes the universal generalized elliptic curve over $U_{K_{p}}\subset \mX_{\Kpp}$, and $H^{1}_{\dR}((E_{\Kpp}|_{U_{K_{p}}})/U_{K_{p}}):=H^{0}(U_{K_{p}},R^{1}\pi_{*}\Omega^\bullet_{(E_{\Kpp}|_{U_{K_{p}}})/U_{K_{p}},\log})$ denotes the relative log-de Rham cohomology.         
\end{notation}

% \begin{dfn}   
% For any algebraic representation $V$ of $\bar{B}$, we have the associated sheaf $\mcal{V}^{\sm}$ on $\Fl$ as defined above, which is naturally an $\mO^{\sm}$-module, and we define . We define similarly $\omega^{(a,b),\la}:=\omega^{(a,b),\sm}\otimes_{\mO^{\sm}}\mO^{\la}$ and 
% \end{dfn}
\begin{notation}\label{notationBruhatStrat}

We fix the Bruhat stratification on $\Fl\cong\mP^{1}$ 
as $\Fl=U_{w}\cup \{\infty\}, $
where $\infty$ is the unique $B(\QQ_{p})$-fixed point on $\Fl$, and $U_{w}\cong \mA^{1}$ is its complement. Let
    \( j \) and     \( i \)
    denote respectively the embedding of $U_{w}\cong \mA^{1}$ and of $\infty$ into     \( \Fl. \) 

We will write $\mcal{D}_{\infty}:=\mcal{D}(\infty_{\an})$, the category of sheaves valued in the solid \({\QQ}_{p}\)-spaces on the analytic site of \(\infty\). Similar, we write $\mcal{D}_{U_{w}}:=\mcal{D}(U_{w,\an})$. Note that $\mcal{D}_{\infty}\cong D(\mrm{Solid}_{\QQ_{p,\square}})$.
    For any sheaf $\mcal{V}^{\sm}$ on $\Fl$, we denote $j_{!}\mcal{V}^{\sm}:=j_{!}j^{-1}\mcal{V}^{\sm}$ and $i_{*}\mcal{V}^{\sm}:=i_{*}i^{-1}\mcal{V}^{\sm}$, where $\mcal{V}^{\sm}$ is regarded as an analytic sheaf of $\QQ_{p}$-vector spaces on $\Fl_{\an}$. We moreover denote \[H^{i}_{c}(U_{w},\mcal{V}^{\sm}):=H^{i}(\Fl,j_{!}\mcal{V}^{\sm}),\;H^{i}(\infty,\mcal{V}^{\sm})\cong H^{i}(\Fl,i_{*}\mcal{V}^{\sm}).\]
    Note that the latter is the stalk of \(\mcal{V}^{\sm}\) at \(\infty\).
%     In particular, we have $i_{*}\mO^{\sm}$ and $i_{*}\omega^{(a,b),\sm}$ for $(a,b)\in\ZZ^{\oplus 2}$. 
\end{notation}

    It is worth pointing out that in \cite[\S 2.3.3]{Pilloni22}, $\omega^{(a,b),\sm}_{\infty}\in \mcal{D}$ is defined 
for $(a,b)\in\CC_{p}^{\oplus 2}$, and forms a $p$-adic family. 
When $(a,b)\in\ZZ^{\oplus 2}$, $\omega^{(a,b),\sm}_{\infty}$ in \cite{Pilloni22} differs from our $i^{-1}\omega^{(a,b),\sm}$ by a twist of $B(\QQ_{p})$-action, since $B(\QQ_{p})$ acts on our $i^{-1}\omega^{(a,b),\sm}$ smoothly.

\begin{dfn}
        The space of the \emph{overconvergent modular
forms } of weight  $(a,b)$ and tame level $K^{p}$ is defined as $M_{(a,b)}^{\dagger}:=H^{0}(\infty,\omega^{(a,b),\sm})$. We will denote $M_{k+1}^{\dagger}:=M_{(1,-k)}^{\dagger}$.  Note that we have fixed the tame level $K^{p}$ at the beginning, and we omit $K^{p}$ from the expression. 
\end{dfn}
The relation with the classical overconvergent modular forms of Coleman is explained by \cite[Proposition 3.17]{Pilloni22}, that is,  \[(M_{k+1}^{\dagger})^{B(\ZZ_{p})}\cong\varinjlim_{n} M_{k+1}^{\dagger}(K^{p}\Gamma_{1}(p^{n})),\] where $M_{k+1}^{\dagger}(K^{p}\Gamma_{1}(p^{n}))$ denotes the classical overconvergent modular forms of level $K^{p}\Gamma_{1}(p^{n})$ of Coleman (\cite{Coleman1997classical}). Note that the twist $-\otimes\CC_{p}(\omega_{0}\chi)$ in \cite[Proposition 3.17]{Pilloni22} doesn't appear in our setting for the difference explained above.

\subsection{Geometric Sen theory}\label{subsectionGeoSen}
In this subsection, we collect results in geometric Sen theory in \cite{Pilloni22}, \cite{Juan2022.05GeoSen}, \cite{Juan2022.09locallyShi}, which give reformulations of \cite{Pan22}. The goal is to define a functor \(\VB\),
which connects the equivariant sheaves on \(\Fl\) to the automorphic sheaves on \(\Fl\). 
% We start by setting up the relevant categories. 

\begin{dfn}\label{dfnVBNaive}
        We define the functor $\VBn$ as     
                \begin{align*}
                \VBn:\mrm{QCoh}_{\fg}(U)&\to \Mod(\mO^{\sm}\otimes\TT(K^{p})|_{U}),\\
                \mF&\mapsto \left(R\pi_{\HT,*}\circ L\pi_{\HT}^{*}(\mF)\right)^{R-\sm}\cong (\hat{\mO}\hatotimes_{\mO_{\Fl}}\mF)^{R-\sm}.
                \end{align*}
        \end{dfn}
The functor \(\VBn\) is almost always derived, but we can modify the definition to obtain a better functor.        
% \begin{thm}\label{thmMainThmGeomSenVBMainText}
% There is an exact functor \[\VB:\QCoh_{\fg}^{\fn^{0}}\to \Mod(\mO^{\sm}\otimes\TT(K^{p})_{U}),\]
% such that $\VB\cong \VBn\circ R\Gamma(\fn^{0},-)$.

% Moreover, for $\mF\in\QCoh^{\rla}_{\fg}(U)^{\fn^{0}}$, $\VB(\mF)$ is concentrated in degree $0$. In particular, $\VB$ sends a short exact sequence in     \( \QCoh^{\rla}_{\fg}(U)^{\fn^{0}} \) to a short exact sequence of $\mO^{\sm}\otimes\TT(K^{p})$-modules. 
% \end{thm}

% Let us give the proof of Theorem \ref{thmMainThmGeomSenVBMainText}. 
% The following construction of $\VB$ is due to Juan Esteban Rodr\'iguez Camargo:
\begin{dfn}\label{dfnVB}
We define a functor $\VB:\QCoh_{\fg}(U)^{\fn^{0}}\to \Mod(\mO^{\sm}\otimes\TT(K^{p})|_{U})$ by \[\VB((\mF,i)):=R\Gamma(\fg^{0}/\fn^{0},\mO^{\la}\otimes^{L}_{\mO_{\Fl},\square}\mF),\] where we have a canonical action of $\fg^{0}/\fn^{0}$ on $\mF$ thanks to the null homotopy $i$ of $\mF\to\mF\otimes(\fn^{0})^{\vee}$, which makes $\mF$ a direct summand of $R\Gamma(\fn^{0},\mF)$, and we have an action of \(\fg^{0}/\fn^{0}\) on \(\mO^{\la}\) by \cite[Theorem 4.2.7]{Pan22}.
\end{dfn}
\begin{lem}\label{lemVBnToVB} The functor
$\mF\mapsto \Fib(\mF\to\mF\otimes(\fn^{0})^{\vee})$ 
defines a natural functor $R\Gamma(\fn^{0},-):\QCoh_{\fg}(U)\to\QCoh_{\fg}(U)^{\fn^{0}}$. 
Then \(
        \VBn\cong \VB\circ R\Gamma(\fn^{0},-) .
        \)
\end{lem}
\begin{proof}
For $\mcal{G}=\Fib(\mF\to\mF\otimes(\fn^{0})^{\vee})$, there is a canonical homotopy equivalence between $\mcal{G}\to \mcal{G}\otimes(\fn^{0})^{\vee}$ and $0$, and thus we obtain the functor     \( R\Gamma(\fn^{0},-):\QCoh_{\fg}(U)\to\QCoh_{\fg}(U)^{\fn^{0}}. \)
Now we want to factorize $\VBn$. By Theorem 1.5 and Theorem 1.7 of \cite{JRC2021solid},      \[
        \VBn(\mF)\cong (\hat{\mO}\otimes^{L}_{\mO_{\Fl},\square}\mF)^{R-\sm}\cong R\Gamma(\fg,(\hat{\mO}\otimes^{L}_{\mO_{\Fl},\square}\mF)^{R-\la})\cong 
        R\Gamma(\fg,\mO^{\la}\otimes^{L}_{\mO_{\Fl},\square}\mF)  .
        \] Here Corollary \ref{corRLAcommuteTensorLA} is used.
         Now we can extend the action of $\fg$ to $\fg^{0}=\fg\otimes\mO_{\Fl}$, and 
\begin{align*}
R\Gamma(\fg,\mO^{\la}\otimes^{L}_{\mO_{\Fl},\square}\mF)\cong R\Gamma(\fg^{0},\mO^{\la}\otimes^{L}_{\mO_{\Fl},\square}\mF)
&\cong R\Gamma(\fn^{0},\fg^{0}/\fn^{0},R\Gamma(\mO^{\la}\otimes^{L}_{\mO_{\Fl},\square}\mF))\\&\cong R\Gamma(\fg^{0}/\fn^{0},\mO^{\la}\otimes^{L}_{\mO_{\Fl},\square}R\Gamma(\fn^{0},\mF))         ,
\end{align*} where in the last step we use the fact that $\fn^{0}$ acts on $\mO^{\la}$ by zero (\cite[Theorem 4.2.7]{Pan22}). Therefore, we are done by putting $\VB:=R\Gamma(\fg^{0}/\fn^{0},\mO^{\la}\otimes^{L}_{\mO_{\Fl},\square}-)$.
\end{proof}

The following is the main theorem of geometric Sen theory: 
\begin{thm}[{
        % \cite[Corollary 3.3.3]{Juan2022.05GeoSen}, 
        \cite[Theorem 5.2.5, Theorem 5.2.1]{Juan2022.09locallyShi}}]\label{thmMainThmGeomSenVB}
For $\mF\in\QCoh^{\rla}_{\fg}(U)^{\fn^{0}}$, $\VB(\mF)$ is concentrated in degree $0$. In particular, \(\VB\) sends short exact sequences in     \( \QCoh^{\rla}_{\fg}(U)^{\fn^{0}} \) to those in \(\Mod(\mO^{\sm}\otimes \TT(K^{p})|_{U})\).
\end{thm}
\begin{proof}
% This is essentially \cite[Theorem 4.2.5(2)]{Juan2022.09locallyShi}. 
Since the property that $\VB(\mF)$ is concentrated in degree $0$ 
is closed under filtered colimit, idempotents and extensions, we are reduced to the case when $\mF$ is relatively locally analytic and the action of \(\fn^{0}\) is trivial on \(\mF\). We further fix an affinoid open subgroup group \(G''\subset \GL_{2}\) as in Definition \ref{dfnRelLA}. 

Then for any \(K_{p}\subset G''(\QQ_{p})\), \(\pi_{\HT}^{*}(\mF)\) equipped with the action of 
\(K_{p}\) defines a bundle over \(\mX_{K^{p}K_{p},\proet}\). Moreover, by Theorem \ref{thmScholzeHTPeriod}, analytically over \(\mX_{K^{p}K_{p},}\), \(\pi_{\HT}^{*}(\mF)\) admits a lattice \(\mF^{0}\) such that \(\mF^{0}/p^{\epsilon}\) is almost isomorphic to \(\bigoplus_{J}(\mO_{\mX_{\Kpp}}^{+}/p^{\epsilon})\)
for some \(\epsilon>0\) and some index set \(J\). In other words, 
\(\pi_{\HT}^{*}(\mF)\)
as a pro-Kummer-\'etale \(\hat{\mO}_{X}\)-module over \(\mX_{\Kpp}\) is relatively locally analytic ON Banach (\cite[Definition 3.2.1]{Juan2022.05GeoSen}). Therefore, we can apply \cite[Corollary 3.3.3]{Juan2022.05GeoSen} and \cite[Theorem 5.2.5]{Juan2022.09locallyShi} to \(\pi_{\HT}^{*}(\mF)\), which implies that \[
R^{i}\eta_{K_{p},*}(\pi_{\HT,*}\mF)\cong \eta_{K_{p},*}H^{i}(\fn^{0},\pi_{\HT}^{*}\mF)\cong \eta_{K_{p},*}\circ\pi_{\HT}^{*}(H^{i}(\fn^{0},\mF)),
\] where \(\eta_{K_{p},*}\) denotes the push-forward from \(\mX_{\Kpp,\proket}\)
to \(\mX_{\Kpp,\an}\). 
Note that \(R\Gamma(\fn^{0},\mF)\cong \mF\oplus \mF\otimes \fn^{0,\vee}[-1]\). Thus \(H^{i}(\mX_{\Kpp,\an},R\eta_{K_{p},*}(\pi_{\HT}^{*}(\mF)))\)
is zero unless \(i=0,1\). 

Let \(U'\subset U\) be an open affinoid subspace, such that \(\pi_{\HT}^{-1}(U')\) is affinoid, and descends for small enough \(K_{p}\) to an open affinoid \(U'_{K_{p}}\subset \mX_{\Kpp}\).
For \(K_{p}'\subset K_{p}\),
we denote by \(U_{K_{p}'}'\) the preimage of \(U'_{K_{p}}\)
along \(\mX_{K^{p}K_{p}'}\to \mX_{\Kpp}\). Then \[H^{i}(U'_{K_{p}',\an},R\eta_{K_{p}',*}(\pi_{\HT}^{*}\mF))\cong H^{i}(K_{p}',R\Gamma(\pi_{\HT}^{-1}(U')_{\proket},\pi_{\HT}^{*}(\mF))),
\] and thus \[R\Gamma(U',\VBn(\mF))\cong\varinjlim_{K_{p}'}R\Gamma(U'_{K_{p}',\an},R\nu_{*}(\pi_{\HT}^{*}(\mF))).\]
Therefore, \(\VBn(\mF)\) is concentrated in degree \(0,1\).

% Then it is in particular squarable and locally analytic $G(\QQ_{p})^{\dagger}-$equivalent, so we can apply \cite[Theorem 4.2.5(2)]{Juan2022.09locallyShi}. 
% $R\nu_{\infty,*}$ in \cite{Juan2022.09locallyShi} is the same as $(-)^{R-\sm}$ here. In particular, take $\pi:=\pi_{\HT}$ in the quoted theorem,  we know  for     \( \mF\in\QCoh_{\fg}(U), \)       \[
% \VBn(\mF)\cong R\nu_{\infty,*}(\pi_{\HT}^{*}\mF)\cong \nu_{\infty,*}R\Gamma(\fn^{0},\pi^{*}_{\HT}\mF)\cong \nu_{\infty,*}\circ\pi^{*}_{\HT}R\Gamma(\fn^{0},\mF)\in D^{[0,1]}(\mO^{\sm}|_{U}) ,
%         \] where the second isomorphism comes from \cite[Theorem 4.3.1]{Juan2022.09locallyShi}.

% On the other hand, for $\mF\in\QCoh^{\rla}_{\fg}(U)^{\fn^{0}}$,     \( R\Gamma(\fn^{0},\mF)\cong\mF\oplus\mF\otimes(\fn^{0})^{\vee}[-1]. \) 
A priori, $\VB(\mF\otimes(\fn^{0})^{\vee})\in D^{\ge0}(\mO^{\sm}|_{U})$. We know that \[\VBn(\mF)\cong\VB(R\Gamma(\fn^{0},\mF))\cong \VB(\mF)\oplus\VB(\mF\otimes\fn^{0,\vee})[-1].\] We then know that \(\VB(\mF\otimes \fn^{0,\vee})\) is concentrated in degree \(0\).
 Replacing $\mF$ with $\mF\otimes\fn^{0}$, we know that $\VB(\mF)\in D^{[0,0]}(\mO^{\sm}|_{U})$.
\end{proof}

The functor $\VB$ will allow us to reduce the calculation of automorphic sheaves to those of equivariant sheaves.
Let us give some examples:
\begin{eg}\label{egVBomegaEquOmega}
$\VB(\omega^{(b,a)}_{\Fl})\cong\omega^{(b,a),\sm}(a)$ (\cite[Corollaire 3.14]{Pilloni22}), where \((a)\) denotes the Tate twist.
\end{eg}
Another important example will be given in Theorem \ref{thmVBClaIsOla}.

\subsection{Arithmetic Sen operator and Fontaine operator}\label{subsecArithSenFontaine}
In this subsection, we discuss the following important functor \(\Darith(-)\). We will try to make the construction as formal as possible, as we will need to apply the functor to general derived objects.
\begin{notation}
For any (derived solid) $\Gal_{\QQ_{p}}$-representation \(M\) over $\QQ_{p}$, we denote                 \[
                D_{\mrm{arith}}(M):=(R\Gamma(H,M))^{R-
                \Gamma-\la} ,
                \] with \(H:=\Gal_{\QQ_{p}(\zeta_{p^{\infty}})}\)
                and \(\Gamma:=\Gal(\QQ_{p}(\zeta_{p^{\infty}})/\QQ_{p})\cong \ZZ_{p}^{\times}\).
                This construction gives a lax symmetric monoidal functor \[\Darith:\mrm{Rep}_{\QQ_{p,\square}}(\Gal_{\QQ_{p}})\to\mrm{Rep}^{\la}_{\QQ_{p,\square}}(\Gamma),
\]

We fix a generator \(\Theta\in\Lie(\Gamma)\).
by results in \cite[\S 4.3]{JacintoJoaquínJuan2023SolidLARepII}, we can identify \(\Rep^{\la}_{\QQ_{p},\square}(\Gamma)\) with \(\mrm{coMod}_{C^{\la}(\Gamma,\QQ_{p})}(\mrm{Solid}_{\QQ_{p}})\), and we have a morphism of co-algebras \(i_{1}^{-1}:C^{\la}(\Gamma,\QQ_{p})\to \QQ_{p}\llbracket t\rrbracket\) by restricting to the formal neighborhood of \(1\in \Gamma\), where \(t\) is the coordinate such that \(t(\Theta)=1\). 

Therefore, composing the base change along \(i_{1}^{-1}\) with \(\Darith\), we obtain \[\Darith:\Rep_{\QQ_{p,\square}}(\Gal_{\QQ_{p}})\to \Rep^{\la}_{\QQ_{p,\square}}(\Gamma)\to \mrm{Mod}_{\QQ_{p}[\Theta]}(\mrm{Solid}_{\QQ_{p}}),
\] which is also lax symmetric monoidal, where \(\mrm{Mod}_{\QQ_{p}[\Theta]}(\mrm{Solid}_{\QQ_{p}})\) is endowed with the \emph{convolution symmetric monoidal structure}. 
% Using Proposition \ref{propRepEqToStack} below, we have an endomorphism \(\Theta\in \End(\Darith(M))\).
\end{notation}

The following discussion concerns (discrete) vector spaces over \(\QQ\). The main result is Proposition \ref{propf*equalsE0}. One can then apply them to the setting of condensed or solid vector spaces by applying the construction point-wise. 
\begin{dfn}
We denote by $\GG_{a}\cong\Spec \QQ[X]$ the additive group scheme over $\QQ$. We  denote by $\widehat{\GG}_{a}$ its completion at the unit; more precisely, for any $\QQ$-algebra $R$,     \( \widehat{\GG}_{a}(R):=\mrm{Nil}(R).  \). 
\end{dfn}
\begin{rmk}
$\widehat{\GG}_{a}$ is an fppf sheaf by results of de Jong (See Remark 2.2.18 of \cite{BhattLurie2022prismatic}), and we can consider the fppf stacks $B\GG_{a}$ and  $B\widehat{\GG}_{a}$.
\end{rmk}
\begin{prop}[{\cite[Example 2.2.12 \& Proposition 2.4.4]{BhattLurie2022prismatic}}]
        \label{propRepEqToStack}
We have natural symmetric monoidal equivalences of categories:         \[
        \QCoh(B\widehat{\GG}_{a})
        % \cong
        %  \QCoh(\GG_{a})
        % \mrm{coMod}_{\QQ\llbracket t\rrbracket}(\Mod_{\QQ})
        \cong D(\QQ[\Theta]),\] and
        \[\QCoh(B\GG_{a})
         \cong 
        %  \QCoh(\widehat{\GG}_{a})
        \mrm{coMod}_{\QQ[t]}(\Mod_{\QQ})
         \cong 
         D_{\Theta^{\infty}-\mrm{torsion}}(\QQ[\Theta]) ,
        \] 
        where 
        \begin{itemize}
        \item 
        $\QCoh(-)$ denotes the derived category of quasi-coherent sheaves;
        \item $D_{\Theta^{\infty}-\mrm{torsion}}(\QQ[\Theta])$ denotes the full subcategory of $D(\QQ[\Theta])$ such that all the cohomologies are $\Theta^{\infty}$-torsion. 
        \end{itemize} 
\end{prop}
\begin{proof}
It suffices to note that in characteristic $0$, $\GG_{a}\cong \GG_{a}^{\#}$. Then the results follow from the corresponding results in \cite{BhattLurie2022prismatic}. A more general result with a more complete proof is provided by \cite[Proposition 4.2.5]{Juan2024analyticdeRham}.
\end{proof}
Later, we will consider the subspace where the action of \(\Theta\) is nilpotent. This actually has a simple geometric interpretation using stacks.
\begin{dfn}[The functor \(E_{0}\)]\label{dfnFunctorE0}
Let \(M\in D(\QQ)\), equipped with an endomorphism \(\Theta\), i.e. \(M\in D(\QQ[\Theta])\), then we define the \emph{generalized eigenspace for \(\Theta=0\)} by \[E_{0}(M):=R\Gamma_{(\Theta)}(M):=\Fib(M\to M[1/\Theta]).\]
        % In particular, if $M$ comes from a $\QQ[t]/f(t)$-modules with $f(t)\in\QQ[t]$, then $E_{0}(M)$ is taking its generalized eigenspace $t^{\infty}=0$, so our notation is coherent. 
It is easy to see that \(E_{0}:D(\QQ[\Theta])\to D_{\Theta^{\infty}-\mrm{torsion}}(\QQ[\Theta])\)
is the right adjoint to the natural inclusion \(\iota:D_{\Theta^{\infty}-\mrm{torsion}}(\QQ[\Theta])\hookrightarrow D(\QQ[\Theta])\). 

More generally, for \(k\in\QQ\), we denote \(E_{k}(M):=\mrm{Fib}(M\to M[1/(\Theta-k)])\).
\end{dfn}
We thank Longke Tang for suggesting the following proof using stacks.
\begin{prop}[\(E_{0}\) is lax symmetric monoidal]
        \label{propf*equalsE0}
The inclusion $\widehat{\GG}_{a}\to \GG_{a}$ induces a map $f:B\widehat{\GG}_{a}\to B\GG_{a}$. 
Then the following diagrams commute \[
\begin{tikzcd}
\QCoh(B\widehat{\GG}_{a})\arrow[r,"\sim"]
& D(\QQ[\Theta])\\
\QCoh(B\GG_{a})\arrow[r,"\sim"]\arrow[u,"f^{*}"]
&
D_{\Theta^{\infty}-\mrm{torsion}}(\QQ[\Theta])\arrow[u,"\iota",hook], 
\end{tikzcd} \text{ }
\begin{tikzcd}
\QCoh(B\widehat{\GG}_{a})\arrow[r,"\sim"]\arrow[d,"f_{*}"]
& D(\QQ[\Theta])\arrow[d,"E_{0}"]\\
\QCoh(B\GG_{a})\arrow[r,"\sim"]
&
D_{\Theta^{\infty}-\mrm{torsion}}(\QQ[\Theta]),
\end{tikzcd}\] where the horizontal equivalences are given by Proposition \ref{propRepEqToStack}. 

In particular, \(E_{0}:D(\QQ[\Theta])\to D_{\Theta^{\infty}-\mrm{torsion}}(\QQ[\Theta])\) is lax symmetric monoidal for the convolution symmetric monoidal structure. 
% For any $\mF_{M}\in \QCoh(B\widehat{\GG}_{a})$ that corresponds to $M\in D(\QQ[\Theta])$ via Proposition \ref{propRepEqToStack},  
% $f_{*}(\mF_{M})$ corresponds to $E_{0}(M)\in $ via the equivalences in Proposition \ref{propRepEqToStack}. Here $f_{*}$ denotes the derived pushforward, and $E_{0}(M):=R\Gamma_{(t)}(M)$ denotes the local cohomology at $(t)$. 
\end{prop}
\begin{proof} 
By adjunction, it suffices to show that the first diagram commutes.
We need to use the explicit construction of the equivalences in Proposition \ref{propRepEqToStack}. 
See \cite{BhattLurie2022prismatic} for details. Given a \(\QQ[t]\)-comodule \(M\) (resp. \(\QQ\llbracket t\rrbracket\)-comodule), we have \(M\to M\otimes_{\QQ}\QQ[t]\) (resp. \(M\to M\otimes_{\QQ}\QQ\llbracket t\rrbracket\)), and then \(\Theta:M\to M\) corresponds to the coefficient of \(t\) on the right hand side. Then it is straight-forward to verify that the first diagram commutes.
% Given a $k[t]$-module $M$, the corresponding $\mF\in \QCoh(B\widehat{\GG}_{a})$
% is given by the $\mO(\widehat{\GG}_{a})$-comodule:
%         \[
%         M\to M\llbracket X\rrbracket, m\mapsto \exp(tX)(m):=\sum_{i=0}^{\infty} \frac{t^{i}(m)}{i!}\cdot X^{i}.
%         \] On the other hand, given a $t^{\infty}$-torsion $k[t]$-module $M'$, the corresponding $\mF'\in \QCoh(B\GG_{a})$
% is given by the $\mO(\GG_{a})$-comodule:
%         \[
%         M\to M[X], m\mapsto \exp(tX)(m):=\sum_{i=0}^{\infty} \frac{t^{i}(m)}{i!}\cdot X^{i}. 
%         \] Note the sum is finite as $M$ is locally nilpotent. In this way, we see that, via the equivalences in Proposition \ref{propRepEqToStack}, the natural embedding     \( D_{t^{\infty}-\mrm{torsion}}(k\llbracket t\rrbracket)\hookrightarrow D(k[t]) \)
%         corresponds to $f^{*}:\QCoh(B\GG_{a})\to \QCoh(B\widehat{\GG}_{a}) $. Taking their right adjoint, we know $f_{*}$ corresponds to $E_{0}(-):=R\Gamma_{(t)}(-)$ via the equivalences.
\end{proof}

\begin{dfn}[Arithmetic Sen operator]\label{dfnArithSenOp}
We regard \(\CC_{p}\) as an \(\mbb{E}_{\infty}\)-algebra object in \(\mrm{Rep}_{\QQ_{p,\square}}(\Gal_{\QQ_{p}})\). We will refer to an object \(M\in \Mod_{\CC_{p}}(\mrm{Rep}_{\QQ_{p,\square}}(\Gal_{\QQ_{p}}))\) as a \emph{solid semi-linear \(\Gal_{\QQ_{p}}\)-representation over \(\CC_{p}\)}.   

Given \(M\in\Mod_{\CC_{p}}(\Rep_{\QQ_{p,\square}}(\Gal_{\QQ_{p}}))\), 
\(\Darith(M)\in \Mod_{\QQ_{p}(\zeta_{p^{\infty}})}(\mrm{Rep}_{\QQ_{p},\square}^{\la}(\Gamma))\), and \(\Theta\) is \(\QQ_{p}(\zeta_{p^{\infty}})\)-linear, 
as \(\Darith(\CC_{p})\cong \QQ_{p}(\zeta_{p^{\infty}})\). 

We will say that \(M\) \emph{admits an arithmetic Sen operator} if the natural morphism \(\Darith(M)\hatotimes_{\QQ_{p}(\zeta_{p^{\infty}})}\CC_{p}\to M\) 
is an isomorphism, and in this case, \(\Theta\) extends to a unique \(\CC_{p}\)-linear endomorphism \(\Theta\) of \(M\), which we refer to as the \emph{arithmetic Sen operator}.

We will moreover say that \(M\) is \emph{Hodge-Tate} if the action of \(\Theta\) is semisimple, with only finitely many eigenvalues, and all the eigenvalues are in \(\ZZ\). Note that in the derived setting, this is a structure rather than a property. 
\end{dfn}

\begin{eg}
Let \(\rho\) be a finite dimensional representation of \(\Gal_{\QQ_{p}}\) over \(\QQ_{p}\). Then \(\Darith(\rho\otimes\CC_{p})\cong D_{\mrm{sen}}(\rho)\). 
\end{eg} 
\begin{eg}\label{egOlaArithSen}
By the proof of \cite[Theorem 6.3.6]{Juan2022.09locallyShi},
we know $D_{\mrm{arith}}(\mO^{\la})$ is concentrated in degree $0$, and
% \begin{align}\label{equationDescentSenTheoryOla}
$
\mO^{\la}\cong                  D_{\mrm{arith}}(\mO^{\la})\hatotimes_{\QQ_{p}(\zeta_{p^{\infty}})}\CC_{p},$
% \end{align} 
i.e. \(\mO^{\la}\) admits arithmetic Sen operator. 
\end{eg}

\begin{dfn}[The functor \(\hat{E}_{0}\)]
        \label{dfnBdRSemiLinear}
We regard \(B_{\dR}^{+}\) as an \(\mbb{E}_{\infty}\)-algebra object in \(\mrm{Rep}_{\QQ_{p,\square}}(\Gal_{\QQ_{p}})\). We will refer to an object \(M\in \Mod_{B_{\dR}^{+}}(\mrm{Rep}_{\QQ_{p,\square}}(\Gal_{\QQ_{p}}))\) as a \emph{solid semi-linear \(\Gal_{\QQ_{p}}\)-representation over \(B_{\dR}^{+}\)}.   

Given \(M\in\Mod_{B_{\dR}^{+}}(\Rep_{\QQ_{p,\square}}(\Gal_{\QQ_{p}}))\), 
we define 
\[\hatDarith(M):=\varprojlim_{n}\Darith(M/t^{n}) \in \Mod_{\QQ_{p}(\zeta_{p^{\infty}})\llbracket t\rrbracket}(\mrm{Rep}_{\QQ_{p},\square}^{\la}(\Gamma)),\] 
as \(\hatDarith(B_{\dR}^{+})\cong \QQ_{p}(\zeta_{p^{\infty}})\llbracket t\rrbracket\). \(\Theta\) acts on \(\QQ_{p}(\zeta_{p^{\infty}})\llbracket t\rrbracket\) via \(\frac{d}{dt}\), and the \(\Theta\)-action on \(\Darith(M)\) satisfies the Leibniz rule. 

We define \[
% \hat{E}_{0}(\hatDarith(M)):=\varprojlim E_{0}(\Darith(M/t^{n})),
\hat{E}_{0}(M):=\varprojlim (E_{0}(\Darith(M/t^{n}))\hatotimes_{\QQ_{p}(\zeta_{p^{\infty}})}\CC_{p})\to \hat{E}_{0}(M[1/t]):=\varinjlim_{i}\hat{E}_{0}(M\cdot t^{-i}).
% \hat{E}_{0}(\hatDarith(M))\hatotimes_{\QQ_{p}(\zeta_{p^{\infty}})}\CC_{p}.
\]
% We define \[\hatDarith(M[1/t]):=\hatDarith(M)[1/t]\in \Mod_{\QQ_{p}(\zeta_{p^{\infty}})\llparenthesis t\rrparenthesis }(\Rep_{\QQ_{p},\square}^{\la}(\Gamma)), \]
% and \[\hat{E}_{0}(\hatDarith(M[1/t])):=\varinjlim_{i}\hat{E}_{0}(\hatDarith(M)\cdot t^{-i}),\hat{E}_{0}(M[1/t]):=\hat{E}_{0}(\hatDarith(M[1/t]))\hatotimes_{\QQ_{p}(\zeta_{p^{\infty}})}\CC_{p},
% \] 
where \(\cdot t^{-i}\) means twisting \(\Theta\)-action by \(-i\).

More generally, for \(M\in \Mod_{B_{\dR}^{+}}(\mrm{Fil}(\Rep_{\QQ_{p},\square}(\Gal_{\QQ_{p}})))\), where \(\Fil\) denotes the category of filtered objects as in \cite[\S 5.1]{BhattMorrowScholze2019topological}, we define \[
% \hatDarith(M):=\varprojlim_{n}\Darith(M/\Fil^{n}), \;E_{0}(\hatDarith(M)):=\varprojlim_{n}E_{0}(\Darith(M/\Fil^{n})),\;
\hat{E}_{0}(M):=\varprojlim_{n}(E_{0}(\Darith(M/\Fil^{n}))\hatotimes_{\QQ_{p}(\zeta_{p^{\infty}})}\CC_{p})\to \hat{E}_{0}(M[1/t]):=\varinjlim_{n}\hat{E}_{0}(M\cdot t^{-n}).
\]
%  where \(t^{-i}\)

% We will say that \(M\) \emph{admits an arithmetic Sen operator} if the natural morphism \(\Darith(M)\hatotimes_{\QQ_{p}(\zeta_{p^{\infty}})}\CC_{p}\to M\) 
% is an isomorphism, and in this case, \(\Theta\) extends to a unique \(\CC_{p}\)-linear endomorphism \(\Theta\) of \(M\), which we refer to as the \emph{arithmetic Sen operator}.
Assume that \(M/t\) is Hodge-Tate of weights \(a_{0}<a_{1}<\cdots<a_{n}\), i.e. \[M/t\cong \bigoplus_{i=0}^{n}(M/t)^{\Theta=-a_{i}},
\] then the \(t\)-adic ascending filtration on \(M\) (\(\mrm{Fil}^{t}_{-i}:=t^{i}M\)) induces a finite filtration on
\(\hat{E}_{0}(M[1/t])\), such that the non-trivial graded pieces are \[\gr^{t}_{-a_{i}}(\hat{E}_{0}(M[1/t]))\cong 
E_{0}(\gr^{t}_{-a_{i}}(M))\cong (M/t)^{\Theta=-a_{i}}(a_{i}). 
\] 
Moreover, the nilpotent action of \(\Theta\)
on \(\hat{E}_{0}(M[1/t])\) is homotopic to zero when restricted \(\gr^{t}_{*}\) (by a fixed null homotopy given by the structure of being Hodge-Tate), and in particular, \(\Theta^{n+1}\cong 0\).
% i.e. a filtered solid semi-linear representation over \(B_{\dR}^{+}\)
\end{dfn}
\begin{dfn}[Fontaine operator]\label{dfnFontaineOp}
Assume that \(M/t\) is Hodge-Tate, and has only two Hodge-Tate weights \(a_{0}<a_{1}\). Then we have a fiber sequence \[(M/t)^{\Theta=-a_{0}}(a_{0})\to \hat{E}_{0}(M[1/t])\to (M/t)^{\Theta=-a_{1}}(a_{1}),
\] and the action of \(\Theta\) on \(\hat{E}_{0}(M[1/t])\) plus the null-homotopy of \(\Theta\)-action on the graded pieces induces a unique morphism 
\[N^{a_{1}-a_{0}}:(M/t)^{\Theta=-a_{1}}(a_{1})\to (M/t)^{\Theta=-a_{0}}(a_{0})
\] which we refer to as the \emph{Fontaine operator}.
\end{dfn}

We now recall the theorem of Fontaine, which says that the property of being de Rham is determined by the Fontaine operator.
% and de Rhamness. 
% We note that the construction is parallel to that in the geometric setting
% (see Construction \ref{constructionGeomFontaine}).
\begin{prop}[{\cite[Théorème 4.1]{Fontaine2004arithmetique}}]\label{propFontaineOpeClassical}
        Let \(V\) be a finite dimensional Hodge-Tate representation of \(\Gal_{\QQ_{p}}\) over \(\QQ_{p}\). Then \(V\) is de Rham if and only if the action of \(\Theta\) on \(\hat{E}_{0}(V\otimes_{\QQ_{p}}B_{\dR})\) is zero.

        In particular, if \(V\) has only two Hodge-Tate weights \(a_{0}<a_{1}\), then \(V\)
        is de Rham if and only if the Fontaine operator \(N^{a_{1}-a_{0}}=0\).
	% We put for any $\Gal_{\QQ_{p}}$-module $M$,                 \[
	% D_{\mrm{arith}}(M):=R\Gamma(\Gal_{\QQ_{p}(\zeta_{p^{\infty}})},M)^{R-\Gal(\QQ_{p}(\zeta_{p^{\infty}})/\QQ_{p})-\la}              .
	% \] Assume that $(V,\rho)$ is a $2$-dimensional Galois representation
	% \( \rho:\Gal_{\QQ_{p}}\to \GL_{2}(\bar{\QQ}_{p}) \)
	% of Hodge-Tate weight $0,k$ with $k\in\ZZ_{\ge 1}$.
	% Let $N\in \ZZ_{\ge k+1}$. Then $D_{\mrm{arith}}(V\otimes_{\QQ_{p}}\CC_{p})$ (resp. $D_{\mrm{arith}}(V\otimes_{\QQ_{p}}B_{\dR}^{+}/t^{N})$) is concentrated in degree $0$ and is free of rank $2$ over $D_{\mrm{arith}}(\bar{\QQ}_{p}\otimes_{\QQ_{p}}\CC_{p})=\bar{\QQ}_{p}\otimes_{\QQ_{p}} \QQ_{p}(\zeta_{p^{\infty}})$ (resp. over $D_{\mrm{arith}}(\bar{\QQ}_{p}\otimes_{\QQ_{p}} B_{\dR}^{+}/t^{N})=\bar{\QQ}_{p}\otimes_{\QQ_{p}} \QQ_{p}(\zeta_{p^{\infty}})[t]/t^{N}$), and we have an exact sequence                 \[
	% 0\to D_{\mrm{arith}}(V\otimes \CC_{p})^{\Theta=-k}(k)\to E_{0}(D_{\mrm{arith}}(V\otimes B_{\dR}^{+}/t^{N}))\to D_{\mrm{arith}}(V\otimes\CC_{p})^{\Theta=0}\to 0 ,
	% \] where $\Theta$ is the arithmetic Sen operator that comes from taking derivative of the action of $\Gal(\QQ_{p}(\zeta_{p^{\infty}}/\QQ_{p})$,
	% and $E_{0}$ is taking the subspace where the action of $\Theta$ is nilpotent. Then the action of $\Theta$ on the middle term is
	% induced by a natural \emph{Fontaine operator}
	% \[
	% N^{k}:D_{\mrm{arith}}(V\otimes\CC_{p})^{\Theta=0}\to D_{\mrm{arith}}(V\otimes \CC_{p})^{\Theta=-k}(k) .
	% \] Then $V$ is de Rham if and only if the Fontaine operator $N^{k}=0$.
\end{prop}

\section{Geometric Fontaine operator}\label{secGeoFontaine}

The goal of this section is to define the relevant objects and state the main geometric theorem (Theorem \ref{thmFontaine=Theta}). Subsection \ref{subsecBcoho} will recall the calculation of 
\(\fb\)-cohomology of \(\mO^{\la}\) in \cite{Pan22}, \cite{Pilloni22}, and define the geometric Fontaine operator (Corollary \ref{corGeomFontaine}) by applying the general construction in Definition \ref{dfnFontaineOp}. Subsection \ref{subsectionStatement} will state the main theorem (Theorem \ref{thmFontaine=Theta}), which vaguely says that ``geometric Fontaine operator=theta operator''. Subsection \ref{subsectionPerverse} defines perverse t-structures on \(D(\Fl_{\an})\) and on (some category related to) \(\QCoh_{\fg}(\Fl)^{\fn^{0}}\). These t-structures will be used in some later proofs. There are two separate cases in Theorem \ref{thmFontaine=Theta}, but using BGG complex, we prove in
Subsection \ref{subsecBGGComplex} that the two cases are equivalent. We will focus on one case, and prove Theorem \ref{thmFontaine=Theta} in the next section. 

\subsection{\(\fb\)-cohomology of \(\mO^{\la}\)}
\label{subsecBcoho}

\begin{notation}\label{egClaVBOla}
For $G=\GL_{2}(\QQ_{p})$, recall \(\mO_{G,1}\) from Definition \ref{dfnRla}.
Note that $\mO_{G,1}$ carries two (infinitesimal) actions of $G$, given respectively left multiplication and right multiplication. We denote them as $*_{1}$ and $*_{2}$ respectively. 
On $\mO_{G,1}\hatotimes\mO_{\Fl}$, there is a third action coming from the action of $G$ on $\mO_{\Fl}$, which we denote as $*_{3}$. 

Then $(\mO_{G,1}\hatotimes\mO_{\Fl},*_{1,3})$ gives an object in $\QCoh_{\fg}^{\rla}(\Fl)$ (Definition \ref{dfnRelLA}). We denote  \[ C^{\la}:= R\Gamma(\fn^{0},\mO_{G,1}\hat{\otimes\mO_{\Fl}})\cong (\mO_{G,1}\hatotimes\mO_{\Fl})^{(\fn^{0},*_{1,3})}. \]
 Then $C^{\la}\in \QCoh_{\fg}^{\rla}(\Fl)^{\fn^{0}}$.
%  This is the central object of geometric Sen theory. 
 It still carries commuting actions $(\fg,*_{1,3})$, $(\fg,*_{2})$ and $(\fh,\theta_{\fh})$, with $(\fh,\theta_{\fh})$ coming from the action of $(\fh\otimes\mO_{\Fl}\cong\fb^{0}/\fn^{0},*_{1,3})$. We also note that $C^{\la}$ is $\GL_{2}(\QQ_{p})$-equivariant over $\Fl$ for the action $*_{1,2,3}$.
\end{notation}
\begin{thm}[{\cite[Theorem 4.1, Lemme 4.5, Lemme 4.6]{Pilloni22}, \cite[Theorem 6.3.6]{Juan2022.09locallyShi}}]\label{thmVBClaIsOla}
We have an \(G(\QQ_{p})\times \Gal_{\QQ_{p}}\)-equivariant isomorphism in \(\Mod(\mO^{\sm}\otimes\TT(K^{p}))\),  $$\VB(C^{\la},*_{1,3})\cong\mO^{\la},$$ along which $(\fg,*_{2})$
%  and $(\fh,\theta_{\fh})$ 
 corresponds to constant action of $\fg$.
%   and the minus of the horizontal action of $\fh$ on $\mO^{\la}$ respectively. 
  Moreover, \(\theta_{\fh}(1,0)\) acting on \(C^{\la}\) induces the arithmetic Sen operator on \(\mO^{\la}\) (Example \ref{egOlaArithSen}).
\end{thm}

We can then calculate the $\fb$-cohomology of $\mO^{\la}$ by first calculating that of \((C^{\la},*_{2})\). We start by introducing some notation.

% In particular, the horizontal action of $(1,0)\in\fh$ induces the arithmetic Sen operator on $\mO^{\la}$ by our normalization. 
\begin{notation}[The functor \(\mE_{(a,b)}\)]\label{notationEab}
For $(a,b)\in\ZZ^{\oplus 2}$, let \(M\in \Mod_{\QQ_{p}[\Theta]}(\mrm{Solid}_{\QQ_{p}})\) equipped with a commuting action of \(\fb\),
we define the functor         \[
        \mE_{(a,b)}(M):=E_{0}(\RHom_{\fb}((a,b),M))\otimes \chi^{(-a,-b)},
        \]  where $E_{0}$ is taking the nilpotent part with respect to the action of $\Theta$, and $-\otimes\chi^{(-a,-b)}$ only twists the action of $B(\QQ_{p})$ such that the action of $B(\QQ_{p})$ on \(\mE_{(a,b)}(M)\) is smooth. 

If \(M\in \Mod_{B_{\dR}^{+}}(\Rep_{\QQ_{p,\square}}(\Gal_{\QQ_{p}}))\) equipped with a commuting action of \(\fb\), such that \(M/t\) admits an arithmetic Sen operator, we define \[\hat{\mE}_{(a,b)}(M):=\hat{E}_{0}(\RHom_{\fb}((a,b),M))\otimes \chi^{(-a,-b)},
\] with \(\hat{E}_{0}\)
as in Definition \ref{dfnBdRSemiLinear}.

We will also define \(\mE_{(a,b)}(C^{\la}(i))\) using the same formula, where $E_{0}$ is taking the generalized eigenspace $\theta_{\fh}(1,0)=0$, $-(i)$ refers to a twist of action of $\theta_{\fh}(1,0)$ by $-i$, and \(\fb\) acts on \(C^{\la}\) via \(*_{2}\).

Then by Theorem \ref{thmVBClaIsOla}, \[\VB(\mE_{(a,b)}(C^{\la}(i))\cong \mE_{(a,b)}(\mO^{\la}(i)).\]
\end{notation}

%  In particular, Theorem \ref{mainthmPanPilloni} comes from applying $\VB$ to the following theorem in \cite{Pilloni22}:
\begin{dfn}\label{dfnidagger}
We will use Notation \ref{notationBruhatStrat}.

We define
\[i^{\dagger}(\mF):=\varinjlim_{\infty\in U}R\Gamma(U,\mF)\] with the colimit taken over the open neighborhoods of \(\infty\in\Fl\).
This gives a quasi-coherent sheaf on the dagger space $\infty^{\dagger}:=\varprojlim_{\infty\in U} U$, i.e. a point equipped with the structure sheaf \(\mO_{\Fl,\infty}:=\varinjlim_{\infty\in U}H^{0}(U,\mO_{\Fl})\). 

We define \(j_{!}:\QCoh(U_{w})\to \QCoh(\Fl)\) to be the left adjoint to \(j^{*}:\QCoh(\Fl)\to \QCoh(U)\)

This definition makes sure that for \(\mF\in\QCoh(\Fl)\), we have a fiber sequence in \(\QCoh(\Fl)\)         \[
        j_{!}j^{*}\mF\to\mF\to i_{*}i^{\dagger}\mF\xrightarrow{+1}.
\]
\end{dfn}
\begin{thm}[\cite{Pilloni22}]\label{mainthmPilloniOnFl}
% Let     \( (a,b)\in\CC_{p}^{\oplus2}, \)
% with     \( a-b\notin -\mrm{Liouv}. \) 
Let     \( (a,b)\in\ZZ^{\oplus 2}, \) $k:=a-b+1$. 

(1) If     \( k>0, \)
then we have a $(B(\QQ_{p}),*_{2})\times\theta_{\fh}(1,0)$-equivariant isomorphism in $\QCoh^{\rla}_{\fg}(U)^{\fn^{0}}$    \[
        \RHom_{(\fb,*_{2})}((a,b),C^{\la})\otimes\chi^{(-a,-b)}\cong \mcal{N}_{b}\oplus \mcal{N}_{1+a} ,
        \] where $\mcal{N}_{1+a}\cong i_{*}i^{\dagger}\omega^{(1-b,-1-a)}_{\Fl}[-1]$, and $\mcal{N}_{b}\cong \left[\omega^{(-a,-b)}_{\Fl}-i_{*}i^{\dagger}\omega^{(-a,-b)}_{\Fl}[-1]\right]$. Moreover, $\theta_{\fh}(1,0)$ acts on $\mcal{N}_{t}$ by $-t$ for $t\in\{b,1+a\}$.

(2) If     \( k\le 0, \)
then we have a $(B(\QQ_{p}),*_{2})\times\theta_{\fh}(1,0)$-equivariant isomorphism in $\QCoh^{\rla}_{\fg}(U)^{\fn^{0}}$        
  \[\RHom_{(\fb,*_{2})}((a,b),C^{\la})\otimes\chi^{(-a,-b)}\cong
      \left[  \mcal{N}_{b} - \mcal{N}_{1+a}[-1] \right],
        \]with $\mcal{N}_{b}\cong j_{!}j^{*}\omega^{(-a,-b)}_{\Fl}$, and $\mcal{N}_{1+a}\cong i_{*}i^{\dagger}\omega^{(1-b,-1-a)}_{\Fl}[-1]$. Moreover, $\theta_{\fh}(1,0)$ acts on $\mcal{N}_{t}$ by $-t$ for $t\in\{b,1+a\}$.

In particular, if $k\ne 0$, we have isomorphisms                \[
                \mE_{(a,b)}(C^{\la}(i))\cong \begin{cases}
                \mcal{N}_{i},&i=b,1+a;\\
                0,&\text{else},
                \end{cases}
                \] with \(\mcal{N}_{i}\) described as in (1) or (2).
\end{thm}
\begin{rmk}
Note that the action of $\theta_{\fh}(1,0)$ on $\omega^{(a,b)}_{\Fl}$ is the scalar multiplication by $b$, following the normalization of \cite{Pilloni22}.
\end{rmk}

We can now apply the functor \(\VB(-)\).
\begin{thm}[\cite{Pilloni22}]\label{mainthmPanPilloni}
% Let     \( (a,b)\in\CC_{p}^{\oplus2}, \)
% with     \( a-b\notin -\mrm{Liouv}. \) 
Let     \( (a,b)\in\ZZ^{\oplus 2}. \)
Then 

(1) if     \( a-b\ge 0, \)
then we have a $\TT(K^{p})\times B(\QQ_{p})\times \Gal_{\QQ_{p}}$-equivariant isomorphism  in $D(\Fl_{\an})$    \[
        \RHom_{\fb}((a,b),\mO^{\la})\otimes\chi^{(-a,-b)}\cong \mcal{N}_{b}(-b)\oplus \mcal{N}_{1+a}(-1-a) ,
        \] where $\mcal{N}_{1+a}\cong i_{*}\omega^{(1-b,-1-a),\sm}[-1]$, and $\mcal{N}_{b}$ lies in a distinguished triangle         \[
                \omega^{(-a,-b),\sm}\to \mcal{N}_{b}\to i_{*}\omega^{(-a,-b),\sm}[-1]\xrightarrow{+1}.
                \]  

(2) if     \( a-b <0, \)
then we have a $\TT(K^{p})\times B(\QQ_{p})\times \Gal_{\QQ_{p}}$-equivariant distinguished triangle in $D(\Fl_{\an})$       
  \[
        \mcal{N}_{b}(-b) \to \RHom_{\mfk{b}}((a,b),\mO^{\la})\otimes \chi^{(-a,-b)}\to \mcal{N}_{1+a} (-1-a),
        \] where \(\mcal{N}_{1+a}\cong i_{*}\omega^{(1-b,-1-a),\sm}[-1]\) and \(\mcal{N}_{b}\cong j_{!}\omega^{(-a,-b),\sm}\).

In particular, if $a-b\ne -1$, we have isomorphisms                \[
                \mE_{(a,b)}(\mO^{\la}(i))\cong \begin{cases}
                \mcal{N}_{i},&i=b,1+a;\\
                0,&\text{else},
                \end{cases}
                \] with \(\mcal{N}_{i}\) described as in (1) or (2).
\end{thm}
\begin{proof}
Using Theorem \ref{thmMainThmGeomSenVB} and Theorem \ref{thmVBClaIsOla}, the results follow from Theorem \ref{mainthmPilloniOnFl} and Lemma \ref{rmkDeduceSVfromFl} below. 
\end{proof}
\begin{lem}\label{rmkDeduceSVfromFl}
We have \(\VB(j_{!}j^{*}\omega^{(-a,-b)}_{\Fl})\cong j_{!}\omega^{(-a,-b),\sm}\) and \(\VB(i_{*}i^{\dagger}\omega^{(-a,-b)}_{\Fl})\cong i_{*}\omega^{(-a,-b),\sm}\).
\end{lem}
\begin{proof}
Given Example \ref{egVBomegaEquOmega}, we only note that the same formalism of $\VB$ also works for $\infty^{\dagger}$ (by extending the functor \(\VB\) by colimits) and for $U_{w}$, and that $\VB$ commutes with $i^{\dagger}$ as well as with $j^{*}$. 
\end{proof}
% \begin{rmk}
% To deduce Theorem \ref{mainthmPanPilloni}, we need to prove that \(\VB(j_{!}j^{*}\omega^{(-a,-b)}_{\Fl})\cong j_{!}\omega^{(-a,-b),\sm}\) and \(\VB(i_{*}i^{\dagger}\omega^{(-a,-b)}_{\Fl})\cong i_{*}\omega^{(-a,-b),\sm}\). 
% \end{rmk}

We can also apply the construction of Definition \ref{dfnBdRSemiLinear}
to the period sheaves.
\begin{dfn}[de Rham period sheaves]
        \label{dfnShvBdRonFl}
For \(n\in\NN\),
we define the sheaf $\BdR^{+}/t^{n}$ on $\Fl_{\an}$ as $\pi_{\HT,*}(\mbb{B}^{+}_{\dR,\log,\mX_{\Kpp}}/t^{n}|_{\mX_{K^{p}}})$, where $\mbb{B}^{+}_{\dR,\log,\mX_{\Kpp}}$ is the pro-Kummer-\'etale de Rham sheaf on $\mX_{\Kpp}$ (\cite[Definition 2.2.3]{DLLZ2022logarithmicJAMS})
and $\mX_{K^{p}}$
is regarded as the object in the pro-Kummer-\'etale site of $\mX_{\Kpp}$. We further define \(\mbb{B}_{\dR}^{+}:=\varprojlim \mbb{B}_{\dR}^{+}/t^{n}\). 

We define \(\BdR^{+,\la}/t^{n}\)
% $t^{N_{1}}\BdR^{+,\la}/t^{N_{2}}\BdR^{+,\la}\subset t^{N_{1}}\BdR^{+}/t^{N_{2}}\BdR^{+}$ 
as the sheaf of derived $G(\QQ_{p})$-locally analytic vectors of \(\BdR^{+}/t^{n}\). By Proposition \ref{propORlaConcentrat}, we know that \(\BdR^{+,\la}/t^{n}\) is concentrated in degree \(0\), and is filtered by 
\(\mO^{\la}(i)\) for \(0\le i\le n-1\). We define \(\BdR^{+,\la}:=\varprojlim_{n}\BdR^{+,\la}/t^{n}\) and \(\BdR^{\la}:=\BdR^{+,\la}[1/t]\).          

We define similarly $\OBdR^{+,\la}/\Fil^{n}$ on $\Fl_{\an}$ as the subsheaf of $G(\QQ_{p})$-locally analytic vectors in $\pi_{\HT,*}(\mO\mbb{B}^{+}_{\dR,\log,\mX_{\Kpp}}/\Fil^{n}|_{\mX_{K^{p}}})$ for $n\in\NN$, where \(\mO\mbb{B}^{+}_{\dR,\log,\mX_{\Kpp}}\)
is as in \cite[Definition 2.2.10]{DLLZ2022logarithmicJAMS}. We define \(\OBdR^{+,\la}:=\varprojlim_{n}\OBdR^{+,\la}/\Fil^{n}\), and \(\OBdR^{+}:=\OBdR^{+,\la}[1/t]\). 
\end{dfn}
\begin{rmk}
We are using the non-standard notations and writing $\mO\mbb{B}^{+}_{\dR}:=\mO\mbb{B}^{+}_{\dR,\log}$ for simplicity of notations.
\end{rmk}
\begin{cor}[Geometric Fontaine operator]\label{corGeomFontaine}
For \((a,b)\in\ZZ^{\oplus2},k:=a+1-b\ne 0\), we have a fiber sequence \[\mcal{N}_{\max(b,1+a)}\to \hat{\mE}_{(a,b)}(\BdR^{\la})\to \mcal{N}_{\min(b,1+a)},
\] and the action of \(\Theta\) on \(\hat{\mE}_{(a,b)}(\BdR^{\la})\)
is induced by a unique morphism \[N^{|k|}:\mcal{N}_{\min(b,1+a)}\to \mcal{N}_{\max(b,1+a)},
\] which we refer to as the \emph{geometric Fontaine operator}.
Here \(\hat{\mE}_{(a,b)}(-)\) is as in Notation \ref{notationEab}, and \(\mcal{N}_{i}\)'s are as in Theorem \ref{mainthmPanPilloni}. 
\end{cor}
\begin{proof}
By Theorem \ref{mainthmPanPilloni}, \(\BdR^{+}/t\) is Hodge-Tate, and then the rest follows as in Definition \ref{dfnBdRSemiLinear} and \ref{dfnFontaineOp}.
\end{proof}

\subsection{Geometric Fontaine operator}\label{subsectionStatement} The main theorem of this article concerns an alternative description of the geometric Fontaine operator \(N^{|k|}\) in Corollary \ref{corGeomFontaine}. 

% In this way,
% Theorem \ref{mainthmPanPilloni} provides us a concrete description of         \( \mE_{(a,b)}(t^{k_{0}}\BdR^{+,\la}/t^{k_{1}}\BdR^{+,\la}). \)
% \begin{prop}\label{propStrEabBdR}
% For $(a,b)\in\ZZ^{\oplus2}$. Then for     \( k_{0}\le \min(1+a,b)\le\max(1+a,b)+1\le k_{1} \), we have the following  exact sequences in $\Perv$ that are equivariant with respect to the action of $B(\QQ_{p})\times \Gal_{\QQ_{p}}$:

% (1) If $a-b>-1$, then         \[
%        0 \to \mE_{(a,b)}(\mO^{\la}(1+a)) \to \mE_{(a,b)}(t^{k_{0}}\BdR^{+,\la}/t^{k_{1}}\BdR^{+,\la})\to\mE_{(a,b)}(\mO^{\la}(b))\to 0, 
%         \] with     \( \mE_{(a,b)}(\mO^{\la}(1+a))\cong i_{*}\omega^{(1-b,-1-a),\sm} [-1]\), and         \[
%                 0\to \omega^{(-a,-b),\sm}\to \mE_{(a,b)}(\mO^{\la}(b))\to i_{*}\omega^{(-a,-b),\sm}[-1]\to 0 .
%                 \]

% (2) If $a-b\le -1$, then       \[
%         0 \to j_{!}\omega^{(-a,-b),\sm}\to \mE_{(a,b)}(t^{k_{0}}\BdR^{+,\la}/t^{k_{1}}\BdR^{+,\la})\to i_{*}\omega^{(1-b,-1-a),\sm}[-1]\to 0.
%         \] 
% \end{prop}
% \begin{proof}
% This follows immediately from the description of         \( \mE_{(a,b)}(\mO^{\la}(i)) \) for $i\in\ZZ$ given by Theorem \ref{mainthmPanPilloni}.
% \end{proof}
% The following is our main theorem:

\begin{thm}\label{thmFontaine=Theta}

Let the notation be as in Corollary \ref{corGeomFontaine}. Then the geometric Fontaine operator \( N^{|k|}:\mcal{N}_{\min(b,1+a)}\to \mcal{N}_{\max(b,1+a)} \)
% as in Construction \ref{constructionGeomFontaine}
can be described as follows: 

(1) If $k>0$, $N^{k}$ coincides with the composition \begin{align*}N^{k}:\mcal{N}_{b}\to i_{*}\omega^{(-a,-b),\sm}[-1]\xrightarrow{\theta^{k}} i_{*}\omega^{(1-b,-1-a),\sm} [-1]\cong \mcal{N}_{1+a}.
\end{align*}

(2) If $k<0$, $N^{-k}$ coincides with either composition of the following square   \[
         \begin{tikzcd} \mcal{N}_{1+a}\arrow[r,equal] & i_{*}\omega^{(1-b,-1-a),\sm}[-1]\arrow[rd,"N^{-k}"]\arrow[r,"\mrm{Cou}"]\arrow[d,"\theta^{-k}"] & j_{!}\omega^{(1-b,-1-a),\sm}\arrow[d,"\theta^{-k}"] \\  & i_{*}\omega^{(-a,-b),\sm}[-1]\arrow[r,"\mrm{Cou}"]      &
          j_{!}\omega^{(-a,-b),\sm}\arrow[r,equal] & \mcal{N}_{b},
                 \end{tikzcd}        
                \] where the Cousin map $\mrm{Cou}$ is defined as the connecting morphism coming from the natural extension         \[
0\to j_{!}\omega^{(a,b),\sm}\to \omega^{(a,b),\sm}\to i_{*}\omega^{(a,b),\sm}\to 0 .
                        \]
\end{thm}
\begin{rmk}
If $k=0$, with suitable formulation, the arithmetic Sen operator can be described in the same way as (2). 
This case is essentially Proposition 6.3 of \cite{Pilloni22}.
\end{rmk}

The proof of Theorem \ref{thmFontaine=Theta} 
will be finished in Subsection \ref{subsectionGeneral(a,b)a-bge0}. In fact, part (2) follows from part (1) (Corollary \ref{cork<0FollowsFromk>0}). Roughly,
parts (1) and (2) are related by the BGG complex to be discussed in Subsection \ref{subsecBGGComplex}, and the case \(k<0\) injects into the case \(k>0\) in the perverse t-structure to be defined in Subsection \ref{subsectionPerverse}.  See the proof of Corollary \ref{cork<0FollowsFromk>0} for details, especially (\ref{equaBasicRelk>0andk<0}).
% and Subsection \ref{subsectionGeneral(a,b)a-ble0}. 

Let us first explain one reduction step: 
\begin{rmk}\label{rmkTwistInBdR}
Although we have stated the theorem for general $(a,b)\in\ZZ^{\oplus 2}$ with $a-b+1\ne 0$, 
the cases for $(a,b)$ and for  $(a+1,b+1)$ 
are essentially the same. In particular, we could always restrict ourselves to the case $(k,0)$. Let us explain the reason here: taking the $1$-dimensional representation $\det$, we obtain the automorphic vector bundle $\omega^{(1,1)}_{\mX_{\Kpp}}$ (with Hodge filtration and Gauss-Manin connection), and $\underline{\det}$ the associated local system. By the Riemann-Hilbert correspondence (\cite{DLLZ2022logarithmicJAMS}), we have \[
                \omega^{(1,1)}_{\mX_{\Kpp}}\otimes_{\mO_{\mX_{\Kpp}}}\mO \mbb{B}_{\dR,\mX_{\Kpp},\log}\cong \underline{\det}\otimes_{\QQ_{p}}\mO \mbb{B}_{\dR,\mX_{\Kpp},\log} .
                \] 
so we have a canonical map                 \[
                \omega^{(1,1)}_{\mX_{\Kpp}}\otimes_{\QQ_{p}}\underline{\det}^{-1}\to \mO \mbb{B}_{\dR,\mX_{\Kpp},\log}  .
                \] 
                Evaluating on $\mX_{K^{p}}$ and pushing forward to $\Fl$, we obtain a canonical non-zero map                \[
                                \omega^{(1,1),\sm}\otimes_{\QQ_{p}}{\det}^{-1}\to t^{-1}\OBdR^{+,\la},
                                \] which is compatible with connections on both sides.
The polarization gives a canonical non-zero section $s$ of $\omega^{(1,1),\sm}\cong \omega_{E}^{-1}\otimes\omega_{E^{t}}$. We choose an arbitrary generator of the $1$-dimensional $\QQ_{p}$-vector space ${\det}^{-1}$,
say $v_{0}$, then $s_{1}:=s\otimes v_{0}$
gives a section of $t^{-1}\OBdR^{+,\la}$.
                                 Now that $\nabla_{\mrm{GM}}(s)=0$, we know $\nabla(s_{1})=0$, which implies 
by \cite[Corollary 6.13]{Scholze13} that $s_{1}$   
is a section of 
$t^{-1}\BdR^{+,\la}$. 

We note that the action of $B(\QQ_{p})$ on $s_{1}$ is given by ${\det}^{-1}=(-1,-1)$.
On the other hand, the action of $\Gal_{\QQ_{p}}$
on $s_{1}$ is induced by its action on $s\in H^{0}(\Fl,\omega^{(1,1),\sm})$, which is trivial. 

We note that the unique locally analytic function in $\mX_{K^{p}}$ whose $\GL_{2}(\QQ_{p})$-action is given by determinant is $t_{0}\in H^{0}(\Fl,\mO^{\la})$ constructed in \cite[Subsection 4.3.1]{Pan22} (which is denoted as $t$ loc. cit.), which comes essentially from the Tate pairing. However, $\Gal_{\QQ_{p}}$ acting on $t_{0}$ by $\chi_{\mrm{cycl}}$. Therefore, we know that $s_{1}\notin \BdR^{+,\la}, $
and its image along the projection         \( t^{-1}\BdR^{+,\la}\twoheadrightarrow \mO^{\la}(-1) \)
is given by $t^{-1}t_{0}$. 
% Thus multiplying by $s_{1}$ induces an isomorphism $s_{1}\mO^{\la}(i)\cong \mO^{\la}(-1)$.

Thus by multiplying by 
$s_{1}$ on $\BdR$, the action of $\Gal_{\QQ_{p}}$ is unchanged, the Hecke action is twisted by $\omega^{(1,1),\sm}$, and the action of $B(\QQ_{p})$ is changed by ${\det}^{-1}$, so we can reduce the case of $(a,b)$ to that of $(a-1,b-1)$.
\end{rmk}

\subsection{Perversity}\label{subsectionPerverse}

We can observe that \(\RHom_{\fb}((a,b),\mO^{\la})\) looks like a perverse sheaf.
In this subsection, we define a perverse t-structure on \(D(\Fl_{\an})\) to make this idea precise. 
This will be used in 
% Subsection \ref{subsectionGeneral(a,b)a-ble0} 
% for 
the proof of Corollary \ref{cork<0FollowsFromk>0}. 

\subsubsection{Perverse t-structure on \(D(\Fl_{\an})\)}
We will use the \emph{recollement} of t-structures from \cite{BBDG2018faisceaux}. More precisely, we fix the Bruhat stratification  \( \Fl=\{\infty\}\cup U_{w}, \) and assign them with dimension $0$ and $1$ respectively.
\begin{dfn}\label{dfnPerverseShv}
We define         \[
        ^{p}\mcal{D}^{\ge i}:=\{\mF\in\mcal{D}(\Fl_{\an}):j^{*}\mF\in\mcal{D}^{\ge i}_{U_{w}},i^{!}\mF\in \mcal{D}^{\ge i+1}_{\infty} \},
        \] and         \[
                ^{p}\mcal{D}^{\le i}:=\{
\mF\in\mcal{D}(\Fl_{\an}):j^{!}\mF\in\mcal{D}^{\le i}_{U_{w}},i^{-1}\mF\in \mcal{D}^{\le i+1}_{\infty}
                \} .
                \] We define the category of \emph{(Bruhat) perverse sheaves} 
                as     \( \Perv:=\pmD^{\ge 0}\cap \pmD^{\le 0}. \)
\end{dfn}

\begin{rmk}
Here $i^{!}$ is defined as follows: we first define the non-derived $i^{!}_{0}$ as taking the sections with support at $\infty$, and then define $i^{!}$ as the right derived functor of $i^{!}_{0}$. 
\end{rmk}

\begin{prop}[{\cite[Theorem 1.4.10]{BBDG2018faisceaux}}]\label{propPerverseBBDG}
        \( (^{p}\mcal{D}^{\le 0},^{p}\mcal{D}^{\ge0}) \)
        defines a $t$-structure of $\mcal{D}(\Fl_{\an})$.
\end{prop}

\begin{lem}\label{lemPerversity} For $(a,b)\in\ZZ^{\oplus 2}$, 
    \( j_{!}\omega^{(a,b),\sm}\), $\omega^{(a,b),\sm}$  and     \( i_{*}\omega^{(a,b),\sm}[-1] \)
    are perverse sheaves. 
\end{lem}
\begin{proof}
 \( i_{*}\omega^{(a,b),\sm}[-1] \) is clearly perverse. For $\omega^{(a,b),\sm}$, we have $i^{-1}\omega^{(a,b),\sm}\in \mcal{D}^{\le 0}_{\infty}$. We claim that we have $i^{!}\omega^{(a,b),\sm}\in \mcal{D}^{\ge 1}_{\infty}$. This is because     \( H^{0}(i^{!}\omega^{(a,b),\sm})=i_{0}^{!}\omega^{(a,b),\sm}. \) Assume that we have     \( s\in H^{0}(\infty,i_{0}^{!}\omega^{(a,b),\sm}), \)
then $s$ gives rise to a section     \( s\in H^{0}(\Fl,\omega^{(a,b),\sm}) \)
that is supported at $\infty$. $s$ arise from some $s_{0}\in H^{0}(\mX_{K},\omega^{(a,b)})$, but this forces $s$ to be zero because the image of $U_{w}$ in $\mX_{K}$ is dense. Now we have         \( i_{0}^{!}(j_{!}\omega^{(a,b),\sm})\subset i_{0}^{!}\omega^{(a,b),\sm}, \)
and thus we also obtain the perversity of $j_{!}\omega^{(a,b),\sm}$.
\end{proof}
% \begin{rmk}
% The perversity of \( j_{!}\omega^{(a,b),\sm} \) and     \( i_{*}\omega^{(a,b),\sm}[-1] \)
%     can be generalized to general Shimura varieties, but the generalization of the perversity of $\omega^{(a,b),\sm}$ still seems difficult.
% \end{rmk}

\begin{cor}\label{corPerverse}
For   \( (a,b)\in\ZZ^{\oplus2}, \)
% with     \( a-b\notin -\mrm{Liouv}, \)
    \( \RHom_{\mfk{b}}((a,b),\mO^{\la})\in \Perv. \)
    Moreover, for any $k\in\ZZ$,     \( \mE_{(a,b)}(\mO^{\la}(k))\in \Perv \).
\end{cor}
\begin{proof}
This follows immediately from Theorem \ref{mainthmPanPilloni} and Lemma \ref{lemPerversity}.
\end{proof}
% \begin{rmk}
% The functor $E_{k}$ makes sense on the level of derived category. This follows from \cite[Tag 0955]{stacks-project} as all the cohomology of    \( \RHom_{\mfk{b}}((a,b),\mO^{\la}) \)
% is torsion for $\QQ[\Theta]$. The decomposition`' is functorial, and thus also applies to sheaves and is compatible with various structures.
% \end{rmk} 
% Then by Corollary \ref{corPerverse}, for any $(a,b)\in\ZZ^{\oplus2}$ and $i\in\ZZ$, $\mE_{(a,b)}(\mO^{\la}(i))$ is a perverse sheaf.

\subsubsection{t-structure on quasicoherent sheaves}
It will also be useful to have a t-structure (perverse or natural) on \(\QCoh^{\rla}_{\fg}(\Fl)^{\fn^{0}}\), such that the sheaves $\omega^{(a,b)}_{\Fl},i_{*}i^{\dagger}\omega^{(a,b)}_{\Fl}[-1]$, and $j_{!}j^{*}\omega^{(a,b)}_{\Fl}$ lie in the heart. This part will only be used in some technical proof, e.g. Lemma \ref{lemCasek=1CoveredInPilloni}, proof of Lemma \ref{lemFaltingsExtOnFl}, and that of Proposition \ref{propCalculateg-cohomo}.

The problem is that the derived $\infty$-category defined in \cite{Andreychev21pseudocoherent} 
does not carry a natural t-structure. We apply the following general construction to produce a t-structure (on another suitable category): 
\begin{lem}
\label{lemPerverseOnQCoh}
Let $\mcal{C}_{0}\subset\QCoh_{\fg}(\Fl)^{\fn^{0}}$ be the full subcategory spanned by extensions of    \( \omega^{(a,b)}_{\Fl},i_{*}i^{\dagger}\omega^{(a,b)}_{\Fl}[-1],j_{!}j^{*}\omega^{(a,b)}_{\Fl} \) for all $(a,b)\in\ZZ^{\oplus2}$. Let $B_{0}^{\le0}$ (resp. $B_{0}$) be the full subcategory of     \( \QCoh_{\fg}(\Fl)^{\fn^{0}} \) generated
by     \( \mcal{C}_{0} \) under small colimits and extensions (resp. small colimits, extensions and shift). Then $B_{0}$ is a stable subcategory. 

Let $B:=\Ind(B_{0})$. Then \(B\) admits a ``\emph{perverse t-structure}"     \( ({}^{p}B^{\le0},{}^{p}B^{\ge0}) \) such that $\mcal{C}_{0}\subset {}^{p}B^{\ge 0}\cap {}^{p}B^{\le 0}={}^{p}B^{\heartsuit}$. We will denote by ${}^{p}H^{i}(-)$ the $i$-th cohomology with respect to this perverse t-structure.
\end{lem}
\begin{rmk}
The embedding $\mcal{C}_{0}\hookrightarrow B$ is fully faithful. Thus the abelian category ${}^{p}B^{\heartsuit}$ provides 
a convenient ambient category.
\end{rmk}
\begin{proof}
There is no nonzero map $i_{*}i^{\dagger}\omega^{(a,b)}_{\Fl}\to\omega^{(a',b')}_{\Fl}$ or to $j_{!}j^{*}\omega^{(a',b')}_{\Fl}$, which can be seen by taking an affinoid chart around $\infty$.
and we have a canonical colimit-preserving fully faithful functor
    \( \Ind(B_{0}^{\le0})\hookrightarrow \Ind(B_{0}) \). Thus by adjoint functor theorem (\cite[Corollary 5.5.2.9]{Lurie2009HTT}), we have a right adjoint $L:\Ind(B_{0})\to \Ind(B_{0}^{\le0})$. Let us denote $B:=\Ind(B_{0})$. Since $\Ind(B_{0}^{\le0})$ is closed under extension, by \cite[Proposition 1.2.1.16]{Lurie2017HA}, $\Ind(B_{0}^{\le0})$ defines a t-structure of $B$, say     \( ({}^{p}B^{\le0},{}^{p}B^{\ge0}), \) 
such that     \( {}^{p}B^{\le0}=\Ind(B_{0}^{\le0}). \) 
Then we have $\mcal{C}_{0}\subset B_{0}^{\le0}\subset {}^{p}B^{\le0}$. Moreover, ${}^{p}B^{\ge 0}\cong\{F\in B|\Hom(F',F[-1])\cong0,\;\forall F'\in {}^{p}B^{\le0}\}\cong \{F\in B|\RHom(F',F)\in D^{\ge0},\;\forall F'\in {}^{p}B^{\le0}\}$.
Now for any $F\in\mcal{C}_{0}$,     \( \{F'\in B|\RHom(F',F)\in D^{\ge 0}\} \)
is closed under colimits and extensions, and contains $\mcal{C}_{0}$, which implies that it contains ${}^{p}B^{\le 0}$. Thus we know $\mcal{C}_{0}\subset {}^{p}B^{\ge 0}$. Hence $\mcal{C}_{0}\subset {}^{p}B^{\ge 0}\cap {}^{p}B^{\le 0}={}^{p}B^{\heartsuit}$. 
\end{proof}
\begin{rmk}[Natural t-structure]\label{rmkNaturalt-str}
Let $\mcal{C}_{1}\subset\QCoh_{\fg}(\Fl)^{\fn^{0}}$ be the full subcategory spanned by extensions of    \( \omega^{(a,b)}_{\Fl},i_{*}i^{\dagger}\omega^{(a,b)}_{\Fl},j_{!}j^{*}\omega^{(a,b)}_{\Fl} \) for all $(a,b)\in\ZZ^{\oplus2}$.
If we replace $\mcal{C}_{0}$ with $\mcal{C}_{1}$, then the same proof still works, and produce another $t$-structure $(B^{\le0},B^{\ge0})$ on $B$, such that $\mcal{C}_{1}\subset B^{\heartsuit}:=B^{\le0}\cap B^{\ge0}$.  We will denote by $H^{i}(-)$ the $i$-th cohomology with respect to this \emph{natural t-structure}.
\end{rmk}

% \begin{rmk}\label{rmkNormalization}
% Our normalization follows that of \cite{Pilloni22}: the representation $\chi^{(a,b)}$ of $B$ gives rise to $\omega_{\Fl}^{(b,a)}$, with the arithmetric Sen operator $\Theta$
% given by $\theta_{\fh}(1,0)$, and the horizontal action $\theta_{\fh}$ of $\fh$
% acting on $\omega_{\Fl}^{(b,a)}$ by $(a,b)$. Moreover, $\VB(\omega^{(a,b)}_{\Fl})\cong\omega^{(a,b),\sm}$. See \cite{Pilloni22} for the definition of the functor $\VB$.
% \end{rmk}

\subsection{BGG complex}\label{subsecBGGComplex}
In this subsection, we study the BGG complex, and the induced morphism on \(\fb\)-cohomologies. This will allow us to relate (1) and (2) in Theorem \ref{thmFontaine=Theta}. As a byproduct, 
we will also compute $$\RHom_{\fg}(\Sym^{k-1}V,\mO^{\la}),$$ 
where $V$ denotes the standard representation of $\GL_{2}$,
which is a reinterpretation of \cite[Theorem 0.5]{Emerton06}.
% By Theorem \ref{thmVBClaIsOla}, it suffices to calculate \(\RHom_{\fg}(\Sym^{k-1}V,(C^{\la},*_{2}))\). 
% In fact,
% this follows from the computation of $\fb$-cohomology via  
\begin{dfn}[Verma modules]
For any $(a,b)\in\ZZ_{p}^{\oplus}$, let $\Delta_{(a,b)}:=U(\fg)\otimes_{U(\fb)}\chi^{(a,b)}$ denote the Verma module. In particular, we have \[\RHom_{\fg}(\Delta_{(a,b)},C^{\la})\cong \RHom_{\fb}((a,b),C^{\la}).\]
\end{dfn}
Let \(k\in\ZZ_{\ge 1}\) and \(m\in\ZZ\).
We have the classical Bernstein-Gelfand-Gelfand complex of $\fg$-representations (\cite{BernsteinGelfandGelfand1976certainCat}) 
\begin{align}\label{alignBGGComp}
0\to \Delta_{(-1+m,k+m)}\to \Delta_{(k-1+m,m)}\to \Sym^{k-1}V\otimes \det{}^{m}\to 0.
\end{align}
Then we have a distinguished triangle \begin{align}\label{alignExactFrombTog}
        \RHom_{\fg}(\Sym^{k-1}V\otimes\det{}^{m},C^{\la})\to \RHom_{\fb}((k-1+m,m),C^{\la})\\
        \to \RHom_{\fb}((-1+m,k+m),C^{\la})\xrightarrow{+1}.
\end{align}
Note that the sequence (\ref{alignExactFrombTog})
is equivariant for $(\fg,B(\QQ_{p}))$-actions, and also for the horizontal action \((\fh,\theta_{\fh}).\) Recall that the action of $\theta_{\fh}(1,0)$ is semisimple, and has eigenvalues $-k-m,-m$ by Theorem \ref{mainthmPilloniOnFl}. 

Therefore, the sequence (\ref{alignExactFrombTog}) has a decomposition into $\theta_{\fh}(1,0)=-m$ and $=-k-m$ described by Theorem \ref{mainthmPilloniOnFl}:  \begin{align}\label{alignExactCohbtoCohg1}
        \RHom_{\fg}(\Sym^{k-1}V\otimes\det{}^{m},C^{\la})^{\theta_{\fh}(1,0)=-k-m}\to \mcal{N}_{k+m}^{(1)}\to \mcal{N}_{k+m}^{(2)}\xrightarrow{+1},
\end{align}
with \(\mcal{N}_{k+m}^{(1)}\cong i_{*}i^{\dagger}\omega^{(1-m,-k-m)}_{\Fl}[-1] \)
and \(\mcal{N}_{k+m}^{(2)}\cong j_{!}j^{*}\omega^{(1-m,-k-m)}_{\Fl}\),
and \begin{align}\label{alignExactCohbtoCohg2}
        \RHom_{\fg}(\Sym^{k-1}V\otimes\det{}^{m},C^{\la})^{\theta_{\fh}(1,0)=-m}\to \mcal{N}_{m}^{(1)}\to \mcal{N}_{m}^{(2)}\xrightarrow{+1},
\end{align}
with \[\mcal{N}_{m}^{(1)}\cong \left[\omega^{(1-k-m,-m)}_{\Fl}-i_{*}i^{\dagger}\omega^{(1-k-m,-m)}_{\Fl}[-1]\right]\]
and \(\mcal{N}_{m}^{(2)}\cong i_{*}i^{\dagger}\omega^{(1-k-m,-m)}_{\Fl}[-1].\)
\begin{prop}\label{propCalculateg-cohomo}
        We have $\GL_{2}(\QQ_{p})$-equivariant isomorphisms 
        \begin{align}\label{alignCalculateg-cohFl}
                \RHom_{\fg}(\Sym^{k-1}V\otimes\det{}^{m},C^{\la})\cong 
                \mcal{N}^{(0)}_{k+m}\oplus \mcal{N}^{(0)}_{m},\\
                \mcal{N}^{(0)}_{k+m}\cong \omega^{(1-m,-k-m)}_{\Fl}[-1],\mcal{N}^{(0)}_{m}\cong \omega^{(1-k-m,-m)}_{\Fl},
        \end{align} with $\theta_{\fh}(1,0)$ acting on $\mcal{N}^{(0)}_{i}$ by the scalar $-i$ ($i=m,m+k$).

        Moreover, the sequences (\ref{alignExactCohbtoCohg1})
        and (\ref{alignExactCohbtoCohg2})
        coincide with the standard sequences (up to a scalar in $\CC_{p}^{\times}$)
        \begin{align}
                \label{alignExact1TranslateFl}
                \omega^{(1-m,-k-m)}_{\Fl}[-1]\to i_{*}i^{\dagger}\omega^{(1-m,-k-m)}_{\Fl}[-1]\xrightarrow{\mrm{Cou}} j_{!}j^{*}\omega^{(1-m,-k-m)}_{\Fl}\xrightarrow{+1},
        \end{align}
        and \begin{align}
                \label{alignExact2TranslateFl}
                \omega^{(1-k-m,-m)}_{\Fl}\to \left[\omega^{(1-k-m,-m)}_{\Fl}-i_{*}i^{\dagger}\omega^{(1-k-m,-m)}_{\Fl}[-1]\right]\\\to i_{*}i^{\dagger}\omega^{(1-k-m,-m)}_{\Fl}[-1]\xrightarrow{+1}.
        \end{align}
\end{prop}
\begin{proof}
When restricted to the open Bruhat strata $U_{w}\cong \mA^{1}\subset \Fl$, $\mcal{N}_{k+m}^{(1)}|_{\mA^{1}}\cong \mcal{N}_{m}^{(2)}|_{\mA^{1}}\cong 0$,
and thus 
\begin{align*}
\RHom_{\fg}(\Sym^{k-1}V\otimes\det{}^{m},C^{\la})^{\theta_{\fh}(1,0)=-k-m}|_{\mA^{1}}\cong \mcal{N}^{(2)}_{k+m}|_{\mA^{1}}[-1]\\\cong \omega^{(1-m,-k-m)}_{\Fl}|_{\mA^{1}}[-1]
\end{align*}
and \[\RHom_{\fg}(\Sym^{k-1}V\otimes\det{}^{m},C^{\la})^{\theta_{\fh}(1,0)=-m}|_{\mA^{1}}\cong \mcal{N}_{m}^{(1)}|_{\mA^{1}}\cong \omega^{(1-k-m,-m)}_{\Fl}|_{\mA^{1}}.\]
In particular, both are locally free of rank $1$ over $\mO_{\Fl}$ when restricted to $\mA^{1}$ (up to shift). 

As the action of $\fg$ on $\Sym^{k-1}V\otimes\det{}^{m}$ can be upgraded to an action of $\GL_{2}(\QQ_{p})$, 
$\RHom_{\fg}(\Sym^{k-1}V\otimes\det{}^{m},C^{\la})$ is $\GL_{2}(\QQ_{p})$-equivariant.

Thus we know \(\RHom_{\fg}(\Sym^{k-1}V\otimes\det{}^{m},C^{\la})^{\theta_{\fh}(1,0)=-m}\) and \(\RHom_{\fg}(\Sym^{k-1}V\otimes\det{}^{m},C^{\la})^{\theta_{\fh}(1,0)=-k-m}[1]\) are $\GL_{2}(\QQ_{p})$-equivariant line bundles over $\Fl$ (up to shift). 

By the classification of $(\fg,\GL_{2}(\QQ_{p}))$-equivariant line bundles on $\Fl$ in  \cite[Lemme 2.16, Proposition 2.18]{Pilloni22}, 
we know that \begin{align*}
\RHom_{\fg}(\Sym^{k-1}V\otimes\det{}^{m},C^{\la})^{\theta_{\fh}(1,0)=-k-m}\cong \omega^{(1-m,-k-m)}_{\Fl}[-1],
\\\RHom_{\fg}(\Sym^{k-1}V\otimes\det{}^{m},C^{\la})^{\theta_{\fh}(1,0)=-m}\cong \omega^{(1-k-m,-m)}_{\Fl}.
\end{align*}
This finishes the proof of (\ref{alignCalculateg-cohFl}).
        % We have already proven 
        % above. 

        This implies in particular that the map \(\mcal{N}^{(0)}_{k+m}\to \mcal{N}_{k+m}^{(1)}\)
        is non-zero, and then by Lemma \ref{lemAnyMapIsIsoInQCohg} below,
        it coincides with the standard map \[\omega^{(1-m,-k-m)}_{\Fl}\to i_{*}i^{\dagger}\omega^{(1-m,-k-m)}_{\Fl}[-1].\]
        This gives the proof of (\ref{alignExact1TranslateFl}).
        For (\ref{alignExact2TranslateFl}),
        we note that $\mcal{N}_{m}^{(1)}\to \mcal{N}^{(2)}_{m}$ is non-zero, and thus it factors through $i_{*}i^{\dagger}\omega^{(1-k-m,-m)}_{\Fl}[-1]\to i_{*}i^{\dagger}\omega^{(1-k-m,-m)}_{\Fl}[-1]$ by the natural t-structure (Remark \ref{rmkNaturalt-str}),
        which is forced to be the scalar multiplication again by Lemma \ref{lemAnyMapIsIsoInQCohg}.
\end{proof}

The following lemma is the ``counterpart" of Lemma \ref{lemNoBinv} over $\Fl$.
\begin{lem}\label{lemAnyMapIsIsoInQCohg}
(1) For any $\chi,\chi'\in\ZZ^{\oplus2}$, if $f:i^{\dagger}\omega^{\chi'}_{\Fl}\to i^{\dagger}\omega^{\chi}_{\Fl}$ is $B(\QQ_{p})$-equivariant, then either $f=0$, or $\chi=\chi'$, and $f$ is multiplication by a scalar in $\CC_{p}$.

(2) For any $\chi\in\ZZ^{\oplus2}$, any open $U\subset \Fl$, and for any $f:\omega^{\chi}_{\Fl}|_{U}\to \omega^{\chi}_{\Fl}|_{U}$ in $\QCoh_{\fg}(U)$, if $f\ne 0$, then $f$ is an isomorphism. The same holds for $f:i^{\dagger}\omega^{\chi}_{\Fl}\to i^{\dagger}\omega^{\chi}_{\Fl}$ in $\QCoh_{\fg}(\infty^{\dagger})$.
\end{lem}
\begin{proof}
(1) Assume $f\ne0$.  Then $f$ corresponds to a unique non-zero section $f\in H^{0}(\infty,i^{\dagger}\omega^{\chi-\chi'}_{\Fl})$ that is $B(\QQ_{p})$-invariant. Then by definition, $f$ can be extended to a neighborhood $\infty\in U\subset \Fl$ where $f$ is nowhere vanishing. Note that such extension has to be unique since the transition map \( H^{0}(U,\omega^{\chi-\chi'}_{\Fl})\to H^{0}(U',\omega^{\chi-\chi'}_{\Fl}) \) is injective for $U'\subset U$. Using $B(\QQ_{p})$-action, we have $B(\QQ_{p})\cdot U=\Fl$, so $f$ can be extended to a nowhere vanishing $B(\QQ_{p})$-invariant global section \( H^{0}(\Fl,\omega_{\Fl}^{\chi-\chi'}). \)
This implies that $\omega_{\Fl}^{\chi-\chi'}$ is isomorphic $B(\QQ_{p})$-equivariantly to $\mO_{\Fl}$, and thus $\chi-\chi'=0$ and $\chi=\chi'$. By GAGA, we know in addition that $f\in\CC_{p}$.

(2) By twisting, we can assume that $\chi=(0,0)$, and then $f$ corresponds to a nonzero section $s\in H^{0}(U,\mO_{\Fl})$. Moreover, since $f$ is $\fg$-equivariant, we know $\fg\cdot s=0$. This implies that $s$ is constant. Thus, $f$ is an isomorphism.
\end{proof}

% Now combining with (\ref{alignExactCohbtoCohg1})
% and (\ref{alignExactCohbtoCohg2}) and Lemma \ref{lemAnyMapIsIsoInQCohg},
% we obtain the following proposition:
We can then apply \(\VB(-)\) to Proposition \ref{propCalculateg-cohomo}.

\begin{prop}\label{propCalculateg-cohomoOnSV} For $k\in\ZZ_{\ge1},m\in \ZZ$,
        we have a $\GL_{2}(\QQ_{p})\times \TT(K^{p})$-equivariant isomorphisms 
        \begin{align}\label{alignCalculateg-coh}
                \RHom_{\fg}(\Sym^{k-1}V\otimes\det{}^{m},\mO^{\la})\cong 
                \mcal{N}^{(0)}_{k+m}(-k-m)\oplus \mcal{N}^{(0)}_{m}(-m),\\
                \mcal{N}^{(0)}_{k+m}\cong \omega^{(1-m,-k-m),\sm}[-1],\mcal{N}^{(0)}_{m}\cong \omega^{(1-k-m,-m),\sm}.
        \end{align}

        Moreover, we have a
        fiber sequence \begin{align}\label{alignSeqExactOnSVbtog}
        \RHom_{\fg}(\Sym^{k-1}V\otimes\det{}^{m},\mO^{\la})\to \RHom_{\fb}((k-1+m,m),\mO^{\la})\\
        \to \RHom_{\fb}((-1+m,k+m),\mO^{\la})\xrightarrow{+1}.
\end{align}
        and the weight $(k+m)$-part and the weight $m$-part coincide respectively with the standard sequences (up to a scalar in $\CC_{p}^{\times}$)
        \begin{align}
                \label{alignExact1Translate}
                \omega^{(1-m,-k-m),\sm}[-1]\to i_{*}\omega^{(1-m,-k-m),\sm}[-1]\xrightarrow{\mrm{Cou}} j_{!}\omega^{(1-m,-k-m),\sm}\xrightarrow{+1},
        \end{align}
        and \begin{align}
                \label{alignExact2Translate}
                \omega^{(1-k-m,-m ),\sm}\to \left[\omega^{(1-k-m,-m ),\sm}-i_{*}\omega^{(1-k-m,-m ),\sm}[-1]\right]\\\to i_{*}\omega^{(1-k-m,-m ),\sm}[-1]\xrightarrow{+1}.
        \end{align}
      \end{prop}
\begin{proof}
        This follows from Proposition \ref{propCalculateg-cohomo} by applying the functor $\VB$ and using Lemma \ref{rmkDeduceSVfromFl}.
\end{proof}      
% \begin{rmk}
%         Taking $R\Gamma(\Fl,-)$, we recover part of \cite[Section 4.3.4]{Emerton06},
%         that is, the description of the locally algebraic vectors in the completed cohomology as classical modular forms.
% \end{rmk}

\begin{cor}\label{cork<0FollowsFromk>0}
Theorem \ref{thmFontaine=Theta} (2) follows from Theorem \ref{thmFontaine=Theta} (1), and vice versa.
\end{cor}
\begin{proof}
By Remark \ref{rmkTwistInBdR}, it suffices to consider \((a,b)=(-1,k)\) or \((k-1,0)\) for \(k\in\ZZ_{\ge 1}\).
Apply \(\RHom_{\fg}(-,\BdR^{\la})\) to (\ref{alignBGGComp}), and
we have a morphism  \[
        \bar{\fn}^{k}:     \hat{\mE}_{(k-1,0)}(\BdR^{+,\la})\to \hat{\mE}_{(-1,k)}(\BdR^{+,\la})\otimes (\bar{\fn}^{\vee})^{\otimes k},
            \] which is compatible with the arithmetic Sen operator $\Theta$.

            By Theorem \ref{mainthmPanPilloni},
            let us denote  \begin{align*}
                    \RHom_{\fb}&((k-1,0),\mO^{\la})\cong \mcal{N}_{k}^{(1)}\oplus \mcal{N}_{0}^{(1)}(-k),\\
                    \mcal{N}_{k}^{(1)}&\cong i_{*}\omega^{(1,-k),\sm}[-1],\mcal{N}_{0}^{(1)}\cong [\omega^{(1-k,0),\sm}-i_{*}\omega^{(1-k,0),\sm}[-1]] ,\\
            \RHom_{\fb}&((-1,k),\mO^{\la})\cong \mcal{N}_{k}^{(2)}\oplus \mcal{N}_{0}^{(2)}(-k),\\\mcal{N}_{k}^{(2)}&\cong j_{!}\omega^{(1,-k),\sm}[-1],\mcal{N}_{0}^{(2)}\cong i_{*}\omega^{(1-k,0),\sm}[-1].   
            \end{align*}
            Hence, we obtain the following commutative diagram (in the category of perverse sheaves) with the horizontal morphisms given by the geometric Fontaine operators (in Corollary \ref{corGeomFontaine}) 
                \begin{equation}\label{equaBasicRelk>0andk<0}
                \begin{tikzcd}
 \mcal{N}^{(1)}_{0}\arrow[d,two heads]\arrow[r,"N^{k,(1)}"] & \mcal{N}^{(1)}_{k}\arrow[d,"\mrm{Cou}",hook]\\
 \mcal{N}^{(2)}_{0}\arrow[r,"N^{k,(2)}"]\arrow[ur,dotted,"\theta^{k}"] & \mcal{N}^{(2)}_{k}                
                 \end{tikzcd} .
                \end{equation}

                Given Theorem \ref{thmFontaine=Theta} (1), we have a unique factorization of \(N^{k,(1)}:\mcal{N}_{0}^{(1)}\to\mcal{N}_{k}^{(1)}\)
                as 
                \[ \mcal{N}^{(1)}_{0}\twoheadrightarrow 
                        \mcal{N}_{0}^{(2)}\cong i_{*}\omega^{(1-k,0),\sm}[-1]\xrightarrow{\theta^{k}} i_{*}\omega^{(1,-k),\sm}[-1]\cong \mcal{N}_{k}^{(1)}. 
                        % \xhookrightarrow{\mrm{Cou}}\mcal{N}_{k}^{(2)}.
                        \]
                
                Since \(\mcal{N}^{(1)}_{0}\twoheadrightarrow 
                        \mcal{N}_{0}^{(2)}\)
                        is perverse surjective, the lower triangle is forced to be commutative, which is precisely the statement of Theorem \ref{thmFontaine=Theta} (2). 

                The inverse direction is similar, where we use (\ref{equaBasicRelk>0andk<0}) and the fact that \(\mrm{Cou}\) is perverse injective.
                % we have a unique factorization of \(N^{k,(2)}:\mcal{N}_{0}^{(2)}\to\mcal{N}_{k}^{(2)}\) as 
                % \[ 
                %         \mcal{N}_{0}^{(2)}\cong i_{*}\omega^{(1-k,0),\sm}[-1]\xrightarrow{\theta^{k}} i_{*}\omega^{(1,-k),\sm}[-1]\cong \mcal{N}_{k}^{(1)} \xhookrightarrow{\mrm{Cou}}\mcal{N}_{k}^{(2)}.
                %         \]
                % since 
                % we know that the upper triangle is forced to be commutative, i.e.
                % \[N^{k,(1)}\cong 
                % \]               
                % Then the induced map          is forced to be the theta map, by the case $(k-1,0)$ proven in Subsection \ref{subsectionGeneral(a,b)a-bge0}.
\end{proof}

\section{Proof of the main theorem}\label{sectionPfofMainThmdeRhamSheaves}

We will finish the proof of Theorem \ref{thmFontaine=Theta}. By Remark \ref{rmkTwistInBdR} and Corollary \ref{cork<0FollowsFromk>0}, it suffices to consider \((a,b)=(k-1,0)\) for \(k\in \ZZ_{\ge1}\). We remark that a simpler proof will appear in \cite{Jiang2025HMF}. We will use the natation of de Rham period sheaves from Definition \ref{dfnShvBdRonFl}.

\subsection{Strategy of the proof}\label{subsectionStrategy}
% In this section, we give the proof of Theorem \ref{thmFontaine=Theta}. 
The proof is a bit technical.
We will start with a sketch of the proof when $(a,b)=(0,0)$ to illustrate the idea. The plan of this section will be given at the end of this subsection.

When $(a,b)=(0,0)$, the theta operator $\theta^{1}$ is the (log) connection $\nabla_{\mrm{GM}}:i_{*}\mO^{\sm}\to i_{*}\Omega_{\log}^{1,\sm}$. For the Fontaine operator, we are looking at the action of \(\Theta\) on  $$0\to \mcal{N}_{1}\to \hat{\mE}_{(0,0)}(\BdR^{+,\la})\to \mcal{N}_{0}\to 0,$$
and we want to understand the complex $[\mcal{N}_{0}\xrightarrow{N^{1}}\mcal{N}_{1}]$. Note that these maps are not expected to be $\mO^{\sm}$-linear. As in \cite{Pan2209.06II}, 
we linearize the situation by resolving $\BdR^{+}$ by $\OBdR^{+}$. We consider the resolution in \cite{Scholze13}                \[0\to
                \BdR^{+,\la}\to \OBdR^{+,\la}\to \OBdR^{+,\la}\otimes\Omega_{\log}^{1,\sm}\to 0 .
                \] 
The complex is compatible with  the ``$t$-adic filtration", which we denote as \[\Fil^{t}_{i}\OBdR^{+,\la}:=t^{-i}(\OBdR^{+,\la}).\]

We can take \(\hat{\mE}_{(0,0)}(-)\) of the sequence. We will see that that the action of \(\Theta\) on \(\hat{\mE}_{(0,0)}(\gr^{t}_{i}\OBdR^{+,\la})\) is zero (Lemma \ref{lemOlaDecompoEasy}). Therefore, \(\Theta\)-action on \(\hat{\mE}_{(0,0)}(\OBdR^{+,\la})\)
also induces a monodromy operator \[\ti{N}^{1}:\hat{\mE}_{(0,0)}(\gr^{t}_{0}\OBdR^{+,\la})\to \hat{\mE}_{(0,0)}(\gr^{t}_{-1}\OBdR^{+,\la}).
\]
As a result,
we have
%  the following diagram                 \[
% \begin{tikzcd}
%  \mcal{N}_{0}=\hat{\mE}_{(0,0)}(\mO^{\la})\arrow[r]
%  & \hat{\mE}_{(0,0)}(\OBdR^{+,\la}/(t+\Fil^{2})) \arrow[r,"\nabla"]
%  & \hat{\mE}_{(0,0)}(\mO^{\la})\otimes\Omega_{\log}^{1,\sm}\\
%  \hat{\mE}_{(0,0)}(\BdR^{+,\la}/t^{2})\arrow[r] \arrow[u]
%  & \hat{\mE}_{(0,0)}(\OBdRR{2}) \arrow[r,"\nabla"]\arrow[u]
%  & \hat{\mE}_{(0,0)}(\mO^{\la})\otimes\Omega_{\log}^{1,\sm} \arrow[u]\\
%  \mcal{N}_{1}=\hat{\mE}_{(0,0)}(\mO^{\la}(1))\arrow[u]\arrow[r] &
%  \mcal{E}_{(0,0)}(\mO^{\la}(1))\arrow[u]\arrow[r,"\nabla"] & 0\arrow[u],
%  \end{tikzcd} 
%                 \] where all rows and columns are short exact sequences of perverse sheaves, and $\Theta$ acts as zero on the first row and the third row.
% In this way, the action of $\Theta$ induces the Fontaine operators mapping from the first row to the third row.
                %  So 
 the following commutative diagram:                 \[
\begin{tikzcd}
 \mcal{N}_{0}
%  =\hat{\mE}_{(0,0)}(\mO^{\la})
 \arrow[r,"\iota"]
 \arrow[d,"N^{1}"]
 & \hat{\mE}_{(0,0)}(\gr^{t}_{0}(\OBdR^{+,\la})) \arrow[r,"\nabla"]\arrow[d,"\ti{N}^{1}"]
 & \hat{\mE}_{(0,0)}(\gr^{t}_{0}(\OBdR^{+,\la}))\otimes\Omega_{\log}^{1,\sm}\arrow[d,"\ti{N}^{1}\otimes 1"]\\
 \mcal{N}_{1}
%  =\hat{\mE}_{(0,0)}(\mO^{\la}(1))
 \arrow[r,"\iota"] &
 \hat{\mE}_{(0,0)}(\gr^{t}_{-1}(\OBdR^{+,\la}))\arrow[r] & \hat{\mE}_{(0,0)}(\gr^{t}_{-1}(\OBdR^{+,\la}))\otimes\Omega^{1,\sm}_{\log},
 \end{tikzcd}                            \] 
where each row is a short exact sequence of perverse sheaves. 
% To avoid confusion of notations, let us write \(\)

Moreover, we can endow \(\gr^{t}_{i}(\OBdR^{+,\la})\) with the Hodge filtration induced from \(\OBdR^{+,\la}\), and then the diagram is compatible with the filtrations.
% , and $\ti{N}^{1}_{k}$ ($k=0,1$) is defined similarly as $\ti{N}^{1}$. 
% So we are reduced to study $\ti{N}^{1}_{1}$ and $\nabla$.

By Theorem \ref{mainthmPanPilloni}, the objects on the second row are concentrated in degree \(1\) for the natural \(t\)-structure, while those on the first row are in degree \([0,1]\). So it suffices to understand 
\[
\begin{tikzcd}
        & 
        \Fib(H^{1}(\ti{N}^{1})) \arrow[r,"\nabla"]\arrow[d]
        & \Fib(H^{1}(\ti{N}^{1}))\otimes \Omega^{1,\sm}_{\log}\arrow[d]
\\
 H^{1}(\mcal{N}_{0})
%  \cong i_{*}\mO^{\sm}
 \arrow[r,"\iota"]
 \arrow[d,"H^{1}(N^{1})"]
 & \hat{\mE}_{(0,0)}(\gr^{t}_{0}(\OBdR^{+,\la})) \arrow[r,"\nabla"]\arrow[d,"H^{1}(\ti{N}^{1})"]
 & \hat{\mE}_{(0,0)}(\gr^{t}_{0}(\OBdR^{+,\la}))\otimes\Omega_{\log}^{1,\sm}\arrow[d,"H^{1}(\ti{N}^{1})\otimes 1"]\\
 H^{1}(\mcal{N}_{1})
%  \cong i_{*}\Omega^{1,\sm}_{\log}
 \arrow[r,"\iota"] &
 \hat{\mE}_{(0,0)}(\gr^{t}_{-1}(\OBdR^{+,\la}))\arrow[r] & \hat{\mE}_{(0,0)}(\gr^{t}_{-1}(\OBdR^{+,\la}))\otimes\Omega^{1,\sm}_{\log}.
 \end{tikzcd}                    \] 

% the diagram is explicitly isomorphic to 
% \[
% \begin{tikzcd}
%  {\left[\mO^{\sm}-i_{*}\mO^{\sm}[-1]\right]}\arrow[r,"\iota"]\arrow[d,"N^{1}"]
%  & \hat{\mE}_{(0,0)}(\OBdR^{+,\la}/(t+\Fil^{2})) \arrow[r,"\nabla"]\arrow[d,"N^{1}_{1}"]
%  & {\left[\mO^{\sm}-i_{*}\mO^{\sm}[-1]\right]}\otimes\Omega_{\log}^{1,\sm}\arrow[d,"N^{1}_{0}\otimes 1"]\\
%  {i_{*}\Omega^{1,\sm}_{\log}[-1]}
%  \arrow[r,"\iota","\sim"'] & {i_{*}\Omega^{1,\sm}_{\log}[-1]}
% \arrow[r] & 0,
%  \end{tikzcd}            
% \] where as usual, $\left[\mO^{\sm}-i_{*}\mO^{\sm}[-1]\right]$ refers to an extension with $\mO^{\sm}$ being the subobject, and with $i_{*}\mO^{\sm}[-1]$ being the quotient. 
% We note that $\mcal{N}_{1}$ is concentrated in degree $1$ with respect to the natural t-structure, so $N^{1}_{k}$ is determined by $H^{1}(N^{1}_{k})$. 

We want to understand the filtered morphism \[H^{1}(\ti{N}^{1}):\hat{\mE}_{(0,0)}(\gr^{t}_{0}(\OBdR^{+,\la}))\to \hat{\mE}_{(0,0)}(\gr^{t}_{-1}(\OBdR^{+,\la})),
\] so we can start by considering its graded pieces, which (up to twists) are symmetric powers of Faltings's extensions. These actually lie in the image of the functor \(\VB\), so we can calculate it explicitly over the flag variety. As it turns out, \(H^{1}(\ti{N}^{1})\) induces an isomorphism on \(\gr^{i}\) for any \(i\ne 0\) (Proposition \ref{propFaltingsExt}), and \(\Fib(H^{1}(\ti{N}^{1}))\cong i_{*}\mO^{\sm}\), equipped with the trivial filtration. 

Moreover, the connection \(\nabla\) on \(\OBdR\)
induces a connection \(\nabla:\mrm{Fib}(H^{1}(\ti{N}^{1}))\cong i_{*}\mO^{\sm}\to i_{*}\Omega^{1,\sm}\). We claim that it has to coincide with \(\nabla_{\mrm{GM}}\). This is because 
$\nabla-\nabla_{\mrm{GM}}$ is a $B(\QQ_{p})$-equivariant $\mO^{\sm}$-linear morphism $i_{*}\mO^{\sm}\to i_{*}\Omega^{1,\sm}_{\log}$, which corresponds to a $B(\QQ_{p})$-invariant weight 2 overconvergent modular forms. Any such form 
has to be zero by Lemma \ref{lemNoBinv}.
This shows that $\nabla=\nabla_{\mrm{GM}}$! 

Using the diagram above, we have an isomorphism of \emph{filtered} complex \[[i_{*}\mO^{\sm}\xrightarrow{H^{1}(N^{1})}i_{*}\Omega^{1,\sm}_{\log}]\cong [i_{*}\mO^{\sm}\xrightarrow{\nabla_{\mrm{GM}}}i_{*}\Omega^{1,\sm}_{\log}],
\] where both sides are equipped with the filtration from the stupid truncation. This implies that \(H^{1}(N^{1})\cong \nabla_{\mrm{GM}}\). 

Most of the argument above generalizes to \((a,b)=(k-1,0)\) for general \(k\in\ZZ_{\ge 1}\).
In Subsection \ref{subsecTadicFiltration}, we will introduce the \(t\)-adic filtration \(\Fil^{t}_{*}\). The rest of the section mainly focus on proving Proposition \ref{propNabla=GM}, that is, \[H^{1}(\Fib(\ti{N}^{k}))\cong i_{*}\Sym^{k-1}D^{\sm}.
\]

In Subsection \ref{subsectionComputedeRham}, we will compute the derived locally algebraic vectors in \(\OBdR\), where the main result is Proposition \ref{propH0OBdR}, which states that \[H^{0}(\Fib(\ti{N}^{k})/\Fil^{k})\cong \Sym^{k-1}D^{\sm}.
\] 

In Subsection \ref{subsecFaltingsExt}, we will compute the Sen operator for Faltings's extension, which shows that \[H^{1}(\Fib(\ti{N}^{k}))\cong H^{1}(\Fib(\ti{N}^{k})/\Fil^{k}).
\]
% The idea is to use the functor $\VB$ (defined in \cite{Pilloni22}) to reduce to computations on the flag varieties. As we will see, the essential computation on the flag variety has been done in \cite{Pilloni22}. 

% These two subsections will provide a description of the kernel \(H^{1}(\ti{N}^{k})\). 
Finally, we finish the proof of Theorem \ref{thmFontaine=Theta} in Subsection \ref{subsectionGeneral(a,b)a-bge0}. The most non-trivial part of the proof is Proposition \ref{propNabla=GM}. The key is to consider the action of \(\hat{\mE}_{(0,0)}(\OBdR^{\la})\)
on \(\hat{\mE}_{(k-1,0)}(\OBdR^{\la})\), which induces a map \[
H^{1}(\Fib(\ti{N}^{1}))\otimes H^{0}(\Fil(\ti{N}^{k})/\Fil^{k})\to H^{1}(\Fil(\ti{N}^{k})/\Fil^{k}),
\] which one can verify to be an isomorphism.

\begin{rmk}[Comparison with \cite{Pan2209.06II}]\label{rmkComparisonPan}
For people familiar with \cite{Pan2209.06II}, 
our proof is parallel to that of \cite{Pan2209.06II}. In fact, given Theorem 1.2.10 of \cite{Pan2209.06II}, we obtain Theorem \ref{thmFontaine=ThetaIntro} by applying $\mE_{(k-1,0)}$, with $\mE_{(k-1,0)}(d^{k})=\theta^{k}$ and $\mE_{(k-1,0)}(\bar{d}^{k})$ becomes the natural projection $\mcal{N}_{0}\to M^{\dagger}_{1-k}$.  The key simplification in our proof is that after taking $\fb$-cohomology, $\mE_{(k-1,0)}(\bar{d}^{k})$ becomes much simpler, and induces an isomorphism in \(H^{1}\). In \cite{Pan2209.06II}, Pan achieves the same goal using a non-canonical construction of a lifting $\mO^{\la,(0,0)}\to \OBdR^{\la}$. Taking $\fb$-cohomology also simplifies the computation of cohomology, as well as the proof of the classicality. Our proof generalizes naturally to overconvergent Hilbert modular forms.
\end{rmk}

\subsection{\(t\)-adic filtration}\label{subsecTadicFiltration}
For later usage, we introduce the following $t$-adic (ascending) filtration:
% \begin{construction}
% Let \(M\) be an abelian group (concentrated in degree 0). Assume that \(M\) are equipped with an ascending filtrations  \(\Fil^{*}\), and a descending filtration \(\Fil^{t}_{*}\). Then 
% \end{construction}
\begin{dfn}[\(t\)-adic filtration]\label{dfntadicFil}
We will write \(\Fil^{*}\) for the descending Hodge filtration on \(\OBdR\) or the filtrations induced on \(\OBdR^{+}\) or \(\OBdRR{\ell}\) etc.

We define another ascending $t$-adic filtration on \(\OBdR\) via \[ \Fil_{-i}^{t}(\OBdR):=t^{i}\OBdR^{+}. \]
% for $0\le i\le \ell+1$.  
We denote the corresponding graded pieces as         \( \gr^{t}_{*}\OBdR. \)
We also use the notation $\Fil^{t}_{*}$ and $\gr^{t}_{*}$ for the induced filtrations on \(\OBdR^{\la},\;\OBdR^{+,\la},\;\OBdRR{\ell},\;\gr^{i}(\OBdR^{+})\) or
$\mE_{(a,b)}(\OBdRR{\ell+1})$ for $(a,b)\in\ZZ^{\oplus 2}$ etc. 

% We denote as $\Fil^{j}\OBdR^{+,\la}$ its standard (descending) Hodge filtration, and denote as $\gr^{j}\OBdR^{+,\la}$ its associated graded pieces. Both $\Fil^{j}\OBdR^{+,\la}$ and $\gr^{j}\OBdR^{=,\la}$ are endowed with an induced $t$-adic filtration, which we still denote as $\gr^{t}_{i}$. 

Concretely, for \(i\in\ZZ\), we equip \(\Fil^{t}_{i}(\OBdR)\)
with the Hodge filtration and equip \(\Fil^{n}(\OBdR)\) with the induced filtration by putting \[\Fil^{t}_{i}\Fil^{n}(\OBdR):=\Fil^{n}\Fil^{t}_{i}(\OBdR):=\Fil^{n}(\OBdR)\cap \Fil^{t}_{i}(\OBdR),\]
equip \(\gr^{t}_{i}(\OBdR)\)
with the Hodge filtration by putting \[\Fil^{n}\gr^{t}_{i}(\OBdR):=\Fil^{n}\Fil^{t}_{i}(\OBdR)/\Fil^{n}\Fil^{t}_{i-1}(\OBdR),\]
and equip \(\gr^{n}(\OBdR)\) with the induced filtration \[\Fil^{t}_{i}\gr^{n}(\OBdR):=\Fil^{t}_{i}\Fil^{n}(\OBdR)/\Fil^{t}_{i}\Fil^{n+1}(\OBdR).
\]
\end{dfn}
\begin{eg}
By definition $\gr^{t}_{0}(\OBdR^{+,\la})=\OBdR^{+,\la}/(t)$, and has a descending filtration induced by the Hodge filtration, where the non-trivial graded pieces are $\gr^{t}_{0}\gr^{j}\OBdR^{+,\la}\cong (\Omega^{1,\la})^{\otimes j}$ for $j\in \NN$.
\end{eg}

The \(t\)-adic graded pieces have the property that the action of \(\Theta\) is semisimple after taking \(\fb\)-cohomology in the case of regular weights. 

\begin{lem} \label{lemOlaDecompoEasy}
Let \(i\in \ZZ\).
The Hodge filtration of \(\OBdR\) induces a filtration \(Fil^{*}\) on \(\gr^{t}_{i}(\OBdR^{\la})\) where the non-trivial graded pieces are \[\gr^{j}(\gr^{t}_{i}(\OBdR^{\la}))\cong (\Omega^{1,\la}_{\log})^{\otimes_{\mO^{\la}}(j-i)}(-i)
\] for \(j\in \ZZ_{\ge i}\). 
% \(\)
% For $\ell\le i\le 0$,
% The $t$-graded piece $\gr^{t}_{i}(\OBdRR{\ell+1})$
% admits a filtration by 
% $(\Omega_{\log}^{1,\la})^{\otimes j}(-i)$ for $0\le j\le \ell+i$.
% More precisely,         \[
%         \gr_{i}^{t}(\OBdRR{\ell+1})\cong
%         \left[
%         (\Omega^{1,\la}_{\log})^{\otimes (\ell+i)}(-i)-\cdots-\Omega^{1,\la}_{\log}(-i)-\mO^{\la}(-i)
%         \right] .
%         \]
For \((a,b)\in\ZZ^{\oplus 2},a\ne b+1
\), the action of $\Theta$ on $\hat{\mE}_{(a,b)}(\gr_{i}^{t}\OBdR^{\la})$ is zero. 

Moreover, for \(i\notin\{-b,-1-a\}\), 
 \(\gr_{i}^{t}(\hat{\mE}_{(a,b)}(\OBdR^{\la})):=\hat{\mE}_{(a,b)}(\gr_{i}^{t}\OBdR^{\la})\cong 0\), and for \(i\in \{-b,-1-a\}\), 
\(\gr_{i}^{t}(\hat{\mE}_{(a,b)}(\OBdR^{\la}))\) has a descending filtration, where the non-trivial graded pieces are \[\gr^{j}(\gr^{t}_{i}(\hat{\mE}_{(a,b)}(\OBdR^{\la})))\cong \mcal{N}_{-i}\otimes_{\mO^{\sm}}(\Omega^{1,\sm}_{\log})^{\otimes_{\mO^{\sm}}(j-i)},\;j\in\ZZ_{\ge i}.
\] where \(\mcal{N}_{-i}\) are as in Theorem \ref{mainthmPanPilloni}.
% In particular, 
% for $k,\ell\in\NN$, the graded pieces
%  $\gr^{t}_{i}\hat{\mE}_{(k-1,0)}(\OBdR^{\la})\cong 0$  unless $i=0$ or $i=-k$. 
\end{lem}
% \begin{rmk}
% Below, we will prove that the action of \(\Theta\) on \(\hat{\mE}_{(a,b)}(\gr^{t}_{i}(\OBdR^{\la}/\Fil^{n}))\) is \(0\). Thanks to \ref{lemPerversity}, we can work on the level of abelian categories, so 
% \end{rmk}

\begin{proof}
Multiplying by $t$ induces an isomorphism         \[ \gr^{t}_{i+1}(\OBdRR{\ell+1})\cong \gr^{t}_{i}(\OBdRR{\ell+2})(-i). \] In particular, it suffices to prove the claim for $i=0$. We proceed by induction on $\ell$. The case when $\ell=0$ is obvious. Given the case for $\ell\in\NN$, we consider the short exact sequence         \( 0\to \BdR^{+,\la}\to \OBdRR{\ell+1}\to \OBdRR{\ell}\otimes_{\mO^{\sm}}\Omega_{\log}^{1,\sm}\to 0 \). Taking $\gr^{t}_{0}$ gives         \( 0\to \mO^{\la}\to \gr^{t}_{0}(\OBdRR{\ell+1})\to \gr^{t}_{0}(\OBdRR{\ell})\otimes_{\mO^{\sm}}\Omega_{\log}^{1,\sm}\to 0. \) On the other hand, the projection $\OBdRR{\ell+1}\twoheadrightarrow \OBdRR{1}\cong \mO^{\la}$ induces a natural splitting of the sequence above, and 
thus we have an \(\mO^{\la}\)-linear isomorphism  \[
\gr_{0}^{t}(\OBdRR{\ell+1})\cong \mO^{\la}\oplus (\gr_{0}^{t}(\OBdRR{\ell})\otimes_{\mO^{\sm}}\Omega^{1,\sm}_{\log})      .
        \]  We are done by induction. Note that the action of $\Theta$ on $\Omega^{1,\sm}_{\log}$ is zero.
% For the last claim, we have by definition \( \gr^{t}_{i}(\mE_{(0,0)}(\OBdRR{\ell+1}))=\mE_{(k-1,0)}(\gr^{t}_{i}(\OBdRR{\ell+1})). \)
%  By 
The last part follows from
Theorem \ref{mainthmPanPilloni} as
\( \mE_{(a,b)}((\Omega_{\log}^{1,\la})^{\otimes j}(-i))\cong (\Omega_{\log}^{1,\sm})^{\otimes j}\otimes_{\mO^{\sm}}\mE_{(a,b)}(\mO^{\la}(-i))\cong 0 \).
% unless $i=-k$ or $i=0$. 
\end{proof}
\begin{rmk}\label{rmkOsmLinearDistinguish}
We remark that there is a bit of subtlety in the term ``$\mO^{\la}$-linear" here. On $\gr^{t}_{\bullet}(\OBdR^{+,\la}/\Fil^{\ell+1})$, there are the action of $\mO^{\sm}$ coming from the embedding $\mO^{\sm}\hookrightarrow \OBdR$, and that of $\mO^{\la}=\BdR^{+,\la}/t$. In particular, we have two \emph{different} actions of $\mO^{\sm}$, although they coincide on the graded pieces. In what follows, when we say ``$\mO^{\sm}$-linear",
we refer to the former from the embedding $\mO^{\sm}\hookrightarrow\OBdR$, and when we say ``$\mO^{\la}$-linear", we refer to the latter. 
% In particular, the $\mO^{\la}$-module structure of $\gr_{i}^{t}(\OBdRR{\ell+1})$ has the simple description as given in Lemma \ref{lemOlaDecompoEasy}. However, $\mO^{\sm}$-structure is more complicated and interested, and is the content of Proposition \ref{propH0OBdR}.

On the other hand, the action of $\mO^{\sm}$ and $\mO^{\la}$ are compatible on $\gr^{j}\OBdR^{+,\la}$, essentially because they coincide in $\gr^{0}\OBdR\cong\hat{\mO}$. As a result, their actions are also compatible on     \( \gr^{t}_{i}\gr^{j}\OBdR^{+,\la}. \)
\end{rmk}

\subsection{Locally algebraic vectors in \(\OBdR\)}\label{subsectionComputedeRham}
Using Theorem \ref{thmMainThmGeomSenVB}, it is easy to see that \(\hat{\mO}^{R-\sm}\cong \mO^{\sm}\oplus \Omega^{1,\sm}_{\log}[-1]\). Interestingly, \(\OBdR^{R-\sm}\) is concentrated in degree \(0\). The result of this section generalizes easily to general Shimura varieties.

\begin{notation}[{\cite[Definition 2.2.10 (3)]{DLLZ2022logarithmicJAMS}}]
Let  \(\OC:=\gr^{0}\OBdR\), and denote by \(\OC^{\la}\) its subsheaf of locally analytic vectors. Note that \(\OC^{\la}\cong \OC^{R-\la}\) by Proposition \ref{propORlaConcentrat}. 
\end{notation}

\begin{prop}\label{propOCRsm}
We have a canonical isomorphism \[\OC^{R-\sm}\cong \RHom_{\fg}(1,\OC^{\la})\cong \mO^{\sm}.
\]       
\end{prop} 
\begin{rmk}
Although in a different language and through different proof, this proposition is essentially a reformulation of \cite[Proposition 6.16]{Scholze12} or \cite[Lemma 3.3.2]{DLLZ2022logarithmicJAMS}.
\end{rmk}
\begin{proof}
By \cite[Theorem 4.2.2]{Pan22} (or \cite[Theorem 5.1.4]{Juan2022.09locallyShi} for the general Shimura varieties), we know that \(\OC\cong \pi_{\HT,*}\circ \pi_{\HT}^{*}(\underline{O(N)})\), where \(\underline{O(N)}\) is a colimit of algebraic \(B=TN\)-representations, where \(N\) acts by right multiplication on \(N\), and \(T\) acts by conjugation, and \(\underline{O(N)}\) is the colimit of the associated \(\GL_{2}\)-equivariant vector bundles via localization. 
Then \[\OC^{\la}\cong \OC^{R-\la}\cong (\pi_{\HT,*}\circ \pi_{\HT}^{*}(\underline{O(N)})\hatotimes \mO_{G,1})^{R-\sm}\cong \VBn(\underline{O(N)}\hatotimes \mO_{G,1}),
\] where \(\VBn\) is as in Definition \ref{dfnVBNaive}. Here \(\underline{O(N)}\) is a filtered colimit of vector bundles, so \(\pi_{\HT}^{*}=L\pi_{\HT}^{*}\), and \(R\pi_{\HT,*}=\pi_{\HT,*}\) by Lemma \ref{lemPushForwardVanish}. Now \[\OC^{R-\sm}\cong \VBn(\RHom_{\fg}(1,(\underline{O(N)}\hatotimes \mO_{G,1},*_{2})))\cong \VBn(\underline{O(N)}).
\] where \(*_{2}\) denotes the action of \(\fg\) on \(\mO_{G,1}\) by right multipication, and thus \[\RHom_{\fg}(1,\mO_{G,1})\cong \CC_{p}.\] Now by Lemma \ref{lemVBnToVB} and Theorem \ref{thmMainThmGeomSenVB},
we have \[\VBn(\underline{O(N)})\cong \VB(R\Gamma(\fn^{0},\underline{O(N)}))\cong \VB(\mO_{\Fl})\cong \mO^{\sm},
\] as desired.
\end{proof}
\begin{cor}\label{corOBdRRsm}
We have a canonical isomorphism \[\hat{E}_{0}(\RHom_{\fg}(1,\OBdR^{\la}))\cong \mO^{\sm},
\]   where \(\hat{E}_{0}\) is as in Definition \ref{dfnBdRSemiLinear}.     
\end{cor} 
\begin{proof}
By Proposition \ref{propOCRsm},
the Hodge filtration on \(\OBdR^{\la}\)
induces a filtration on \(\RHom_{\fg}(1,\OBdR^{\la})\), such that \(\gr^{i}\cong \mO^{\sm}(i)\),
and thus \(E_{0}(\gr^{i})\cong 0\) unless \(i=0\), as desired.
\end{proof}
\begin{cor}\label{corOBdRLAlg}
Let \(W\) be an algeraic representation of \(G\), and let \(D_{\dR,\Kpp}(W)\) be the associated 
vector bundle with an integrable connection defined over \(\mX_{\Kpp}\). 
Denote \(D_{\dR,K^{p}}(W)^{\sm}:=\pi_{\HT,*}(\varinjlim_{K_{p}}\pi_{K_{p}^{-1}}D_{\dR,\Kpp}(W))\).
Then we have a canonical filtered isomorphism \[
\hat{E}_{0}(\RHom_{\fg}(W^{\vee},\OBdR^{\la}))\cong 
 D_{\dR,K^{p}}(W)^{\sm}
\] that is compatible with connections.
\end{cor}
\begin{proof}
By the (logarithm) Riemann-Hilbert correspondence (\cite[Theorem 1.5]{DLLZ2022logarithmicJAMS}), we have a natural filtered isomorphism that is compatible with the Gauss-Manin connections  \begin{align}\label{alignRHCor}
        D_{\dR,K^{p}}(W)^{\sm}\otimes_{\mO^{\sm}}\OBdR\cong W\otimes_{\QQ_{p}}\OBdR ,
\end{align} and thus \[D_{\dR,K^{p}}(W)^{\sm}\otimes_{\mO^{\sm}}\OBdR^{\la}\cong W\otimes_{\QQ_{p}}\OBdR^{\la}.
\] We can take \(\hat{E}_{0}(\RHom_{\fg}(1,-))\) on both sides. The LHS becomes \(D_{\dR,K^{p}}(W)^{\sm}\) by Corollary \ref{corOBdRRsm}, and the RHS becomes \(\RHom_{\fg}(W^{\vee},\OBdR^{\la})
\), as desired. 
\end{proof}
\begin{eg}
In the case of modular curves, for \(W=\Sym^{k-1}V^{\vee}\), \(D_{\dR,K^{p}}(W)^{\sm}=\Sym^{k-1}D^{\sm}\). 
\end{eg}
\begin{cor}\label{propH0OBdR}
Let \(k\in\ZZ_{\ge 1}\), and consider \[\RHom_{\fg}(\Sym^{k-1}V,\OBdR^{\la})\to \RHom_{\fb}((k-1,0),\OBdR^{\la})
\] as is induced by (\ref{alignBGGComp}). Then it induces a filtered morphism \[\Sym^{k-1}D^{\sm}\cong \hat{E}_{0}(\RHom_{\fg}(\Sym^{k-1}V,\OBdR^{\la}))\to \hat{\mE}_{(k-1,0)}
(\OBdR^{+,\la})
\] that is compatible with connections, and induces an isomorphism \[\Sym^{k-1}D^{\sm}\cong H^{0}(\hat{\mE}_{(k-1,0)}(\gr_{0}^{t}(\OBdR/\Fil^{k}))).
\]
\end{cor}
\begin{proof}
Only the last isomorphism requires some explication. By (\ref{alignBGGComp}), we have a fiber sequence of filtered sheaves \[\hat{E}_{0}(
\RHom_{\fg}(\Sym^{k-1}V,\OBdR^{\la}))\to \hat{\mE}_{(k-1,0)}(\OBdR^{\la})\to \hat{\mE}_{(-1,k)}(\OBdR^{\la}).
\] 
By Lemma \ref{lemOlaDecompoEasy}, \((\gr^{i}(\hat{\mE}_{(-1,k)}(\OBdR^{\la})))\cong 0\) is concentrated in cohomological degree \(1\) if \(0\le i <k\),
and \(\hat{\mE}_{(k-1,0)}(\OBdR^{\la})\cong \hat{\mE}_{(k-1,0)}(\OBdR^{+,\la})\).

On the other hand, by Corollary \ref{corOBdRLAlg}, \(\hat{E}_{0}(
\RHom_{\fg}(\Sym^{k-1}V,\OBdR^{\la}))\cong \Sym^{k-1}D^{\sm}\)
is concentrated in degree \(0\), and \(\gr^{i}\cong 0\) unless \(0\le i<k\). Thus we have an isomorphism \(\Sym^{k-1}D^{\sm}\cong H^{0}(\hat{\mE}_{(k-1,0)}(\OBdR^{+,\la}/\Fil^{k}))\).
\end{proof}

\subsection{Computation of Faltings's extension}\label{subsecFaltingsExt}
In this section, we compute the Sen operator for Faltings's extension in terms of Theorem \ref{mainthmPanPilloni}.  

Using Kodaira-Spencer isomorphism, we will identify $\Omega^{1,\sm}_{\log}\cong\omega^{(1,-1),\sm}$. By Definition \ref{dfntadicFil}, for \(m\in\NN\), \(\gr^{m}(\OBdR^{+,\la})\) has an ascending filtration \(\Fil^{t}_{*}\)
such that the non-trivial graded pieces are \[\gr^{t}_{i}(\gr^{m}\OBdR^{+,\la})\cong  (\Omega^{1,\sm}_{\log})^{\otimes (m+i)}\otimes_{\mO^{\la}}\hat{\mO}(-i)\cong \omega^{(m+i,-m-i),\la}(-i)
\] for \(-m\le i \le 0\).

% \begin{rmk}
% As the action of $\Theta$ on $\OBdR^{+,\la}/\Fil^{k+1}$ is $\mO^{\sm}$-linear, it is actually possible that $\OBdR^{+,\la}/\Fil^{k}$ comes from the flag variety. If this is true, it will essentially simplify the proof in the Hiblert case. It surely does not come from $\OBdR^{+}$ of the flag variety, because $\VB(\Omega_{\log}_{\Fl}^{1})$ is not $\Omega_{\log}_{\mX_{K^{p}}}^{1,\sm}$, but is its dual. It could be a twisted version of     \( \underline{\mO(N)}. \)
% Recall it is proven in \cite{Juan2022.09locallyShi} that      \( \pi_{\HT}^{-1}(\underline{\mO(N)})\cong O\mathbb{C}=\gr^{0}\OBdR \) and thus  \( \VB(\underline{\mO(N)}^{\le k}\otimes_{\mO_{\Fl}}C^{\la})\cong (\gr^{0}\OBdR)^{\le k,\la}, \)
% which is closely related to     \( \OBdR^{+,\la}/\Fil^{k}. \)
% The possible twist can be gives as follows:     \( \underline{\mO(N)} \)
% comes from functions on $N^{0}\to \Fl$ (the tautogical $N$-torsor), and we twist the action of $\Gal_{\QQ_{p}}$ on $N^{0}$ via the composition of    \( \Gal_{\QQ_{p}}\xrightarrow{\log\chi} \ZZ_{p}\cong N(\ZZ_{p}) \)
% with the action of $N$ on $N^{0}$. The key difficulty is that it is hard to find a natural identification with $\OBdR^{+,\la}/\Fil^{k}$, and there is no way to find an identification before taking decompletion, because $\OBdR$ is \emph{not} a $\hat{\mO}$-algebra, contrary to the case of $O\mathbb{C}$.
% \end{rmk}

\begin{prop}\label{propFaltingsExt} Let     \( (a,b)\in\ZZ^{\oplus 2} \)
and $k:=|a-b+1|$. Let $m\in\ZZ_{\ge0}$, and write \(\ti{\mF}_{m+k}:=\gr^{m+k}\OBdR^{+,\la}(\min(b,1+a))\) equipped with the ascending filtration \(\Fil^{t}_{*}\). We will use the notation of Theorem \ref{mainthmPanPilloni}. 

(1) If $a+1< b$, then we have a short exact sequence in \(\mrm{Perv}\)       \[
     0\to \mE_{(a,b)}(\gr^{t}_{-k}(\ti{\mF}_{m+k}))\to  \mE_{(a,b)}(\ti{\mF}_{m+k})   \to \mE_{(a,b)}(\gr^{t}_{0}(\ti{\mF}_{m+k}))\to 0  ,
        \] and the action of $\Theta$ on \(\mcal{E}_{(a,b)}(\ti{\mF}_{m+k})\) is zero on the graded pieces, and is induced by the composition \begin{align*}
        \mE_{(a,b)}(\gr^{t}_{0}(\ti{\mF}_{m+k}))\cong (\Omega^{1,\sm}_{\log})^{\otimes(m+k)}\otimes \mcal{N}_{1+a}\cong i_{*}\omega^{(m-a,-m-b),\sm}[-1]\\
        \xrightarrow{\mrm{Cou}} j_{!}\omega^{(m-a,-m-b),\sm}
        \cong (\Omega^{1,\sm}_{\log})^{\otimes m}\otimes \mcal{N}_{b}
        \cong\mE_{(a,b)}(\gr^{t}_{-k}(\ti{\mF}_{m+k})).
        \end{align*} 

(2) If $a+1>b$, then we have a short exact sequence in \(\mrm{Perv}\)     \[
     0\to \mE_{(a,b)}(\gr^{t}_{-k}(\ti{\mF}_{m+k}))\to  \mE_{(a,b)}(\ti{\mF}_{m+k})   \to \mE_{(a,b)}(\gr^{t}_{0}(\ti{\mF}_{m+k}))\to 0  ,
        \] and the action of $\Theta$ on \(\mcal{E}_{(a,b)}(\ti{\mF}_{m+k})\) is zero on the graded pieces, and is induced by the composition \begin{align*}
        % \mE_{(a,b)}&(\omega^{(m+k,-m-k),\la}(b))\cong \omega^{(m+k,-m-k),\sm}\otimes [\omega^{(-a-b),\sm}-i_{*}\omega^{(-a,-b),\sm}[-1]]\\&\twoheadrightarrow \omega^{(m+k,-m-k),\sm}\otimes\omega^{(-a,-b),\sm}[-1]\cong i_{*}\omega^{(m+1-b,-m-1-a),\sm}[-1]
        % \cong\mE_{(a,b)}(\omega^{(m,-m),\la}(1+a)).
        \mE_{(a,b)}(\gr^{t}_{0}(\ti{\mF}_{m+k}))\cong (\Omega^{1,\sm}_{\log})^{\otimes(m+k)}\otimes \mcal{N}_{b}\twoheadrightarrow i_{*}\omega^{(m+1-b,-m-a-1),\sm}[-1]\\
        \cong (\Omega^{1,\sm}_{\log})^{\otimes m}\otimes \mcal{N}_{1+a}
        \cong\mE_{(a,b)}(\gr^{t}_{-k}(\ti{\mF}_{m+k})).
        \end{align*}

\end{prop}
\begin{rmk}
We are repeatedly using the following trick: \[
\mE_{(a,b)}(\omega^{\chi,\la})\cong \mE_{(a,b)}(\omega^{\chi,\sm}\otimes_{\mO^{\sm}}\mO^{\la})\cong \omega^{\chi,\sm}\otimes_{\mO^{\sm}}\mE_{(a,b)}(\mO^{\la}) .
\]
\end{rmk}
\begin{proof} We first explain how to reduce the problem to a calculation over the flag variety, and then the proof will be finished in Lemma \ref{lemFaltingsExtOnFl}.

The claim only depends on $k$ by Remark \ref{rmkTwistInBdR}. Therefore, in the case of (1), we consider $(a,b)=(k-1,0)$, and in the case of (2), we put $(a,b)=(-1,k)$ for $k\in\ZZ_{\ge1}$. In this way, we are working with the Faltings's extension $\gr^{*}\OBdR^{+,\la}$ without the Tate twist.

By Theorem 4.2.2 of \cite{Pan22}, $\gr^{1}\OBdR^{+}\cong (V\otimes \hat{\mO})\otimes_{\mO^{\sm}} \omega^{(0,-1),\sm}$, and the Faltings's extension coincides (up to a sign) with the $\omega^{(0,-1),\sm}$-twist of         \[
        0\to \omega^{(0,1)}(1)\to V\otimes\hat{\mO}\to \omega^{(1,0)}\to 0 .
        \] Note that our normalization is different from \cite{Pan22}, which changes the Hodge-Tate weight here.
The sequence above is the pull-back 
of the following sequence on $\Fl$:
        \[
        0\to \omega^{(0,1)}_{\Fl}\to V\otimes\mO_{\Fl}\to \omega^{(1,0)}_{\Fl}\to 0 .
        \]  
We define an ascending filtration \(\Fil^{t}_{*}\) on \(V\otimes\mO_{\Fl}\) such that \(\gr^{t}_{0}\cong \omega^{(1,0)}_{\Fl}\), and \(\gr^{t}_{-1}\cong \omega^{(0,1)}_{\Fl}\), which induces an ascending filtration on \(\Sym^{k}V\otimes \mO_{\Fl}\) such that \(\gr^{t}_{-i}\cong \omega_{\Fl}^{(k-i,i)}\).

For \(k\in\ZZ_{\ge0}\), $\gr^{k}\OBdR^{+,\la}\cong\Sym^{k}_{\mO^{\la}}\gr^{1}\OBdR^{+,\la}$, so \begin{align}\label{alignFaltingsIsTwist}
\gr^{k}\OBdR^{+,\la}\cong \omega^{(0,-k),\sm}\otimes_{\mO^{\sm}}\Sym^{k}V\otimes\mO^{\la},
\end{align} that is \(G(\QQ_{p})\)-equivariant and compatible with filtrations \(\Fil^{t}_{*}\) on both sides. 

% Now we consider the sheaf $\mO_{G,1}\hatotimes\mO_{\Fl}\hatotimes\Sym^{m+k}V$ over $\Fl$, where $\mO_{G,1}$ denotes the germ of locally analytic functions of $G$ around $1\in G$. Then we have four actions of $\fg$: we denote by $*_{1}$ and $*_{2}$ the action on $\mO_{G,1}$ induced by left and right multiplication, denote by $*_{3}$ the action on $\mO_{\Fl}$, and by $*_{4}$ the action on $\Sym^{m+k}V$. See Example \ref{egClaVBOla} for a more details treatment for the case when $k=0$. 

We consider the sheaf       \[
                \mF_{m+k}:=(\Sym^{m+k}V\otimes C^{\la},*_{1,3})\in\QCoh_{\fg}^{\rla}(\Fl)^{\fn^{0}}
                \] 
                where \(*_{1,3}\) is as in Notation \ref{egClaVBOla}. We will denote by $*_{4}$ the \(\fg\)-action on $\Sym^{m+k}V$.

                We define an ascending filtration on 
                \(\mF_{m+k}\cong (\Sym^{m+k}V\otimes\mO_{\Fl})\hatotimes_{\mO_{\Fl}}C^{\la}\) as the one induced from \(\Sym^{m+k}V\otimes \mO_{\Fl}\). 
\begin{lem}\label{lemVBFmkGrOBdR}
% The sheaf \(\gr^{m+k}\OBdR^{+,\la}\) admits an arithmetic Sen operator (Definition \ref{dfnArithSenOp}). 
We have a filtered isomorphism \[\VB(\mF_{m+k})\cong \gr^{m+k}\OBdR^{+,\la}\otimes_{\mO^{\sm}}\omega^{(0,m+k),\sm},
\] the action of $\fg$ on the RHS is induced by $*_{2,4}$ on $\mF_{m+k}$, and the arithmetic Sen operator $\Theta$ on the RHS is induced by \(\theta_{\fh}(1,0)\).
\end{lem}
\begin{proof}[Proof of the lemma]
This follows from (\ref{alignFaltingsIsTwist}) and Theorem \ref{thmVBClaIsOla}.
\end{proof}
%  In this way,  by Theorem \ref{thmVBClaIsOla}.

By Lemma \ref{lemVBFmkGrOBdR}, Theorem \ref{thmMainThmGeomSenVB} and Lemma \ref{rmkDeduceSVfromFl}, the proof of Proposition \ref{propFaltingsExt} reduces to Lemma \ref{lemFaltingsExtOnFl} below.
\end{proof}

\begin{lem}\label{lemFaltingsExtOnFl}
        Let $k\in\ZZ_{\ge1}$ and $m\in\ZZ_{\ge0}$. In what follows, \(\mcal{E}_{(a,b)}\) is taken with respect to \(*_{2,4}\). 

(1) We have a distinguished triangle       \[
     \mE_{(-1,k)}(\gr^{t}_{-k}(\mF_{m+k}))
     \to  \mE_{(-1,k)}(\mF_{m+k})   \to \mE_{(-1,k)}(\gr^{t}_{0}(\mF_{m+k}))\xrightarrow{+1},
        \] and the action of $\theta_{\fh}(1,0)$ on $ \mE_{(-1,k)}(\mF_{m+k})$ is induced by the composition \begin{align*}
        \mE_{(-1,k)}(\gr^{t}_{0}(\mF_{m+k}))\cong i_{*}i^{\dagger}\omega^{(m+1,0)}_{\Fl}[-1]
        \xrightarrow{\mrm{Cou}} j_{!}j^{*}\omega^{(m+1,0)}_{\Fl}\cong \mE_{(-1,k)}(\gr^{t}_{-k}(\mF_{m+k})).
        \end{align*} 

(2) We have a distinguished triangle     \[
     \mE_{(k-1,0)}(\gr^{t}_{-k}(\mF_{m+k}))
     \to  \mE_{(k-1,0)}(\mF_{m+k})   \to \mE_{(k-1,0)}(\gr^{t}_{0}(\mF_{m+k}))\xrightarrow{+1},
        \] and the action of $\theta_{\fh}(1,0)$ on $ \mE_{(k-1,0)}(\mF_{m+k})$ is induced by the composition \begin{align*}
        \mE_{(k-1,0)}(\gr^{t}_{0}(\mF_{m+k}))&\cong \left[\omega^{(m+1,0)}_{\Fl}-i_{*}i^{\dagger}\omega^{(m+1,0)}_{\Fl}[-1]\right]\\&
        \xrightarrow{p} i_{*}i^{\dagger}\omega^{(m+1,0)}_{\Fl}[-1] \cong \mE_{(k-1,0)}(\gr^{t}_{-k}(\mF_{m+k})),
        \end{align*}  where the map $p$ is the natural projecting.
\end{lem}
The proof will be finished at the end of this subsection. 
The key idea is to reduce the problem to the case $m=k=0$ treated in \cite{Pilloni22}. 

\begin{lem}\label{lemReducektok=1}
For $k\in\ZZ_{\ge 1},m\in\ZZ_{\ge0}$, we have an isomorphism         \[
        \mE_{(-1,k)}(\mF_{m+k})\cong \mE_{(-1,0)}(C^{\la})\otimes_{\mO_{\Fl}}\omega^{(m,0)}_{\Fl}\cong\mE_{(k-1,0)}(\mF_{m+k}) ,
        \] which is compatible with respect to the action of $\theta_{\fh}(1,0)$.
\end{lem}
\begin{proof}  Note
\[\Sym^{m+k}V\cong \left[\chi^{(m+k,0)}-\chi^{(m+k-1,0)}-\cdots-\chi^{(0,m+k)}\right].\] So     \( \Ext_{(\fb,*_{2,4})}((-1,k),\Sym^{m+k}V\otimes C^{\la}) \)
is filtered by         \[
\Ext_{(\fb,*_{2,4})}((-1,k),\chi^{(i,m+k-i)}\otimes C^{\la})\cong \Ext_{(\fb,*_{2})}((-1-i,i-m),C^{\la}),\;0\le i\le m+k ,
        \] on which $\Theta:=\theta_{\fh}(1,0)$ acts by $i$ and $m-i$ by Theorem \ref{mainthmPilloniOnFl}.         
So when taking $E_{0}(-)$, 
only the graded piece with $i=0$ and $i=m$ will survive. We have a \(\fb\)-equivariant map $\Sym^{m+k}V\to \Sym^{m}V\otimes\chi^{(0,k)}$, and
we have \begin{align*}
% E_{0}(\mE_{(\fb,*_{2})}((-1,k),\mF_{m+k}))&\cong
        \mE_{(-1,k)}(\Sym^{m+k}V\otimes C^{\la})
        % \\
\cong\mE_{(-1,k)}(\Sym^{m}V\otimes\chi^{(0,k)}\otimes C^{\la})\\\cong  \mE_{(-1,0)}(\Sym^{m}V\otimes C^{\la}),
\end{align*} that is, \[
        \mE_{(-1,k)}(\mF_{m+k})\cong \mE_{(-1,0)}(\mF_{m}) \cong \mE_{(-1,0)}(C^{\la})\otimes\omega^{(m,0)}_{\Fl},
        \] where for the last isomorphism, we use the filtration \(\Fil^{t}_{*}\) on \(\Sym^{m}V\otimes C^{\la}\) with \(\gr^{t}_{-i}\cong C^{\la}\hatotimes \omega^{(m-i,i)}_{\Fl}\), and \(\mE_{(-1,0)}(\gr^{t}_{-i})\cong 0\) unless \(i=0\). 
The isomorphism is clearly compatible with $\Theta:=\theta_{\fh}(1,0)$-action.

The same argument also works for $(k-1,0)$ in place of $(-1,k)$.
%  We are now left to prove that $\mF_{(-1,0)}(\mF_{m})\cong\mF_{(-1,0)}(C^{\la})\otimes\omega^{(m,0)}_{\Fl}$. This follows by considering the 
% other filtration \[
%     \mF_{m}\cong\left[
% \omega^{(0,m)}_{\Fl}\hatotimes_{\mO_{\Fl}} C^{\la}-\omega^{(1,m-1)}_{\Fl}\hatotimes_{\mO_{\Fl}} C^{\la}-\cdots-\omega^{(m,0)}_{\Fl}\hatotimes_{\mO_{\Fl}} C^{\la}
% \right] .
% \] Note that $\theta_{\fh}(1,0)$ acts on $\Ext_{(-1,0)}(C^{\la})$ nilpotently, so only the last term survives after taking $\mE_{(-1,0)}(-)$.
\end{proof}

We have the following reformulation of the results of \cite{Pilloni22}, describing the action of
$\Theta$ on $\mE_{(-1,0)}(C^{\la})$.
\begin{lem}[{\cite[Proposition 6.3]{Pilloni22}}]
\label{lemCasek=1CoveredInPilloni}
We have an isomorphism         \[
        \mE_{(-1,0)}(C^{\la})\cong \left[
i_{*}i^{\dagger}\omega^{(1,0)}_{\Fl}[-1]-\omega^{(1,0)}_{\Fl}-i_{*}i^{\dagger}\omega^{(1,0)}_{\Fl}[-1]
    \right].
        \] Moreover, consider     \( \Theta:=\theta_{\fh}(1,0), \) 
    \( \Theta \)
is given by   the composition      \[
        \mE_{(-1,0)}(C^{\la})\twoheadrightarrow i_{*}i^{\dagger}\omega^{(1,0)}_{\Fl}[-1]\hookrightarrow \mE_{(-1,0)}(C^{\la}) ,
        \] where the first map is surjecting onto the top graded piece, and the second map is the inclusion of the lowest graded piece. In particular,
we have         \[
        \Ker(\Theta)\cong \left[
i_{*}i^{\dagger}\omega^{(1,0)}_{\Fl}[-1]-\omega^{(1,0)}_{\Fl}
    \right]\cong j_{!}j^{*}\omega^{(1,0)}_{\Fl},\;\Coker(\Theta)\cong \left[\omega^{(1,0)}_{\Fl}-i_{*}i^{\dagger}\omega^{(1,0)}_{\Fl}[-1]
    \right].
        \] Here $\Ker(\Theta)$ and $\Coker(\Theta)$ are taken in     \( {}^{p}B^{\heartsuit} \) defined in Lemma \ref{lemPerverseOnQCoh}.
\end{lem}
\begin{proof}
Theorem \ref{thmMainThmGeomSenVB} gives
        \[
        0\to j_{!}j^{*}\omega^{(1,0)}_{\Fl}\to \mE_{(-1,0)}(C^{\la})\to i_{*}i^{\dagger}\omega^{(1,0)}_{\Fl}\to 0 .
        \]
This gives the first isomorphism by noting $ j_{!}j^{*}\omega^{(1,0)}_{\Fl}\cong\left[
i_{*}i^{\dagger}\omega^{(1,0)}_{\Fl}[-1]-\omega^{(1,0)}_{\Fl}
    \right]$. Moreover, $\Theta$ is zero on the top graded pieces. In particular, we know $\Theta$ is induced by a canonical map     \( f:i_{*}i^{\dagger}\omega^{(1,0)}_{\Fl}[-1]\to j_{!}j^{*}\omega^{(1,0)}_{\Fl}. \) The claim is equivalent to saying that $f$ coincides with the Cousin map. We apply $\Gamma(\Theta,-)$ to the sequence, and take ${}^{p}H^{i}(-)$ gives a long exact sequence \begin{align*}
0&\to j_{!}j^{*}\omega^{(1,0)}_{\Fl} \xrightarrow{\iota} {}^{p}H^{0}(\Theta,\mE_{(-1,0)}(C^{\la}))\to i_{*}i^{\dagger}\omega^{(1,0)}_{\Fl}[-1]\\
&\xrightarrow{\delta} j_{!}j^{*}\omega^{(1,0)}_{\Fl} \xrightarrow{s} {}^{p}H^{1}(\Theta,\mE_{(-1,0)}(C^{\la})) \xrightarrow{p}  i_{*}i^{\dagger}\omega^{(1,0)}_{\Fl}[-1]\to 0.  
    \end{align*} Simple diagram chasing tells us that $f\cong \delta$. Now the long exact sequence given by natural t-structure (Remark \ref{rmkNaturalt-str}) $H^{*}(R\Gamma(\Theta,-))$ tells us that $H^{0}(\Theta,\mE_{(-1,0)}(C^{\la}))\cong j_{!}j^{*}\omega^{(1,0)}_{\Fl}$,
    and $H^{2}(\Theta,\mE_{(-1,0)}(C^{\la}))\cong i_{*}i^{\dagger}\omega^{(1,0)}_{\Fl}$. 
Moreover, \cite[Proposition 6.3]{Pilloni22} tells us that $H^{1}(\Theta,\mE_{(-1,0)}(C^{\la}))\cong \omega^{(1,0)}_{\Fl}$. Hence         \[
{}^{p}H^{0}(\Theta,\mE_{(-1,0)}(C^{\la}))\cong j_{!}j^{*}\omega^{(1,0)}_{\Fl},\;{}^{p}H^{1}(\Theta,\mE_{(-1,0)}(C^{\la}))\cong \left[\omega^{(1,0)}_{\Fl}-i_{*}i^{\dagger}\omega^{(1,0)}_{\Fl}[-1]
    \right]        .
        \] By Lemma \ref{lemAnyMapIsIsoInQCohg}, we then know $\iota$ is an isomorphism, and $p$ induces an isomorphism $i_{*}i^{\dagger}\omega^{(1,0)}_{\Fl}[-1]\cong i_{*}i^{\dagger}\omega^{(1,0)}_{\Fl}[-1]$. Thus we are left with a triangle        \[
                i_{*}i^{\dagger}\omega^{(1,0)}_{\Fl}[-1]\xrightarrow{\delta} j_{!}j^{*}\omega^{(1,0)}_{\Fl}\xrightarrow{s} \omega^{(1,0)}_{\Fl}\xrightarrow{+1}i_{*}i^{\dagger}\omega^{(1,0)}_{\Fl} .
                \]  Such extension is necessarily the excision extension. This implies $\delta$ coincides with the Cousin map. 
\end{proof}

We can now finish the proof of Lemma \ref{lemFaltingsExtOnFl}.
\begin{proof}[Proof of Lemma \ref{lemFaltingsExtOnFl}] 
Consider \(\Fil^{t}_{*}\) on 
$\mE_{(-1,k)}(\mF_{m+k})$. The graded pieces are \[\gr^{t}_{i-m-k}\cong \mE_{(-1,k)}(\omega^{(i,m+k-i)}_{\Fl}\otimes C^{\la})\] for $0\le i\le m+k$. 
By Theorem \ref{mainthmPilloniOnFl}, we know that
    \( \gr^{t}_{i-m-k}\mcal{E}_{(-1,k)}(\mF_{m+k})\cong 0 \) unless $i=m+k$ or $m$, and  \begin{align}\label{alignOnceObtainShortES}
    \mE_{(-1,k)}(\omega^{(m,k)}_{\Fl}\otimes C^{\la})\cong j_{!}j^{*}\omega^{(m+1,0)}_{\Fl},\;\mE_{(-1,k)}(\omega^{(m+k,0)}_{\Fl}\otimes C^{\la})\cong i_{*}i^{\dagger}\omega^{(m+1,0)}_{\Fl}[-1] .
    \end{align}
 Note that all the terms lie in $\mcal{C}_{0}$, and in particular, we can work in ${}^{p}B^{\heartsuit}$ defined in Lemma \ref{lemPerverseOnQCoh}, and the triangle gives short exact sequences in ${}^{p}B^{\heartsuit}$: 
        \[
        0\to     \mE_{(-1,k)}(\omega^{(m,k)}_{\Fl}\otimes  C^{\la})
     \to  \mE_{(-1,k)}(\mF_{m+k})   \to \mE_{(-1,k)}(\omega^{(m+k,0)}_{\Fl}\otimes C^{\la})\to 0, 
        \] such that $\Theta$ acts trivially on the graded pieces. By Lemma \ref{lemReducektok=1} and Lemma \ref{lemCasek=1CoveredInPilloni}, we have     
                \begin{align*}
\Coker(\Theta)\cong \left[\omega^{(m+1,0)}_{\Fl}-i_{*}i^{\dagger}\omega^{(m+1,0)}_{\Fl}[-1]
    \right]
\twoheadrightarrow  \mE_{(-1,k)}(\omega^{(m+k,0)}_{\Fl}\otimes C^{\la})\\\cong i_{*}i^{\dagger}\omega^{(m+1,0)}_{\Fl}[-1].
                \end{align*}
                Since \[\RHom(\omega^{(m+1,0)}_{\Fl},i_{*}i^{\dagger}\omega^{(m+1,0)}_{\Fl})\in D^{\ge0},\] we know that the map factors uniquely through $i_{*}i^{\dagger}\omega^{(m+1,0)}_{\Fl}[-1]$, which then by Lemma \ref{lemAnyMapIsIsoInQCohg} gives an isomorphism. This implies that by Lemma \ref{lemCasek=1CoveredInPilloni},
        \(
        \mE_{(-1,k)}(\omega^{(m,k)}_{\Fl}\otimes  C^{\la})\cong \Ker(\Theta) ,
        \) and the action of $\Theta$ is as described.

The proof for $(k-1,0)$
is similar. Once we obtain (\ref{alignOnceObtainShortES}), 
we have a surjection \begin{align*}
        p:\Coker(\Theta)\cong \left[\omega^{(m+1,0)}_{\Fl}-i_{*}i^{\dagger}\omega^{(m+1,0)}_{\Fl}[-1]
    \right]
        \twoheadrightarrow \mE_{(k-1,0)}(\omega^{(m+k,0)}_{\Fl}\otimes C^{\la})\\\cong \left[\omega^{(m+1,0)}_{\Fl}-i_{*}i^{\dagger}\omega^{(m+1,0)}_{\Fl}[-1]
    \right] .
\end{align*} We then first project to $i_{*}i^{\dagger}\omega^{(m+1,0)}_{\Fl}[-1]$, and by the same argument using Lemma \ref{lemAnyMapIsIsoInQCohg}, we know $p$ induces an isomorphism on the top graded piece     \(p: i_{*}i^{\dagger}\omega^{(m+1,0)}_{\Fl}[-1]\cong i_{*}i^{\dagger}\omega^{(m+1,0)}_{\Fl}[-1] \). Hence $p$ induces a surjection     \( p:\omega^{(m+1,0)}_{\Fl}\to \omega^{(m+1,0)}_{\Fl} \), which again by Lemma \ref{lemAnyMapIsIsoInQCohg} is an isomorphism. Therefore, we know \[\Coker(\Theta)\cong \mE_{(k-1,0)}(\omega^{(m+k,0)}_{\Fl}\otimes C^{\la}),\] and we are again done by Lemma \ref{lemCasek=1CoveredInPilloni}.
\end{proof}

%%%%%%%%%%%%%%%%%%%%%%%%%%%%%%%%%%%%%%%%%%%%%%%%%%%%%%%%%%%%%%%%%%%%%%%%%%%%%%%%%%%%%%%%%%%%%%%%%%%%%%%%%%%%%%%%%%%%%%%%%%%%%%%%%%%%%%%

\subsection{Proof of Theorem \ref{thmFontaine=Theta}}\label{subsectionGeneral(a,b)a-bge0}
In this subsection, we finish the proof of Theorem \ref{thmFontaine=Theta}. By Remark \ref{rmkTwistInBdR} and Corollary \ref{cork<0FollowsFromk>0}, it suffices to consider \((a,b)=(k-1,0)\) for \(k\in \ZZ_{\ge1}\).
% Note that by Remark \ref{rmkTwistInBdR}, we can restrict ourselves to considering $\hat{\mE}_{(k-1,0)}(\OBdR^{\la})$
% with $k\in\ZZ_{\ge1}$.

By Lemma \ref{lemOlaDecompoEasy},
we have an exact sequence in \(\mrm{Perv}\)          \[
0\to \hat{\mE}_{(k-1,0)}(\gr_{-k}^{t}\OBdR^{\la})\to \hat{\mE}_{(k-1,0)}(\OBdR^{\la})\to \hat{\mE}_{(k-1,0)}(\gr_{0}^{t}\OBdR^{\la})\to 0,
\] and the action of $\Theta$ on $\gr_{*}^{t}$ is zero, 
and thus its action on $\hat{\mE}_{(k-1,0)}(\OBdR^{\la})$ is induced by a unique canonical map 
\begin{align}\label{alignFontaineOBdR}
\ti{N}^{k}:\gr^{t}_{0}\hat{\mE}_{(k-1,0)}(\OBdR^{\la})\to \gr^{t}_{-k}\hat{\mE}_{(k-1,0)}(\OBdR^{\la}).  
\end{align} 
\begin{lem}
The map $\ti{N}^{k}$ as defined above is $\mO^{\sm}$-linear, filtered, and compatible with connections, where both sides are equipped with the induced Hodge filtration (Definition \ref{dfntadicFil}).
\end{lem}
\begin{proof}
By Proposition \ref{propf*equalsE0}, \(E_{0}(-)\) is lax symmetric monoidal. 

Therefore, \(\hat{\mE}_{(a,b)}(\OBdR)\) is an \(E_{0}(\mO^{\sm})\)-module in \(D(\QQ_{p}[\Theta])\) (equipped with convolution symmetric monoidal structure).
% The action of $\Theta$ on $\hat{\mE}_{(a,b)}(\OBdR)$ satisfies Leibniz rule for the action of $\mO^{\sm}$. 
We know that \(E_{0}(\mO^{\sm})\cong \mO^{\sm}\), and $\Theta$ acts on $\mO^{\sm}$ by zero. As a result, the action of $\Theta$ on \(\hat{\mE}_{(a,b)}(\OBdR)\) is $\mO^{\sm}$-linear, filtered, and commutes with connections, and thus $\ti{N}^{k}$ is also \(\mO^{\sm}\)-linear, filtered, and commutes with connections.
\end{proof}
As a corollary, $\Fib(\ti{N}^{k})$ inherits the \(\mO^{\sm}\)-module structure, the Hodge filtration and the connection. By Lemma \ref{lemOlaDecompoEasy}, Theorem \ref{mainthmPanPilloni}, and Proposition \ref{propFaltingsExt}, we see that 
\begin{align}\label{aligngrjFibNkl}\gr^{j}\Fib(\ti{N}^{k})\cong
\begin{cases}\left[\omega^{(j-k+1,-j),\sm}-i_{*}\omega^{(j-k+1,-j),\sm}[-1]\right],&0\le j\le k-1;\\
\omega^{(j-k+1,-j),\sm},&j\ge k.
\end{cases}
\end{align}
By Corollary \ref{propH0OBdR}, we have canonical $\Gal_{\QQ_{p}}$-equivariant
maps         \[
        \Sym^{k-1}D^{\sm}\to \hat{\mE}_{(k-1,0)}(\OBdR^{\la})\to \hat{\mE}_{(k-1,0)}(\gr^{t}_{0}\OBdR^{\la}),
        \] that are compatible with connections.
        
        Since the action of $\Theta$ on $\Sym^{k-1}D^{\sm}$ is zero, this induces a canonical map         \[
                i_{k}:\Sym^{k-1}D^{\sm}\to \Fib(\ti{N}^{k}) ,
                \] that is \(\mO^{\sm}\)-linear, filterd, and compatible with connections.
                Moreover, by Corollary \ref{propH0OBdR}, upon projecting to $\Fib(\ti{N}^{k})/\Fil^{k}$, this induces an isomorphism         \[
                        i_{k}:\Sym^{k-1}D^{\sm}\cong H^{0}(\Fib(\ti{N}^{k})/\Fil^{k}) .
                        \]
% We will therefore consider the  

% Now we want to understand $H^{1}(\Fib(\ti{N}^{k}))$. The idea of the proof is to make use of the action of $\mE_{(0,0)}(\OBdRR{\ell+1})$ to transfer the knowledge of $H^{0}$ to $H^{1}$.
\begin{prop}\label{propNabla=GM} For $k\in\ZZ_{\ge 1}$, we have a 
        % natural 
        filtered     \( B(\QQ_{p})\times\TT(K^{p}) \)-equivariant $\mO^{\sm}$-linear isomorphism    \begin{align}\label{alignH1OBdR}
H^{1}(\Fib(\ti{N}^{k}))\cong i_{*}\Sym^{k-1}D^{\sm}
\end{align} that is compatible with connections. 
% Moreover, for $\ell\ge k$, $\nabla$ induces a map         \[
%         H^{1}(\nabla):H^{1}(\Fib(\ti{N}^{k}))\to H^{1}(\Fib(N_{\ell-1}^{k}))\otimes_{\mO^{\sm}}\Omega_{\log}^{1,\sm} ,
%         \] which in terms of the isomorphism above, coincides with the Gauss-Manin connection $\nabla_{\mrm{GM}}$ of $i_{*}\Sym^{k-1}D^{\sm}$.
\end{prop}
\begin{proof}
        % [Proof of Proposition \ref{propNabla=GM}]
% We first prove (\ref{alignH1OBdR}). 
For $k=1$, from (\ref{aligngrjFibNkl}), we see that $H^{1}(\Fib(\ti{N}^{1}))\cong i_{*}\mO^{\sm}$ as desired. For the compatibility with connections, we use the following uniqueness: let \(\nabla:i_{*}\mO^{\sm}\to i_{*}\Omega^{1,\sm}_{\log}\) be any \(B(\QQ_{p})\)-equivariant connection, then $\nabla-\nabla_{\mrm{GM}}:i_{*}\mO^{\sm}\to i_{*}\Omega_{\log}^{1,\sm}$ is a canonical $\mO^{\sm}$-linear map.
In particular, it is $B(\QQ_{p})$-invariant. This corresponds to some $f\in M_{(1,-1)}^{\dagger}$ that is $B(\QQ_{p})$-invariant.  We conclude that $f=0$ by Lemma \ref{lemNoBinv}, and  thus     \( \nabla=\nabla_{\mrm{GM}} \) as desired. 

% we again start with the case $k=1$. We note that the connection $\nabla$ of $\OBdR^{+,\la}$ is a derivative with respect to its natural algebraic structure (i.e. satisfies the Leibniz rule). Since the map     \( \mO^{\sm}\to \OBdR^{+,\la} \)
% is compatible with the connection, we know for for $s\in H^{0}(\Fib(\ti{N}^{1}))\cong \mO^{\sm}$, $t\in H^{1}(\Fib(\ti{N}^{1}))\cong i_{*}\mO^{\sm}$, we have     \( \nabla(st)=s\nabla(t)+\nabla_{\mrm{GM}}(s)t. \)
% Hence 

For general $k\in\ZZ_{\ge 1}$, by Proposition \ref{propf*equalsE0}, \(\hat{\mE}_{(k-1,0)}(\OBdR^{\la})\) has a natural structure of an \(\hat{\mE}_{(0,0)}(\OBdR^{\la})\)-module. In particular, we have \[\hat{\mE}_{(0,0)}(\OBdR^{\la})\otimes_{\mO^{\sm}}\hat{\mE}_{(k-1,0)}(\OBdR^{\la})\to \hat{\mE}_{(k-1,0)}(\OBdR^{\la}).
\]
This map is compatible with $\Theta$-action, and with Hodge filtrations, and $t$-adic filtrations. Thus this induces a map 
\begin{align}\label{alignImportantIso}
\Sym^{k-1}D^{\sm}\otimes_{\mO^{\sm}}i_{*}\mO^{\sm}\xrightarrow{i_{k}\otimes 1}
                H^{0}(\Fib(\ti{N}^{k}))\otimes_{\mO^{\sm}}H^{1}(\Fib(\ti{N}^{1}))\to H^{1}(\Fib(\ti{N}^{k})),
\end{align}
that is compatible with connections and Hodge filtrations.                 
So it suffices to show that the composition in (\ref{alignImportantIso}) is an isomorphism.
Since the map is compatible with the Hodge filtrations, so it suffices to look at the graded piece \(\gr^{i}\) for \(0\le i\le k-1\), and we are reduced to proving 
Lemma \ref{lemTechnicalSen} below.
% Now for general $k\ge 1$,  we know the connection $\nabla$
% on $H^{1}(\Fib(\ti{N}^{k}))$
% also satisfies the same Leibniz rule: for $x\in H^{0}(\Fib(\ti{N}^{k}))\cong  \Sym^{k-1}D^{\sm}$, and for $y\in H^{1}(\Fib(\ti{N}^{1}))\cong i_{*}\mO^{\sm}$, we have $yx\in H^{1}(\Fib(\ti{N}^{k}))\cong i_{*}\Sym^{k-1}D^{\sm}$, and
%         \[
% \nabla(yx)=\nabla(y)x+y\nabla(x)         .
%         \] Now by the case $k=1$, we know $\nabla(y)=\nabla_{\mrm{GM}}(y)$, and by compatibility of connections in Riemann-Hilbert correspondence, we have $\nabla(x)=\nabla_{\mrm{GM}}(x)$. 
% Thus $\nabla(yx)=\nabla_{\mrm{GM}}(yx)$. Now by the isomorphism (\ref{alignImportantIso}) above, we know $\nabla=\nabla_{\mrm{GM}}$ on $H^{1}(\Fib(\ti{N}^{k}))$. 
\end{proof}
% The following lemma was used in the proof.
\begin{lem}\label{lemTechnicalSen}
By Proposition \ref{propf*equalsE0}, \(\hat{\mE}_{(k-1,0)}(\mO^{\la})\) has a natural structure of an \(\hat{\mE}_{(0,0)}(\mO^{\la})\)-module. Then the induced map 
% (as in (
        % \ref{alignSomeNameAfterTwist}))        
        \[
        H^{0}(\mE_{(k-1,0)}(\mO^{\la}))\otimes_{\mO^{\sm}}H^{1}(\mE_{(0,0)}(\mO^{\la}))\to H^{1}(\mE_{(k-1,0)}(\mO^{\la})) 
        \] is an isomorphism.
\end{lem}
Using $\VB$, this is reduced to the corresponding question on $\Fl$:
\begin{lem}
Let $k\in\ZZ_{\ge1}$. By Proposition \ref{propf*equalsE0}, \(\mE_{(k-1,0)}\) has a structure of an \(\mE_{(0,0)}(C^{\la})\)-module, and the following composition induces an isomorphism \[
        \omega^{(1-k,0)}_{\Fl}\hatotimes_{\mO_{\Fl}} \mE_{(0,0)}(C^{\la})\to \mE_{(k-1,0)}(C^{\la})\otimes_{\mO_{\Fl}}\mE_{(0,0)}(C^{\la})\to   \mE_{(k-1,0)}(C^{\la}).
        \] where the first map is induced by $\omega^{(1-k,0)}_{\Fl}\to \mE_{(k-1,0)}(C^{\la})$ in Theorem \ref{mainthmPilloniOnFl}
\end{lem}
\begin{proof}
We have the natural morphism as above. 
By Theorem \ref{mainthmPilloniOnFl}, we know         \[
        \mE_{(0,0)}(C^{\la})\cong\left[
        \mO_{\Fl}-i_{*}i^{\dagger}\mO_{\Fl}[-1]
        \right],\; \mE_{(k-1,0)}(C^{\la})\cong\left[
        \omega^{(1-k,0)}_{\Fl}-i_{*}i^{\dagger}\omega^{(1-k,0)}_{\Fl}[-1]
        \right].
        \]
First, we see that $\omega^{(1-k,0)}_{\Fl}\otimes_{\mO_{\Fl}}\mO_{\Fl}\to \omega^{(1-k,0)}_{\Fl}$ is an isomorphism. Therefore, we have an induced morphism    \[\omega^{(1-k,0)}_{\Fl}\hatotimes_{\mO_{\Fl}} 
i_{*}i^{\dagger}\mO_{\Fl}[-1]\to i_{*}i^{\dagger}\omega^{(1-k,0)}_{\Fl}[-1].
        \] It suffices to show that it is an isomorphism. Note that both sides are concentrated in $\infty$, so we can verify the isomorphism after applying $i^{\dagger}.$ 

Now we need to go through the proof of \cite[Proposition 5.7, Lemme 5.6]{Pilloni22}. We refer readers to \cite{Pilloni22} for the undefined notation. In what follows, $*_{g}$ and $*_{d}$ denote the left action and right action respectively.
We have by \cite[Lemma 5.5]{Pilloni22}
that         \[
        i^{\dagger}C^{\la}\cong (\mO_{T,1}\hatotimes_{\CC_{p}}\mO_{U\backslash G,1})^{(\fh,*_{g}\otimes*_{g})=0} ,
        \] with $*_{2}$ on the LHS corresponding to $\id\otimes *_{d}$ on the right hand side.
 Then for any $\chi\in \CC_{p}^{\oplus2}$         \[
                \Ext_{(\fb,*_{2})}(\chi,i^{\dagger}C^{\la})\cong
             (\mO_{T,1}\hatotimes_{\CC_{p}}
                \Ext_{(\fb,*_{d})}(\chi,\mO_{U\backslash G,1}))^{(\fh,*_{g}\otimes*_{g})=0} .
                \] 
So now we are reduced to proving that the natural morphism        \begin{align}\label{alignEqOUbackG}
        H^{0}(\mE_{(k-1,0)}(\mO_{U\backslash G,1}))\hatotimes_{\CC_{p}} H^{1}(\mE_{(0,0)}(\mO_{U\backslash G,1}))\to H^{1}(\mE_{(k-1,0)}(\mO_{U\backslash G,1})) 
\end{align}  is an isomorphism, where we are taking $(\fb,*_{d})$ when applying $\mE_{(k-1,0)}$.
By the proof of \cite[Lemme 5.5]{Pilloni22}, we know for $\ell\in\ZZ_{\ge 0}$, $\RHom_{(\fb,*_{d})}((\ell,0),\mO_{U\backslash G,1}) $ is represented by a complex         \[\varinjlim_{m\in\ZZ_{\ge1}}
        \widehat{\bigoplus}_{n\ge 0}\CC_{p}\cdot \frac{1}{p^{nm}} x^{n}a^{\ell-n}d^{n}\xrightarrow{d} \varinjlim_{m\in\ZZ_{\ge1}}\widehat{\bigoplus}_{n\ge -1}\CC_{p}\cdot \frac{1}{p^{nm}} x^{n+1}a^{\ell-n}d^{n} ,
        \] where $x,a,d$ are some variables on $U\backslash G$. The differential $d$ is given by         \[
                d\left(\sum_{n\ge 0}a_{n}x^{n}a^{\ell-n}d^{n}\right)=\sum_{n\ge 0}(\ell-n)a_{n}x^{n+1}a^{\ell-n}d^{n} .
                \]
We see that the part $n=-1$ is always in the cokernel, and corresponds to weight $(\ell+1)$-part by \cite[Lemme 5.5]{Pilloni22}. If we take the weight $0$-part, we see $H^{0}$ is $1$-dimensional generated by $x^{\ell}d^{\ell}$, and $H^{1}$ is also $1$-dimensional generated by $x^{\ell+1}d^{\ell}$ (by taking $n=\ell)$.
In our situation, $H^{0}(\mE_{(k-1,0)}(\mO_{U\backslash G,1}))$ is generated by $x^{k-1}d^{k-1}$, $H^{1}(\mE_{(k-1,0)}(\mO_{U\backslash G,1}))$ is generated by $x^{k}d^{k-1}$, and $H^{1}(\mE_{(0,0)}(\mO_{U\backslash G,1}))$ is generated by $x$. By $(x^{k-1}d^{k-1})\cdot x=x^{k}d^{k-1}$, we know the map (\ref{alignEqOUbackG}) is an isomorphism.
\end{proof}

% \begin{lem}
% The filtration in (\ref{aligngrjFibNkl}) induces a filtration on $H^{i}(\Ker(\ti{N}^{k}))$ with \begin{align*}
% &\gr^{j}H^{0}(\Fib(\ti{N}^{k}))\cong H^{0}(\gr^{j}\Fib(\ti{N}^{k}))\cong\omega^{(j-k+1,-j),\sm},\; 0\le j\le \ell;
% \\
% &\gr^{j}H^{1}(\Fib(\ti{N}^{k}))\cong H^{1}(\gr^{j}\Fib(\ti{N}^{k}))\cong
% \begin{cases}i_{*}\omega^{(j-k+1,-j),\sm}[-1],&0\le j\le k-1;\\0,&k\le j\le \ell.\\
% \end{cases}
% \end{align*}
% \end{lem}
% \begin{proof}
% Given (\ref{aligngrjFibNkl}), it suffices to show that the connecting morphisms induced by the filtration are zero. Assume it is not the case, then we can obtain some nonzero $B(\QQ_{p})$-equivariant $\mO^{\sm}$-linear morphism 
% $f:\omega^{(j-k+1,-j),\sm}\to i_{*}\omega^{(j'-k+1,-j'),\sm}$ for $j'<j$ by looking at the smallest $\ell$ where the statement fails. However, such map $f$ does not exist by Lemma \ref{lemNoBinv}. 
% \end{proof}
\begin{lem}\label{lemNoBinv}
If $H^{0}(\infty,\omega^{(a,b),\sm})^{B(\QQ_{p})}\ne 0$, then \(a=b\). 
\end{lem}
\begin{proof}
Let $f$ be a non-zero $B(\ZZ_{p})$-invariant section \(\omega^{(a,b),\sm}\). Then it comes from a section at level \(K^{p}\Gamma_{0}(p^{n})\) for \(n\) sufficiently large. By conjugating by \(\mrm{diag}(p,1)\in B(\QQ_{p})\), we can assume \(n=1\). 
Let \(x\) be a cusp on any connected component of \(\mX_{K^{p}\Gamma_{0}(p),\infty}^{\mrm{ord}}\). We can look at its $q$-expansion $f=\sum_{i\in \NN}a_{i}q^{i}$ at \(x\). 

We consider the action of Frobenius $\phi$ that is given by the action of $\begin{pmatrix}
p^{-1} & 0\\0 & 1
\end{pmatrix}$, so $\phi(f)=\begin{pmatrix}
p^{-1} & 0\\0 & 1
\end{pmatrix}f=f$. But in terms of the \(q\)-expansion, $\phi(f)=\sum_{i\in \NN}a_{i}q^{pi}$, so $\phi(f)=f$ forces $a_{i}=0$ for $i>0$. Then multiplying by \(f\) induces an isomorphism \(i_{x}^{-1}\omega^{(a,b)}_{\mX_{K^{p}\Gamma_{0}(p)}}\cong i_{x}^{-1}\mO_{\mX_{K^{p}\Gamma_{0}(p)}}\) that is compatible with \(U_{p}\)-operators, which implies that \(a=b\). 
\end{proof}

Now we can finish the proof of Theorem \ref{thmFontaine=Theta}.
\begin{proof}[Proof of Theorem \ref{thmFontaine=Theta}]
Note that by Remark \ref{rmkTwistInBdR} and Corollary \ref{cork<0FollowsFromk>0}, we can restrict ourselves to the case $(a,b)=(k-1,0)$
with $k\in\ZZ_{\ge1}$.
We 
have the de Rham complex (\cite[Corollary 2.4.2]{DLLZ2022logarithmicJAMS})        \begin{equation}
        0\to \BdR^{\la}\to \OBdR^{\la}\to \OBdR^{\la}\otimes_{\mO^{\sm}}\Omega_{\log}^{1,\sm}\to 0 ,\label{proofThmFontaine=Theta(a,b)=(0,0)}
        \end{equation} which after taking $\hat{\mE}_{(k-1,0)}$
        gives         \[
0\to \hat{\mE}_{(k-1,0)}(\BdR^{\la})\to \hat{\mE}_{(k-1,0)}(\OBdR^{\la})\to \hat{\mE}_{(k-1,0)}(\OBdR^{\la})\otimes\Omega_{\log}^{1,\sm}\to 0 ,
                \]
which is compatible with \(\Theta\)-actions, Hodge filtrations \(\Fil^{*}\), and \(t\)-adic filtrations \(\Fil^{t}_{*}\). 

% Note that we write \(N^{k}\) for the geometric Fontaine operator for \(\hat{\mE}_{(k-1,0)}(\BdR^{\la})\) as in the statement of Theorem \ref{thmFontaine=Theta}, and write \(\ti{N}^{k}\) for that for \(\hat{\mE}_{(k-1,0)}(\OBdR^{\la})\) as in (\ref{alignFontaineOBdR}).

                %  Moreover, the sequence is compatible with $\Fil_{\bullet}^{t}$, 
Then we have the following diagram 
             \[
                        \begin{tikzcd}
%  \mscr{C}^{k} \arrow[r] \arrow[d]      & \Fib({\ti{N}^{k}}) \arrow[r,"\nabla"] \arrow[d]  & \Fib(\ti{N}^{k})\otimes\Omega_{\log}^{1,\sm}\arrow[d]   \\
\hat{\mE}_{(k-1,0)}(\mO^{\la})\arrow[r,"\iota"]\arrow[d,"N^{k}"] 
& \gr^{t}_{0} \hat{\mE}_{(k-1,0)}(\OBdR^{\la})\arrow[r,"\nabla"]\arrow[d,"\ti{N}^{k}"] 
& \gr^{t}_{0} \hat{\mE}_{(k-1,0)}(\OBdR^{\la})\otimes\Omega_{\log}^{1,\sm}\arrow[d,"\ti{N}^{k}\otimes 1"] 
\\
\hat{\mE}_{(k-1,0)}(\mO^{\la}(k))\arrow[r,"\iota"] 
& \gr^{t}_{-k} \hat{\mE}_{(k-1,0)}(\OBdR^{\la})\arrow[r,"\nabla"]
& \gr^{t}_{-k} \hat{\mE}_{(k-1,0)}(\OBdR^{\la})\otimes\Omega_{\log}^{1,\sm},
\end{tikzcd} 
                \] 
where  the rows are short exact sequences of perverse sheaves, and the category is in the category of filtered sheaves over \(\Fl_{\an}\), where all the objects are equipped with Hodge filtrations, and \(\Omega^{1,\sm}_{\log}\) is endowed with the trivial filtration where only non-trivial graded piece is \(\gr^{1}\). 

By Theorem \ref{mainthmPanPilloni}, the objects on the second row are concentrated in degree \(1\) for the natural t-structure, and those on the first row are in degree \([1,2]\). Therefore, to understand \(N^{k}\), we can take \(H^{1}(-)\), and we have the following diagram \[
                        \begin{tikzcd}
 \mscr{C}^{k} \arrow[r] \arrow[d]      & i_{*}\Sym^{k-1}D^{\sm} \arrow[r,"\nabla_{GM}"] \arrow[d]  & i_{*}\Sym^{k-1}D^{\sm}\otimes\Omega_{\log}^{1,\sm}\arrow[d]   \\
\hat{\mE}^{1}_{(k-1,0)}(\mO^{\la})\arrow[r,"\iota"]\arrow[d,"H^{1}(N^{k})"] 
& \gr^{t}_{0} \hat{\mE}^{1}_{(k-1,0)}(\OBdR^{\la})\arrow[r,"\nabla"]\arrow[d,"H^{1}(\ti{N}^{k})"] 
& \gr^{t}_{0} \hat{\mE}^{1}_{(k-1,0)}(\OBdR^{\la})\otimes\Omega_{\log}^{1,\sm}\arrow[d,"H^{1}(\ti{N}^{k})\otimes 1"] 
\\
\hat{\mE}^{1}_{(k-1,0)}(\mO^{\la}(k))\arrow[r,"\iota"] 
& \gr^{t}_{-k} \hat{\mE}^{1}_{(k-1,0)}(\OBdR^{\la})\arrow[r,"\nabla"]
& \gr^{t}_{-k} \hat{\mE}^{1}_{(k-1,0)}(\OBdR^{\la})\otimes\Omega_{\log}^{1,\sm},
\end{tikzcd} 
                \] 
with \(\hat{\mE}_{(a,b)}^{1}:=H^{1}(\hat{\mE}_{(a,b)}(-))\) and \(\mscr{C}^{k}:=\Fib(H^{1}(N^{k}))\), \(\hat{\mE}^{1}_{(k-1,0)}(\mO^{\la})\cong i_{*}\omega^{(0,k-1),\sm}\), and \(\hat{\mE}^{1}_{(k-1,0)}(\mO^{\la}(k))\cong i_{*}\omega^{(k,-1),\sm}\), which is compatible with Hodge filtrations and all the rows and columns are fiber sequences. 
In particular, \[\gr^{0}(\mscr{C}^{k})\cong \hat{\mE}^{1}_{(k-1,0)}(\mO^{\la}),\;
\gr^{k}(\mscr{C}^{k})\cong \hat{\mE}^{1}_{(k-1,0)}(\mO^{\la}(k))[-1].
\]
% \(\hat{\mE}^{1}_{(k-1,0)}(\mO^{\la})\) (resp 

On the other hand, the first row shows that we have a filtered isomorphism \(\mscr{C}^{k}\cong i_{*}\dR(\Sym^{k-1}D^{\sm})\), which shows that \(\gr^{0}(\mscr{C}^{k})\cong i_{*}\omega^{(0,k-1),\sm}\), \(\gr^{k}(\mscr{C}^{k})\cong i_{*}\omega^{(k,-1),\sm}[-1]\),
and the connecting morphism \(\delta\) in \[\gr^{k}(\mscr{C}^{k})\to 
\mscr{C}^{k}\to \gr^{0}(\mscr{C}^{k})\xrightarrow{\delta} \gr^{k}(\mscr{C}^{k})[-1]
\] is by definition \(\theta^{k}\).               
\end{proof}

% \subsection{Proof of Theorem \ref{thmFontaine=Theta} when $k< 0$}\label{subsectionGeneral(a,b)a-ble0}

\section{Arithmetic corollary}\label{sectionArithCor}
In this section, we prove the classicality of modular forms (Theorem \ref{thmMainThmClassicality}) using the result we have obtained above. In fact, both (1) and (2) of Theorem \ref{thmFontaine=Theta} give a proof of Theorem \ref{thmMainThmClassicality}. 
% The proof is straightforward.
 We will use Theorem \ref{thmFontaine=Theta} (1) below. 

We first recall the following result from \cite{Pan22}. 
\begin{thm}
\label{PanPilloni}
For $k\in\ZZ_{\ge1}$, we denote $$\tilde{\rho}_{k}:=\RHom_{\fb}((k-1,0),R\Gamma(K^{p},\QQ_{p})^{R-\la})[1].$$ Then
 we have a $B(\QQ_{p})\times\Gal_{\QQ_{p}}\times \TT(K^{p})$-equivariant isomorphism  \begin{align*}\tilde{\rho}_{k}\hatotimes_{\QQ_{p}}\CC_{p}\cong N_{0}\oplus N_{k}(-k),
\end{align*} where        \(
        N_{k}\cong M_{(1,-k)}^{\dagger}[0] ,
        \) and  $N_{0}$ lies in a distinguished triangle         \[
                R\Gamma(\Fl,\omega^{(1-k,0),\sm})[1]\to N_{0}\to M_{(1-k,0)}^{\dagger}[0]\xrightarrow{+1} .
                \]
        Here $(-)$ refers to the Tate twist, and \[R\Gamma(\Fl,\omega^{(1-k,0),\sm})\cong\varinjlim_{K_{p}}R\Gamma(\mX_{\Kpp},\omega^{(1-k,0)}).\]

In particular, \(\ti{\rho}_{k}\hatotimes \CC_{p}\) is Hodge-Tate of weight \(0,k\) (Definition \ref{dfnArithSenOp}), where $\Theta$ equals $0$ on $N_{0}$ and $-k$ on $N_{k}(-k)$.
\end{thm}

% \begin{rmk}
% From now on, we will ignore the action of the center, and identify    \( \omega^{(a,b),\sm} \)
% with         \( \omega^{a-b,\sm} \).
% \end{rmk}
\begin{proof}
% For $k\ne 1$, $\Ext_{\fb}(\tilde{H}^{0,\la}(K^{p},\CC_{p}))\cong 0$. 
Let us give a proof here.
We know by Theorem \ref{mainthmPanPilloni},                 \[
                \RHom_{\fb}((k-1,0),\mO^{\la})[1]\otimes\chi^{(1-k,0)}\cong \mcal{N}_{0}\oplus \mcal{N}_{k}(-k) ,
                \]
and $\mcal{N}_{k}\cong i_{*}\omega^{(1,-k),\sm}$, and  we have a short exact sequence               \[
                0\to \omega^{(1-k,0),\sm}\to \mcal{N}_{0}\to i_{*}\omega^{(1-k,0),\sm}[-1]\to 0 .
                \]
                 We then take cohomology $R\Gamma(\Fl,-)$ 
                and obtain by Corollary \ref{corPrimitiveLA}, \begin{align*}
 \RHom_{\fb} &   
 \left(
 (k-1,0),R\Gamma(K^{p},\QQ_{p})^{R-\la}
 \right)       \hatotimes\CC_{p}\otimes\chi^{(1-k,0)}\\&\cong \RHom_{\fb}((k-1,0),R\Gamma(\Fl,\mO^{\la}))\otimes\chi^{(1-k,0)}\\
 &\cong R\Gamma(\Fl,\RHom_{\fb}((k-1,0),\mO^{\la}))\otimes\chi^{(1-k,0)}
 \\ &\cong R\Gamma(\Fl,\mcal{N}_{0})\oplus R\Gamma(\Fl,\mcal{N}_{k})(-k).
                \end{align*}
We are then done by putting $N_{0}:=R\Gamma(\Fl,\mcal{N}_{0})[1]$ and $N_{k}:=R\Gamma(\Fl,\mcal{N}_{k})$.
\end{proof}

% We can also consider the Fontaine operator for the (infinite dimensional) solid $\Gal_\QQ$-representation $\tilde{\rho}_k$.
As in Example \ref{egOlaArithSen}, we have                 \(
                D_{\mrm{arith}}(\mO^{\la})\hatotimes_{\QQ_{p}(\zeta_{p^{\infty}})}\CC_{p}\cong\mO^{\la} .
                \) Taking $R\Gamma(\Fl,-)$, this implies that \(R\Gamma(K^{p},\CC_{p})^{R-\la}\) admits an arithmetic Sen operator \(\Theta\) (Definition \ref{dfnArithSenOp}).  
                % Let \(\Theta\) denote the arithmetic Sen operator $\Theta$
% on $R\Gamma(K^{p},\QQ_{p})^{R-\la}\hatotimes\CC_{p}$ as in .
Thus \(\ti{\rho}_{k}\hatotimes\CC_{p}\) also admits an arithmetic Sen operator $\Theta$, which is zero on $N_{0}$ and $-k$ on $N_{k}(-k)$.

Now we consider $\tilde{\rho}_{k}\hatotimes_{\QQ_{p}}B_{\dR}^{+}$,
%  and \(\hatDarith(\ti{\rho}_{k}\hatotimes B_{\dR})\) carries an action of the arithmetic Sen operator $\Theta$, where \(\hatDarith(-)\) is as in Definition \ref{dfnBdRSemiLinear}.
% Each of $\tilde{\rho}_{k} admits an arithmetric Sen operator, and thus so does \(\ti{\rho}_{k}\hatotimes \). 
% Moreover, $D_{\mrm{arith}}(\tilde{\rho}_{k}\hatotimes_{\QQ_{p}}B_{\dR}^{+}/t^{k+1})$ carries an action of the arithmetic Sen operator $\Theta$, 
% and is filtered by $D_{\mrm{arith}}(\tilde{\rho}_{k}\hatotimes_{\QQ_{p}}\CC_{p})(i)$ for $0\le i\le k$. Thus we can take $E_{0}(-)$-part of $D_{\mrm{arith}}(\tilde{\rho}_{k}\hatotimes_{\QQ_{p}}B_{\dR}^{+}/t^{k+1})$. 
%   \[
%                 E_{0}(\tilde{\rho}_{k}\hatotimes_{\QQ_{p}}B_{\dR}^{+}/t^{k+1}):=E_{0}(D_{\mrm{arith}}(\tilde{\rho}_{k}\hatotimes_{\QQ_{p}}B_{\dR}^{+}/t^{k+1}))\otimes_{\QQ_{p}(\zeta_{p^{\infty}})}\CC_{p} ,
%                 \] 
% and then the action of $\Theta$ extends $\CC_{p}$-linearly to an action on         \( E_{0}(\tilde{\rho}_{k}\hatotimes B^{+}_{\dR}/t^{k+1}). \) 
Since $\tilde{\rho}_{k}$ is of Hodge-Tate weight $0,k$, As in definition \ref{dfnFontaineOp},  we have a distinguished triangle in \(D(\CC_{p}[\Theta])\) \begin{equation}\label{seqE0Rho}
                                N_{k}\to \hat{E}_{0}(\tilde{\rho}_{k}\hatotimes_{\QQ_{p}}B_{\dR}^{+})\to N_{0}\xrightarrow{+1} ,
                \end{equation}
                where 
\(\hat{E}_{0}\)
is as in Definition \ref{dfnBdRSemiLinear}.

\begin{thm}\label{mainthmFontaine=ThetaCohomologyVer} 
In terms of the isomorphism in Theorem \ref{thmPanPilloniCohomologyIntro},
$\Theta$ on $\hat{E}_{0}(\tilde{\rho}_{k}\hatotimes_{\QQ_{p}}B_{\dR}^{+})$  is induced by the Fontaine operator $N:N_{0}\to N_{k}$, given by       the composition  \[
         N:N_{0}\to M^{\dagger}_{(1-k,0)}\xrightarrow{\theta^{k}}M^{\dagger}_{(1,-k)}\cong N_{k} ,
         \] where $\theta^{k}$ is the classical theta operator as in \cite{Coleman1996classical}.
\end{thm}

\begin{proof}
We have a natural filtered morphism \[R\Gamma(K^{p},\QQ_{p})\hatotimes B_{\dR}^{+}/t^{n}\cong R\Gamma(\mX_{K^{p},\proket},\QQ_{p})\hatotimes B_{\dR}^{+}/t^{n}\to R\Gamma(\Fl,\BdR^{+}/t^{n}),
\] where the first isomorphism follows from \cite[Theorem 4.6.1]{DLLZ2019logarithmicFoundational}, and the second map is induced by the map \(\QQ_{p}\to \mbb{B}_{\dR,\log,\mX_{\Kpp}}|_{\mX_{K^{p}}}\).
 By Theorem \ref{thmPrimitiveComparison}, we know that the map induces isomorphisms on the graded piece, so the map itself is an isomorphism. Taking \((-)^{R-\la}\) on both sides and taking \(\varprojlim_{n}\), 
 we have 
a filtered isomorphism                \[
(R\Gamma(K^{p},\QQ_{p})^{R-\la}\hatotimes_{\QQ_{p}}B_{\dR}^{+})^{\wedge}_{t}\cong R\Gamma(\Fl,\BdR^{+,\la})                 .
                \] 
%                 and thus 
%                  \[
% D_{\mrm{arith}}(R\Gamma(K^{p},\QQ_{p})^{R-\la}\hatotimes_{\QQ_{p}}B_{\dR}^{+}/t^{k+1})\cong R\Gamma(\Fl,D_{\mrm{arith}}(\BdR^{+,\la}/t^{k+1}))                 .
%                 \] 
                We then take $\fb$-cohomology, invert \(t\) and $\hat{E}_{0}(-)$ to obtain a filtered isomorphism
        \[
        \hat{E}_{0}(\tilde{\rho}_{k}\hatotimes B_{\dR})\cong R\Gamma(\Fl,\hat{\mE}_{(k-1,0)}(\BdR^{\la})),
        \] that is compatible with the arithmetic Sen operator \(\Theta\).
%         Recall that by Corollary \ref{corGeomFontaine}, we have an exact sequence in $\Perv$        \[
% 0\to \mE_{(a,b)}(\mO^{\la}(k))\to \mE_{(k-1,0)}(\BdR^{+,\la}/t^{k+1})\to \mE_{(a,b)}(\mO^{\la})
% \to 0 ,
%                 \] 
% where $\mE_{(a,b)}(\mO^{\la}(k))\cong i_{*}\omega^{(1,-k),\sm}[-1]$, and                 \[0\to \omega^{(1-k,0),\sm}\to 
%                 \mE_{(a,b)}(\mO^{\la})\to i_{*}\omega^{(1-k,0),\sm}[-1]\to 0 .
%                 \]
% Taking      \( R\Gamma(\Fl,-) \) gives the distinguished triangle (\ref{seqE0Rho}).
% Moreover, the action of $\Theta$ on $E_{0}(\tilde{\rho}_{k}\hatotimes_{\QQ_{p}}B_{\dR}^{+}/t^{k+1})$ is induced by its action on $\mE_{(k-1,0)}(\BdR^{+,\la}/t^{k+1})$. 
Then Theorem \ref{mainthmFontaine=ThetaCohomologyVer} follows from Theorem \ref{thmFontaine=Theta} by taking $R\Gamma(\Fl,-)$.
\end{proof}

We will need to relate \(\ti{\rho}_{k}[\p_{f}]\)
with \(\rho_{f}\). We study more closely the Eichler-Shimura relation.
\begin{construction}\label{constructionEichlerShimura}
% Let us fix \(\bar{\rho}:\Gal_{\QQ}\to \GL_{2}(k)\) where \(k\) is a finite extension of \(\FF_{p}\). 
We define \[\hat{\TT}^{S}:=\varprojlim_{K_{p},n}\im(\TT^{S}\to \End(R\Gamma_{\et}(X_{\Kpp,\bar{\QQ}},\ZZ/p^{n})),
\] then by \cite[Corollary 9.11]{DospinescuPaskunasSchaen2020infinitesimalFamily}, it has only finitely many open maximal ideals \(\mm\). 

Replacing \(\hat{\TT}^{S}\) by \(\hat{\TT}^{S}\otimes_{\ZZ_{p}}W(\FF)\) for large enough finite extension \(\FF\) over \(\FF_{p}\), we assume that \(\hat{\TT}^{S}/\mm\cong \FF\) for any open maximal ideal \(\mm\subset \hat{\TT}^{S}\). Moreover, by \cite[Corollary 5.1.11]{Scholze15}, there is a continuous \(2\)-dimensional determinant \(D\) of \(\Gal_{\QQ,S}\) valued in \(\hat{\TT}^{S}\). 

For any open maximal ideal \(\mm\) of \(\hat{\TT}^{S}\), 
there is the determinant \(D_{\mm}:=D|_{\hat{\TT}^{S}/\mm}\) valued in \(\FF\). 
Let \(R_{\mm}^{ps}\) denote the pseudo-deformation ring over \(W(\FF)\) of the determinant of \(D_{ \mm}\) as in \cite{Chenevier2014determinants}, which is a Noetherian complete local ring.
Let \(R^{ps}:=\prod_{\mm}R_{\mm}^{ps}\), where 
the product is taken over the finite set of open maximal ideals of \(\hat{\TT}^{S}\). 
Then by the universal property, 
there is a unique continuous morphism \(R^{ps}\to \hat{\TT}^{S}\). 
In particular, we have an action of \(R^{ps}\) on \(R\Gamma(K^{p},W(\FF))\), and thus also on \(M_{\chi}^{\dagger}\) by Theorem \ref{PanPilloni}. 

Let us define the analytic ring structure \(R^{ps}_{\mm,\square}\) on \(R^{ps}_{\mm}\) by putting for any profinite set \(S\), \[R^{ps}_{\square}[S]:=\varprojlim_{i}R^{ps}[S_{i}], 
\] where the limit is taken over finite quotients \(S_{i}\) of \(S\). 
% We consider 

Then by construction, 
\(R^{ps}_{\square}[\Gal_{\QQ,S}]\) acts on \(R\Gamma(K^{p},W(\FF)/p^{n})\). 
Moreover, for any \(g\in \Gal_{\QQ,S}\), let \(D^{univ}\) be the universal determinant over \(R^{ps}\), and denote \(f_{g}(X):=D^{univ}(X-g)\), which is a polynomial of degree \(2\). 
Then by the Eichler-Shimura relation, the action of \(f_{\Frob_{l}}(\Frob_{l})\) is zero on \(H^{i}_{\et}(X_{\Kpp,\bar{\QQ}},W(\FF)/p^{n})\) for \(l\notin S\). 
% Note that 

Let \(I\) denote the \emph{closed} two-sided idea of \(R^{ps}\) generated by \(f_{\Frob_{l}}(\Frob_{l})\) for \(l\notin S\). Then by Chebotarev's density, the action of \(R^{ps}_{\square}[\Gal_{\QQ,S}]\) on \(\ti{H}^{i}(K^{p},W(\FF))\) factors through \(R^{ps}_{\square}[\Gal_{\QQ,S}]/I\).
\end{construction}
% \begin{prop}\label{propEichelerShimura}

% \end{prop}
% \begin{proof}

% Then by the universal property \(R_{\bar{\rho}}\), , which then acts on \(R\Gamma(K^{p},\QQ_{p})\) such that the action of \textbf{To be continued}.
% \end{proof}

\begin{prop}\label{propGalRepIsClassical}
Assume that $f\in M_{(1,-k)}^{\dagger}(K^{p})$ is an overconvergent modular eigenform  of weight $1+k$ with $k\in\ZZ_{\ge 1}$ such that its associated Galois representation $\rho_{f}$ is absolutely irreducible. Then the following are equivalent:

(1) There exists a modular eigenform $f'\in M_{(1,-k)}(K^{p})$ such that $\rho_{f}\cong\rho_{f'}$;

(2) $\rho_{f}$ is de Rham at $p$. 
\end{prop}
\begin{rmk}
The implication from (1) to (2)  is known by \cite{Saito1997modular}.
\end{rmk}
\begin{proof}
Assume that the coefficients of $f$ lie in $L$, which is a finite extension of $W(\FF)[1/p]$, such that \(\rho_{f}\) is defined over \(L\). By Construction \ref{constructionEichlerShimura}, \(\rho_{f}\) determines a map \(\chi_{f}:R^{ps}\to L\), and since \(\TT^{S}\) is dense in \(\hat{\TT}^{S}\), the action of \(R^{ps}\) on \(f\) factors through \(\chi_{f}\). 

We will write $(-)_{L}$ for $-\otimes_{\QQ_{p}}L$. 
% Then the associated Galois representation $\rho_{f}$ has coefficients over $L$ as $\rho_{f}:\Gal_{\QQ,S}\to\GL_{2}(L)$.
 % By enlarging $L$ if necessary, we can assume $\End_{L[\Gal_{\QQ}]}(\rho_{f})\cong L$. 
We denote by $\p_{f}$ the kernel of $\chi_{f}:R^{ps}\otimes L\to L$.
For any $R^{ps}\otimes L$-module $V$ concentrated in degree \(0\), let $V[\p_{f}]$
denote the subspace where the action of $\p_{f}$ is zero. 
% By the Eichler-Shimura relation, $\rho_{f}$ is determined by $\p_{f}$, and \(\)

% To simplify the notation, if \(M\) is a \(\QQ_{p}\)-module, we write \(M_{L}:=M\otimes_{\QQ_{p}} L\).

For $k\in\ZZ_{\ge 2}$, $H^{0}(\Fl,\omega^{(1-k,0),\sm})=0$, and when  $k=1$, $H^{0}(\Fl,\omega^{(1-k,0),\sm})$ consists of locally constant functions, which in particular do not give rise to irreducible Galois representations $\rho_{f}$. Hence we know that for $k\in\ZZ_{\ge1}$, $H^{0}(\Fl,\omega^{(1-k,0),\sm})_{L,\p_{f}}=0$, and $R\Gamma(\Fl,\omega^{(1-k,0),\sm})_{L,\p_{f}}$ is concentrated in degree $1$.
By Theorem \ref{PanPilloni}, $(N_{0})_{\p_{f}}$ and $(\tilde{\rho}_{k})_{\p_{f}}$ are concentrated in degree $0$. Hence $\tilde{\rho}_{k,L}[\p_{f}]$ is a solid $\Gal_{\QQ}\times B(\QQ_{p})\times \TT(K^{p})$-module that is concentrated in degree \(0\).

\begin{lem}\label{lemIsoTopological}
There exists a solid $B(\QQ_{p})\times \TT(K^{p})$-modules $W$ over $L$, such that we have a $B(\QQ_{p})\times \TT(K^{p})$-equivariant isomorphism in \(\Rep_{\QQ_{p,\square}}(\Gal_{\QQ})\)                \[
                \tilde{\rho}_{k,L}[\p_{f}]\cong \rho_{f}\otimes_{L}W .
                \]
\end{lem}
\begin{proof}[Proof of the lemma]
% By the Eichler-Shimura relation, we know that the action of $\Gal_{\QQ}$ on $\ti{H}^{i}(K^{p},\QQ_{p})^{\la}$ satisfies \( \Frob_{l}^{2}-T_{l}\cdot \Frob_{l}+lS_{l}=0 \)
% for $l\notin S$. 
We know \(\ti{H}^{0}(K^{p},\QQ_{p})^{\la}_{L,\p_{f}}\cong 0\) as \(\rho_{f}\) is irreducible, and thus \(\ti{\rho}_{k,L,\p_{f}}\cong \Hom_{\fb}((k-1,0),\ti{H}^{1}(K^{p},\QQ_{p})^{\la}_{\p_{f}})\). In particular, by Construction \ref{constructionEichlerShimura},
the action of \(\Gal_{\QQ}\) on \( \ti{\rho}_{k,L}[\p_{f}] \) factors through \(H^{0}(R^{ps}_{\square}[\Gal_{\QQ,S}]/I\hatotimes_{R^{ps},\chi_{f}}L)\). 
% where \(R^{ps}\to L\) is determined by \(\rho_{f}\). 
Note that \(H^{0}(R^{ps}_{\square}[\Gal_{\QQ,S}]\hatotimes_{R^{ps},\chi_{f}}L)\cong L_{\square}[\Gal_{\QQ,S}]\), and the image of \(I\) in \(L_{\square}[\Gal_{\QQ,S}]\) is the closed 2-sided ideal \(\bar{I}\) generated by \(g^{2}-\Tr(\rho_{f}(g))\cdot g+\det(\rho_{f}(g))\).
By the proof in \cite{BLR91quotients}, we know that \(L_{\square}[\Gal_{\QQ,S}]/\bar{I}\cong \mrm{End}_{L}(\rho_{f})\cong M_{2\times 2}(L)\). 

% also satisfies the Eichler-Shimura relation.
% Thus by Chebotarev's density, we know that for any $g\in \Gal_{\QQ}$, the action of $g$ on $\tilde{\rho}_{k,L}[\p_{f}]$, say $\tau(g)$, satisfies the relation \( \tau(g)^{2}-\Tr(\rho_{f}(g))\tau(g)+\det(\rho_{f}(g))=0 \). 

% Let \(L[\Gal_{\QQ}^{disc}]\) denote the discrete \(L\)-vector space generated by \(\Gal_{\QQ}\) endowed with the discrete topology.
% Let \(R\)
% Let $J$ be the left ideal of $L[\Gal_{\QQ}^{disc}]$ generated by         \( \tau(g)^{2}-\Tr(\rho_{f}(g))\tau(g)+\det(\rho_{f}(g)) \) for all $g\in\Gal_{\QQ}$.   we know that $J$ is a 2-sided ideal, and        \( L[\Gal_{\QQ}^{disc}]/J\to \End_{L}(\rho_{f})\cong M_{2}(L) \)
% is an isomorphism. 
% On the other hand, 
% this map is the restriction of \(\rho_{f}:L_{\square}[\Gal_{\QQ}]\to \End_{L}(\rho_{f})\) to \(L[\Gal_{\QQ}^{disc}]\), and \(L[\Gal_{\QQ}^{disc}]\to L_{\square}[\Gal_{\QQ}]\) has dense image. This implies that 

In this way, $\tilde{\rho}_{k,L}[\p_{f}]$
is a solid  \( \End_{L}(\rho_{f})\times B(\QQ_{p})\times \TT(K^{p}) \)-module. Hence if we put  \begin{align}\label{alignEq0}
                W:=\Hom_{\End_{L}(\rho_{f})}(\rho_{f},\tilde{\rho}_{k,L}[\p_{f}]) ,
\end{align} by Morita equivalence,    we have              \(
                                \tilde{\rho}_{k,L}[\p_{f}]\cong \rho_{f}\otimes_{L}W .\)
%                                 Moreover, by fixing a basis of $\rho_{f}$, we can endow $W$ with the induced topology from $\tilde{\rho}_{k,L}[\p_{f}]$.
% As $\rho_{f}$ is finite dimensional, the endowed topology is independent from the choice of the basis, and in this way, the isomorphism in (\ref{alignEq0}) is 
% a homeomorphism: we have an explicit homeomorphic isomorphism 
% $W^{\oplus 2}\cong \tilde{\rho}_{k,L}[\p_{f}]$ using the diagonal matrices in $\End_{L}(\rho_{f})$. Moreover, $W$ is a topological $B(\QQ_{p})\times \TT(K^{p})$-modules.
\end{proof}
Now by Theorem \ref{PanPilloni}, we have \begin{align}\label{alignEq3}
(\rho_{f}\otimes_{\QQ_{p}}\CC_{p})\otimes_{L\otimes\CC_{p}}(W\hatotimes_{\QQ_{p}}\CC_{p})\cong 
                \tilde{\rho}_{k,L}[\p_{f}]\hatotimes\CC_{p}\cong N_{0,L}[\p_{f}]\oplus N_{k,L}[\p_{f}](-k) .
\end{align} Considering the arithmetic Sen operator $\Theta$, we know that \begin{align}\label{alignEq4}
(\rho_{f}\otimes_{\QQ_{p}}\CC_{p})^{\Theta=0}\otimes (W\hatotimes\CC_{p})\cong N_{0,L}[\p_{f}] ,\\
(\rho_{f}\otimes_{\QQ_{p}}\CC_{p})^{\Theta=-k}(k)\otimes (W\hatotimes\CC_{p})\cong N_{k,L}[\p_{f}].
\end{align}
Theorem \ref{PanPilloni} tells us that       $N_{k,L}[\p_{f}]=M_{(1,-k),L}^{\dagger}[\p_{f}]\ne 0$. So $W$ is non-zero. 
By (\ref{alignEq3}), we know 
the Hodge-Tate weights of $\rho_{f}$ lie in $\{0,k\}$. On the other hand, by Lemma \ref{lemDetHTwtk} below, we know $\det\rho_{f}$ is of Hodge-Tate weight $k$, so we know that $\rho_{f}$ has precisely two distinct weights, $0$ and $k$. 

So we can consider its Fontaine operator, which is compatible with the Fontaine operator of $\hat{E}_{0}(\tilde{\rho}_{k}\hatotimes_{\QQ_{p}}B^{+}_{\dR})$ as in Definition \ref{dfnFontaineOp}. By Theorem \ref{mainthmFontaine=ThetaCohomologyVer}, we know it is induced by                 \[
                N[\p_{f}]:N_{0,L}[\p_{f}]\xrightarrow{p[\p_{f}]} M^{\dagger}_{(1-k,0),L}[\p_{f}]\xrightarrow{\theta^{k}}M^{\dagger}_{(1,-k),L}[\p_{f}]\cong  N_{k,L}[\p_{f}].
                \] Note that the kernel of $p[\p_{f}]$ is $H^{1}(\Fl,\omega^{(1-k,0),\sm})_{L}[\p_{f}]$. 
By Lemma \ref{lemThetakInjective} below, $M^{\dagger}_{(1-k,0),L}[\p_{f}]\xrightarrow{\theta^{k}}M^{\dagger}_{(1,-k),L}[\p_{f}]$
is injective, so \[\Ker(N[\p_{f}])\cong \Ker(p[\p_{f}])\cong H^{1}(\Fl,\omega^{(1-k,0),\sm})_{L}[\p_{f}].\]

By Serra duality, we know               
                \begin{align*}
 H^{1}(\Fl,&\omega^{(1-k,0),\sm})\cong\varinjlim_{K_{p}}H^{1}(\mX_{\Kpp},\omega^{(1-k,0)})\\&\cong \varinjlim_{K_{p}}H^{0}(\mX_{\Kpp},\omega^{(k,-1)}(-C))^{\vee}\cong \varinjlim_{K_{p}}S_{(k,-1)}(\Kpp)^{\vee}.
                \end{align*}
Therefore, by Proposition \ref{propFontaineOpeClassical}, $\rho_{f}$ is de Rham, if and only if $N[\p_{f}]=0$, if and only if $\Ker(N[\p_{f}])\ne 0$, if and only if $H^{1}(\Fl,\omega^{(1-k,0),\sm})_{L}[\p_{f}]\ne 0$, if and only if $\p_{f}$ is associated to a classical eigenform. This finishes the proof.
\end{proof}
\begin{lem}\label{lemDetHTwtk}
For any overconvergent eigenform $f\in M_{(1,-k)}^{\dagger}$ with $k\in\ZZ_{\ge 1}$, $\det\rho_{f}$ is Hodge-Tate of weight $k$.
\end{lem}
\begin{proof}
First, we prove that $\det\rho_{f}$ is of Hodge-Tate weight $k$. We know for $l\nmid p$ such that $K^{p}=\prod_{l\nmid p}K_{l}$ is hyperspecial at $l$, $\det\rho_{f}(\Frob_{l})=lS_{l}$ by the Eichler-Shimura relation, where  \( S_{l}=\left[G(\ZZ_{l})\begin{pmatrix}
l^{-1} & 0 \\ 0 & l^{-1}
\end{pmatrix}G(\ZZ_{l}) \right]. \)
Here we put an inverse so that our Hecke operator acts on the left. We claim that there exist $N\in\ZZ_{>0}$ and $\chi:(\ZZ/N)^{\times}\to \bar{\QQ}_{p}^{\times}$, such that for any $l\nmid N$ such that $K_{l}$ is hyperspecial, $S_{l}=l^{k-1}\chi(l)$.  
If $f$ is $N(\ZZ_{p})$-invariant, then for large enough $n$, $f$ is $\Gamma_{1}(p^{n})$-invariant, where $\Gamma_{1}(p^{n}):=\left\{\begin{pmatrix}
a& b\\c &d
\end{pmatrix}:a-1,b\in p^{n}\ZZ_{p}\right\}\subset \GL_{2}(p^{n})$ is the congruence subgroup. Then we know that $S_{l}=l^{k-1}\langle l\rangle$, where $\langle -\rangle$ is the diamond operator, which is a character of finite order. In general, since the action of $\GL_{2}(\QQ_{p})$ on $f$ 
is smooth, we know $f$ is invariant for some $N(p^{m}\ZZ_{p})=\begin{pmatrix}
1& p^{n}\ZZ_{p}\\0 &1
\end{pmatrix}$ for $m$ large enough. Thus we know         \( \begin{pmatrix}
p^{-n}& 0\\0 & 1
\end{pmatrix} f\) is $N(\ZZ_{p})$-invariant, and it is also an eigenform with the same Hecke eigenvalue  
as $f$. So we conclude by applying the above argument to \( \begin{pmatrix}
p^{-n}& 0\\0 & 1
\end{pmatrix} f\). 

Given the claim, we know         \( \det\rho_{f}(\Frob_{l})=l^{k}\chi(l), \)
and $\chi$ is of finite order. In particular, when restricting to an open subgroup, we know that $\det\rho_{f}$ coincides with $\chi_{\mrm{cycl}}^{k}$, and thus is of Hodge-Tate weight $k$.
\end{proof}
\begin{lem}\label{lemThetakInjective}
For $k\in\ZZ_{\ge 1}$ and $f\in M_{(1,-k)}^{\dagger}[\p_{f}]$, if $\rho_{f}:\Gal_{\QQ}\to \GL_{2}(\bar{\QQ}_{p})$ is irreducible, then $\theta^{k}:M_{(1-k,0)}^{\dagger}[\p_{f}]\to M_{(1,-k)}^{\dagger}[\p_{f}]$
is injective.
\end{lem}
\begin{proof}
Note that  $M^{\dagger}_{(1-k,0)}\xrightarrow{\theta^{k}}M^{\dagger}_{(1,-k)}$ is injective when $k>1$, and the kernel is precisely the locally constant function when $k=1$. To see this, we recall
that in terms of $q$-expansion, $\theta^{k}$ is given by $(q\frac{d}{dq})^{k}$. For any $f'\in\Ker(\theta^{k})$, there exists $n\in\NN$ such that 
$f'$ is fixed by $\begin{pmatrix}
1 & p^{n}\ZZ_{p}\\0 & 1
\end{pmatrix}\subset B(\QQ_{p})$, and we can write the $q$-expansion of $f'$, say $f'=\sum_{i\in \frac{1}{p^{n}}\NN}a_{n}q^{n}$. Then $\theta^{k}(f')=0$ implies that $a_{n}=0$ if $n\ne 0$, and thus $f'$ is locally constant. For such $f'$, $\rho_{f'}$ will be reducible. Thus we have   \( \Ker(\theta^{k})[\p_{f}]=0 \) and we are done.
\end{proof}
So far we have proven that the global Galois representation $\rho_{f}$ is modular if it is de Rham at $p$. Now we go on to prove the slightly stronger result that the form $f$ itself is a modular form. 
\begin{cor}\label{corClassicality}
Assume that $f\in M_{(1,-k)}^{\dagger}(K^{p})$ is an overconvergent modular $\TT^{S}$-eigenform  of weight $1+k$ with $k\in\ZZ_{\ge 1}$ such that its associated Galois representation $\rho_{f}$ is absolutely irreducible. Then
$f$ is a classical modular form if and only if $\rho_{f}$ is de Rham at $p$. 
\end{cor}
\begin{rmk}
Note that we only assume that $f$ is an eigenform for the spherical Hecke algebra $\TT^{S}$ rather than \(\TT(K^{p})\). In particular, the result does not follow immediately from considering the $q$-expansion.
\end{rmk}
\begin{proof}
We consider the first map in (\ref{alignSeqExactOnSVbtog})                \[
                \Hom_{\fg}(\Sym^{k-1}V,\tilde{H}^{1,\la}(K^{p},L))[\p_{f}]\hookrightarrow \Hom_{\fb}((k-1,0),\tilde{H}^{1,\la}(K^{p},L))[\p_{f}]=\tilde{\rho}_{k,L}[\p_{f}] .
                \]
We claim that this is actually an equality.
                 If we denote \[\tilde{\rho}'_{k,L}:=\Hom_{\fg}(\Sym^{k-1}V,\tilde{H}^{1,\la}(K^{p},L)),\] then         \( \tilde{\rho}'_{k,L}[\p_{f}] \)
is a $\Gal_{\QQ}\times B(\QQ_{p})\times \TT(K^{p})$-subrepresentation of         \( \tilde{\rho}_{k,L}[\p_{f}]. \)

Recall that by Lemma \ref{lemIsoTopological}, we have  an isomorphism               \[
                \tilde{\rho}_{k,L}[\p_{f}]\cong \rho_{f}\otimes_{L}W .
                \]
If fixing an $\Gal_{\QQ_{p}}$-equivariant isomorphism $(\rho_{f}\otimes\CC_{p})^{\Theta=0}\cong L\otimes\CC_{p}$,
(\ref{alignEq4}) gives us an isomorphism         \( W\hatotimes\CC_{p}\cong \varinjlim_{\Kpp}S_{(k,-1)}(\Kpp)^{\vee}_{L}[\p_{f}]. \) Moreover, both sides have a canonical model over $L$, and the isomorphism is equivariant for the semi-linear action of $\Gal_{\QQ_{p}}$, so it descends to         \( W\cong \varinjlim_{\Kpp}(S_{(k,-1)}(\Kpp)^{\vee}_{L})^{\Gal_{\QQ_{p}}}[\p_{f}]. \)
In particular, we see that $W$ is an injective limit of finite dimensional spaces, and in particular, $W\hatotimes\CC_{p}\cong W\otimes \CC_{p}$ and thus         \( \tilde{\rho}_{k,L}[\p_{f}]\otimes\CC_{p}\cong \tilde{\rho}_{k,L}[\p_{f}]\hatotimes\CC_{p}. \)

From (\ref{alignEq4}), we know    that  \begin{align}\label{alignEq1}
                (\tilde{\rho}_{k,L}[\p_{f}]{\otimes}\CC_{p})^{\Theta=0}\cong \varinjlim_{K_{p}} S_{(k,-1)}(\Kpp)^{\vee}_{L}[\p_{f}],\;\\ (\tilde{\rho}_{k,L}[\p_{f}]{\otimes}\CC_{p})^{\Theta=-k}(k)\cong M_{(1,-k),L}^{\dagger}[\p_{f}] .
\end{align}

On the other hand, by taking $R\Gamma(\Fl,-)$ of Proposition \ref{propCalculateg-cohomoOnSV}, we know             \[
               \tilde{\rho}'_{k,L}\otimes \CC_{p} \cong \varinjlim_{K_{p}}S_{(k,-1)}(\Kpp)^{\vee }_{L} \oplus M_{(1,-k),L}(-k)  ,
                \] where by Serre duality,       we have identified  \[ H^{1}(\mX_{\Kpp,\CC_{p}},\omega^{(1-k,0)}_{X_{\Kpp}})\cong H^{0}(\mX_{\Kpp},\omega^{(k,-1)}_{X_{\Kpp}}(-C))^{\vee}\cong  S_{(k,-1)}(\Kpp)^{\vee}. \]
Therefore, we see that 
\begin{align}\label{alignEq2}
                (\tilde{\rho}_{k,L}'[\p_{f}]{\otimes}\CC_{p})^{\Theta=0}\cong \varinjlim_{K_{p}} S_{(k,-1)}(\Kpp)^{\vee}_{L}[\p_{f}],\; \\(\tilde{\rho}_{k,L}'[\p_{f}]{\otimes}\CC_{p})^{\Theta=-k}(k)\cong M_{(1,-k),L}[\p_{f}] .
\end{align}
Moreover, by Proposition \ref{propCalculateg-cohomoOnSV}, the map \[(\tilde{\rho}_{k,L}'[\p_{f}]{\otimes}\CC_{p})^{\Theta=-k}(k)\to (\tilde{\rho}_{k,L}[\p_{f}]{\otimes}\CC_{p})^{\Theta=-k}(k)\]
coincides with the natural inclusion 
\(M_{(1,-k),L}\hookrightarrow M_{(1,-k),L}^{\dagger}\).

Now consider $V:=\tilde{\rho}_{k,L}[\p_{f}]/\tilde{\rho}_{k,L}'[\p_{f}]$. Then again by \cite{BLR91quotients}, we know $V\cong \rho_{f}^{\oplus I}$, and by comparing (\ref{alignEq1}) and (\ref{alignEq2}),             \[
                (V\otimes \CC_{p})^{\Theta=0}\cong 0,\;(V\otimes\CC_{p})^{\Theta=-k}(k)\cong M^{\dagger}_{(1,-k),L}[\p_{f}]/M_{(1,-k),L}[\p_{f}] .
                \] 
                In particular, we know that $I=\emptyset$
                and $M_{(1,-k),L}[\p_{f}]=M^{\dagger}_{(1,-k),L}[\p_{f}]$.
\end{proof} 

\appendix
\addtocontents{toc}{\protect\setcounter{tocdepth}{1}}
% \section{Geometric Sen Theory and the Functor $\VB$}\label{appendixGeoSenVB}

% \phantomsection 
% \addcontentsline{toc}{section}{References} 
\printbibliography

\end{document}